\RequirePackage{fix-cm}
\documentclass{svjour3}                     
\smartqed  
\usepackage[margin=1in]{geometry}
\usepackage{graphicx}
\usepackage{amsmath} 
\usepackage{mathtools}
\usepackage{authblk} 
\usepackage{amssymb}
\usepackage{amsfonts}
\usepackage{verbatim}
\usepackage{mathrsfs}
\usepackage{subfigure}
\usepackage{mathtools}
\usepackage{nicefrac}
\usepackage{xcolor}
\usepackage{url}
\usepackage{todonotes}
\usepackage{multirow}
\usepackage{chngcntr}
\usepackage{apptools}
\AtAppendix{\counterwithin{lemma}{section}}

\newcommand{\bydef}{\,\stackrel{\mbox{\tiny\textnormal{\raisebox{0ex}[0ex][0ex]{def}}}}{=}\,}

\newcommand{\bv}{{\boldsymbol{\rm v}}}

\newcommand{\bzero}{\mathbf{0}}

\newcommand{\boldx}{\mathbf{x}}
\newcommand{\boldy}{\mathbf{y}}
\newcommand{\constant}{M}
\newcommand{\conjugacy}{K}
\newcommand{\centerconj}{K_c}
\newcommand{\hyperbolicconj}{K_h}
\newcommand{\centerstableconj}{K_{cs}}

\newcommand{\stablebranch}{h_s}
\newcommand{\unstablebranch}{h_u}
\newcommand{\branch}{h}
\newcommand{\dynamics}{f}
\newcommand{\centerdynamics}{r}
\newcommand{\stabledynamics}{q_{s}}

\newcommand{\centersubspace}{X_c}
\newcommand{\centerstablesubspace}{X_{cs}}

\newcommand{\energy}{\mathcal{E}}

\newcommand{\kap}{ - (1 -2 m_1)}

\begin{document}

\title{Critical homoclinics in a restricted four body problem
\thanks{The first author was partially supported by NWO-VICI grant 639033109.  \\
The second author was partially supported by NSF grant 
DMS-1813501.
}
}
%

\subtitle{numerical continuation and center manifold computations}


\author{Wouter Hetebrij \and
            J.D. Mireles James 
}


\institute{W. Hetebrij 
          \at Vrije Universiteit Amsterdam, 
               Department of Mathematics \\
               \email{w.a.hetebrij@vu.nl}
                            \and
              J.D. Mireles James \at
              Florida Atlantic University, Department of Mathematical Sciences \\
              \email{jmirelesjames@fau.edu}
}

\date{Received: date / Accepted: date}

\maketitle

\begin{abstract}
The present work studies the robustness of certain basic homoclinic motions 
in an equilateral restricted four body problem.  The problem can be 
viewed as a two parameter family of conservative autonomous vector fields. 
The main tools are numerical continuation techniques for homoclinic and periodic 
orbits, as well as formal series methods for computing normal forms and center
stable/unstable manifold parameterizations.  After careful numerical study of a number of
special cases we formulate several conjectures about the global bifurcations 
of the homoclinic families.   
\keywords{$4$-body problem
\and homoclinic dynamics
\and critical equilibria
\and center manifold parameterization}
\PACS{45.50.Jf	 \and 45.50.Pk \and 45.10.-b \and 02.60.Lj \and 05.45.Ac}
\subclass{70K44 \and 34C45 \and 70F15}
\end{abstract}



\section{Introduction} \label{sec:intro}
Suppose that three gravitating bodies are arranged in the equilateral 
triangle configuration of Lagrange.  
The \textit{circular restricted four body problem} (CRFBP)
studies the dynamics of a fourth massless particle in a co-rotating reference frame.
The problem was first introduced by Pedersen
 \cite{pedersen1,pedersen2}, and we recall the equations of motion
and other basic facts in Section \ref{sec:equationsOfMotion}.
A recent work by Kepley and Mireles James  \cite{MR3919451}
studies -- in the case of equal masses -- certain ``short'' or ``basic'' homoclinic motions 
at the center of mass of the three bodies.   In the present work
we are interested in the fate of these basic homoclinic motions 
as the system is  perturbed away from the equal mass case:
in particular their robustness, bifurcations, and eventual disappearance.
In preparation for this discussion we
briefly review what is know about the structure of the equilibrium set.

The equilibrium solutions -- or \textit{libration points} -- of the CRFBP 
are the main topic of the study \cite{MR510556}  by Sim\'{o}.
In that work one finds detailed numerical evidence in support of 
the claim that the problem has either 8, 9, or 10 libration points, whose number and 
location depend on the mass ratios.
This conjecture was eventually settled in the affirmative by Barros and Leandro
using mathematically rigorous computer assisted methods of proof
 \cite{MR2232439,MR2784870,MR3176322}, and we recount their results
 after introducing a little notation and terminology. 

Appropriate choice of units results in a unit value of the gravitational constant and total mass
of the system.  The massive bodies are 
labeled according to the convention that $ m_3 \leq m_2 \leq m_1$.
The parameter space of the CRFBP is then reduced to the $2$-simplex 
\[
\mathfrak{S} = \left\{ (m_1, m_2, m_3) \in \mathbb{R}^3 \, \colon \, m_1 + m_2 + m_3 = 1, 
\mbox{ and }  m_3 \leq m_2 \leq m_1 \right\},
\]
determined by the vertices $v_0 = (1/3,1/3,1/3)$, 
$v_1 = (1/2,1/2,0)$, and $v_2 = (1,0,0)$.

We refer to the special system with mass parameters $m_1 = m_2 = m_3 = 1/3$
as the \textit{triple Copenhagen problem}.  This is a nod to the 
traditional name of the equal mass case of the circular restricted three 
body problem (CRTBP), which is called 
\textit{the Copenhagen problem} in honor of the work done at the 
Copenhagen observatory in the first decades of the Twentieth Century
during the tenure of Elis Str\"{o}mgren.  
See for example the review article \cite{stromgrenRef} by Str\"{o}mgren, 
as well as the detailed discussion in Chapter 9 of the book of Szebehely \cite{theoryOfOrbits}.

We remark that if $m_3 = 0$ then the CRFBP reduces to the CRTBP, and that if 
$m_2 = m_3 = 0$ then the problem reduces further to the rotating Kepler problem.  
The CRTBP is treated in great detail elsewhere, and we will not dwell on it further other than to say that
the problem is well known to have five libration points for all values of the mass ratio $\mu = m_2/m_1 \neq 0$.
For much more thorough discussion we refer the interested reader again to Chapter 9 of the book of 
Szebehely \cite{theoryOfOrbits}, or to the more modern treatment in the book of Meyer and Hall \cite{MR3642697}.

We refer to an equilibrium solution in the interior of the closed equilateral triangle 
as an \textit{inner libration point}, and an equilibrium 
in the complement of this triangle as an \textit{outer libration point}.
A schematic illustrating the phase space
is given in Figure \ref{rotatingframe}.  
From \cite{MR2232439,MR2784870,MR3176322}
we have the following complete description of the equilibrium set.
\begin{itemize}
\item \textbf{(I)} For each  $(m_1, m_2, m_3) \in \mathfrak{S}$ with $m_3 > 0$ there are six outer libration 
points.  We denote these by $L_{4,5,6,7,8,9}$.
\item \textbf{(II)} There is an analytic, simple closed curve $\mathfrak{D} \subset \mathfrak{S}$
from the $m_1 = m_2$ edge to the $m_2 = m_3$ edge of the simplex.  
The number of libration points is constant throughout $\mathfrak{S}$, except on $\mathfrak{D}$.
We refer to $\mathfrak{D}$ as the \textit{critical parameter curve}.
$\mathfrak{D}$ does not contain any vertex, nor does it intersect the $m_3 = 0$ 
edge of the simplex.  
\item \textbf{(IV)} $\mathfrak{S} \backslash \mathfrak{D}$ has two components which
we denote by $\mathfrak{S}_I$ and $\mathfrak{S}_{II}$.  We take $\mathfrak{S}_I$
to be the component containing the triple Copenhagen vertex $v_0$.  
For each $(m_1, m_2, m_3) \in \mathfrak{S}_I$ the system has 4 
inner libration points, making
ten in total.  For each $(m_1, m_2, m_3) \in \mathfrak{S}_{II}$ with $m_3 > 0$ 
the system has 2 inner libration points, making 8 total.  
\item \textbf{(V)} If $(m_1, m_2, m_3) \in \mathfrak{D}$ and $m_1 \neq m_2$, 
then the system has 3 inner libration points, making for 9 total.  
Denote by $v_{{\tiny \mbox{pf}}} \in \mathfrak{S}$ the point where
$\mathfrak{D}$ intersects the $m_1 = m_2$ edge (the reason for the ``pf'' will 
be made clear below).  When the system has parameters
$v_{{\tiny \mbox{pf}}}$ there are 2 inner libration points, for a total of 8.  
\end{itemize}
The parameter simplex is illustrated schematically in Figure \ref{parameterSimplex}.

\begin{figure}[!t]
\centering
\includegraphics[width=4.5in]{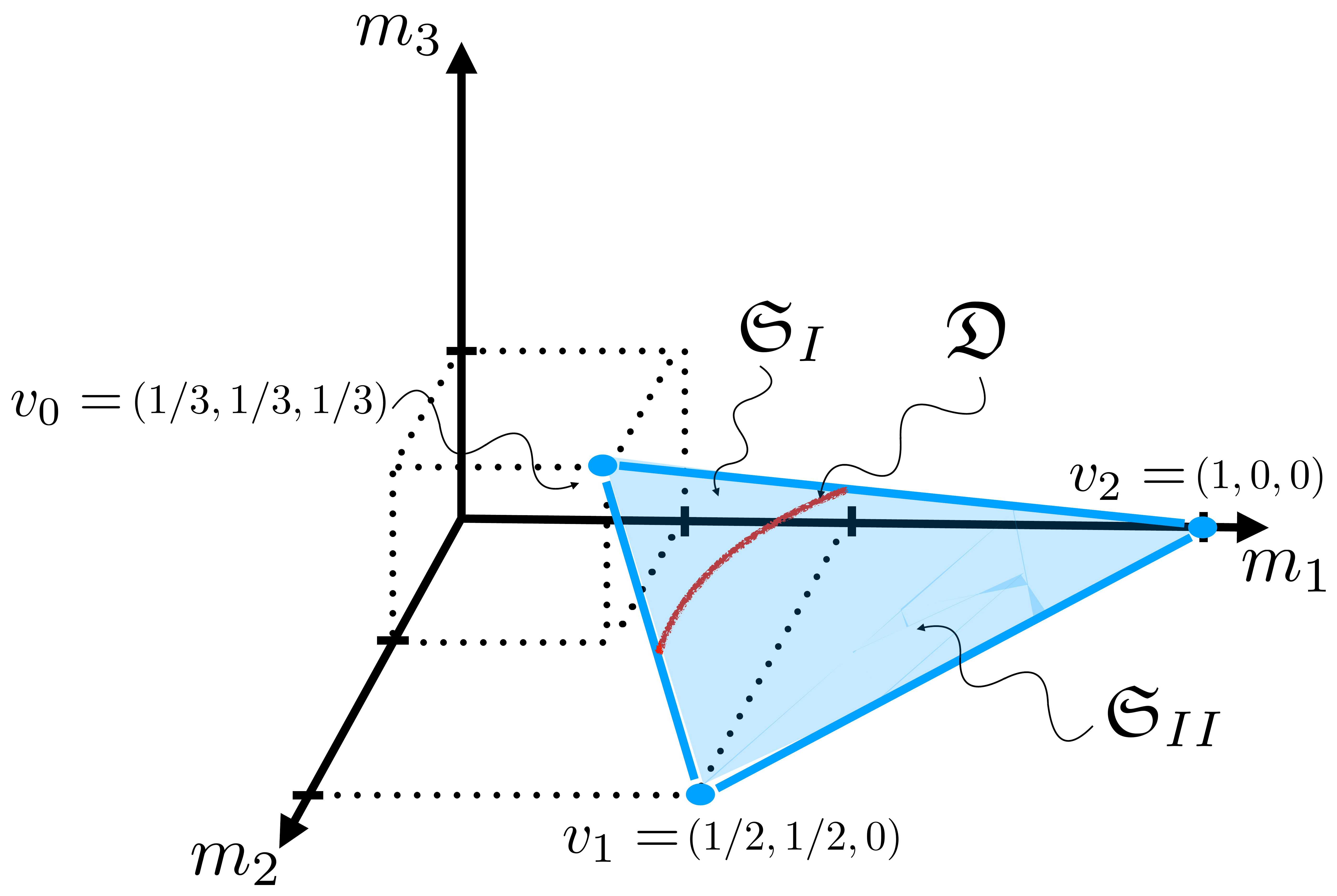}
\caption{\textbf{Parameter simplex for the CRFBP:} normalizing so that 
$m_1 + m_2 + m_3 = 1$ with $m_3 \leq m_2 \leq m_1$ leads to a parameter 
space as depicted in the figure.  The 
simplex is formed by the vertices $v_0 = (1/3, 1/3, 1/3)$ corresponding 
to equal masses (triple Copenhagen problem), $v_1 = (1/2, 1/2, 0)$ corresponding to 
all mass in two equal primaries (restricted three body Copenhagen problem), 
and $v_2 = (1, 0, 0)$ where all the mass is in the largest body (rotating Kepler problem).  
Observe that $m_2 = m_3$ along the edge joining $v_0$ and $v_2$, that 
$m_1 = m_2$ along the edge joining $v_0$ and $v_1$, and that $m_3 = 0$ 
along the edge joining $v_1$ and $v_2$.  The critical parameter 
curve $\mathfrak{D}$ is depicted as a red arc from the $m_1 = m_2$ edge to the 
$m_2 = m_3$ edge, cutting the simplex in two components 
denoted $\mathfrak{S}_I$ and $\mathfrak{S}_{II}$.  In $\mathfrak{S}_I$ the system has 
ten equilibrium solutions and in $\mathfrak{S}_{II}$ only eight.  The number changes only 
on the critical curve $\mathfrak{D}$ where the inner libration point $L_0$ loses hyperbolicity 
and annihilates with $L_2$.  See Figure \ref{rotatingframe} for the approximate locations
of the $8-10$ libration points.
}\label{parameterSimplex}
\end{figure}

Let us elaborate on statements $(IV)$ and $(V)$.  
With  $(m_1, m_2, m_3) \in \mathfrak{S}_I$, denote the 
$4$ inner libration points by $L_{0,1,2,3}$ (as in Figure \ref{rotatingframe}).  
Since there are no bifurcations in $\mathfrak{S}_I$, the locations of 
the libration points vary continuously (even analytically) throughout this region.
In the triple Copenhagen problem the libration point $L_0$ is located at the 
center of mass/origin in state space.  For all parameters in $\mathfrak{S}_I$, 
$L_0$ has saddle type stability (two stable and two unstable eigenvalues)
and the  libration points $L_{1,2,3}$ have saddle $\times$ centers 
stability (one stable, one unstable and a pair of purely imaginary 
conjugate eigenvalues).

Suppose that parameters are
varied continuously from a point in the region $\mathfrak{S}_I$ 
to a point in the region $\mathfrak{S}_{II}$.  
As the parameters cross the critical curve $\mathfrak{D}$ the system undergoes 
a bifurcation involving $L_0$ and $L_2$.
In the general case that $m_1 \neq m_2$, 
the bifurcation is a Hamiltonian saddle node 
wherein $L_0$ and $L_2$ collide and 
annihilate.  At $v_{{\tiny \mbox{pf}}}$ -- where $\mathfrak{D}$
intersects the $m_1 = m_2$ edge -- the bifurcation is a Hamiltonian pitchfork bifurcation involving 
$L_0$, $L_2$, and $L_3$.   Again, $L_0$ and $L_2$ vanish in this pitchfork bifurcation, 
so that in every case it is only the inner libration points $L_1$ and $L_3$ which
remain once $(m_1, m_2, m_3) \in \mathfrak{S}_{II}$. 
We write $L_{{\tiny \mbox{sn}}}$ to denote the libration point
at saddle node bifurcation, and $L_{{\tiny \mbox{pf}}}$ for the pitchfork.
We sometimes write $L_c$ to denote a critical libration point without 
specifying whether we are at the pitchfork or a saddle node bifurcation.

Given that the equilibrium structure of the CRFBP is completely understood,
one natural line of inquiry is to study orbits homoclinic to the equilibria,
and any possible bifurcations of these homoclinic connecting orbits.  
For example, bifurcations at $L_0$ when the system crosses the 
critical curve $\mathfrak{D}$ must trigger corresponding bifurcations for 
orbits homoclinic to $L_0$.  Some studies which consider heteroclinic 
and homoclinic connections in the CRFBP are
Delgado and Burgos \cite{MR3105958}, and  
Baltagiannis and Papadakis\cite{papadakisPO_likeUs}.  
The present work builds on  the 
recent study of homoclinic phenomena in the triple 
Copenhagen problem by Kepley and Mireles James \cite{MR3919451}.

Relevant results from \cite{MR3919451} are reviewed in Section \ref{sec:copenhagenHomoclinics}.
What is important for the purposes of the present introduction is 
that there are six basic homoclinic motions at triple Copenhagen $L_0$, denoted by 
$\gamma_{1,2,3,4,5,6} \colon \mathbb{R} \to \mathbb{R}^4$ (see
Figure \ref{fig:basicHomoclinics}), and that these basic connections
appear to organize all observed homoclinic phenomena at $L_0$.
Given the importance of $\gamma_{1,2,3,4,5,6}$ in the triple Copenhagen problem  
it is natural to investigate their role as parameters
vary.  This leads to some fairly delicate questions about the global dynamics.
The present work focuses primarily on the ``shortest'' homoclinics 
$\gamma_{1,2,3}$.
\begin{itemize}
\item \textbf{Question 1:}   The connections $\gamma_{1,2,3}$
are transverse (in the Hamiltonian sense), and hence
persist for parameter values $(m_1, m_2, m_3) \approx v_0$.  
\textit{What happens to $\gamma_{1,2,3}$ 
as the mass parameters move throughout $\mathfrak{S}_1$ toward $\mathfrak{D}$?}
In particular, how robust are the connections?  Do any of the connections survive all the way to 
$\mathfrak{D}$?
\item \textbf{Question 2:} With $(m_1, m_2, m_3) \in \mathfrak{D}$
consider the critical libration point $L_c$. \textit{Are there homoclinic connections
to the critical libration point?}  How are they related to the basic homoclinic motions
$\gamma_{1,2,3}$  of the triple Copenhagen problem?
\end{itemize}

The CRFBP has a conserved, energy-like quantity
known as the Jacobi integral (see Section \ref{sec:equationsOfMotion}). 
Systems with first integrals enjoy an intimate relationship between homoclinic 
orbits and one parameter families of periodic orbits. 
This connection was first studied by Str\"{o}mgren in connection with the 
(CRTBP) (see for example \cite{stromgrenRef}) where it was observed
that some planar families of Lyapunov periodic orbits appear to accumulate
to ``asymptotic periodic orbits'' -- heteroclinic cycles or 
homoclinic orbits in modern terminology.  This phenomenon involves a global 
bifurcation made precise by the ``blue sky catastrophe'' of Shilnikov, Henrard,  
and Devaney \cite{MR0365628,MR0431274,MR3105958,MR3253906}.

It was observed in \cite{MR3919451} that each
of the short homoclinics $\gamma_{1,2,3}$ 
participates in a blue sky catastrophe,
appearing as the limit of the
planar Lyapunov family associated with the inner libration it winds around.
 More precisely, the planar family of periodic orbits attached to
$L_{i}$ accumulates to $\gamma_{i}$, for $i = 1,2,3$. 
This leads to a third question concerning the phase space structure of the CRFBP
at criticality.  
\begin{itemize}
\item \textbf{Question 3:} For $(m_1, m_2, m_3) \in \mathfrak{D}$, 
what is the asymptotic fate of the planar Lyapunov families 
attached to $L_1$ and $L_3$?   Do these families participate in 
blue sky castrophies with critical homoclinic orbits at $L_{c}$?  If 
not, where do they accumulate?  (Recall that $L_2$ has collided 
with $L_0$ on the critical curve, so that question only makes sense
for $L_{1,3}$).
\end{itemize}

These three questions are the main topic of the present 
study, and are addressed using tools from computational dynamics.  
In particular we apply
numerical continuation methods for periodic/homoclinic orbits, 
as well as high order numerical methods for computing 
invariant manifolds attached to libration points.  
These topics are standard and have been discussed at length in 
other places, and we provide some references when appropriate 
below.  We include, for the sake of completeness,
a short overview of numerical continuation schemes
in conservative system as Appendix \ref{sec:appendixContinuation}. 
We also employ high order methods for computing 
center and center stable/unstable manifolds, as well as normal form 
calculations, to illuminate the dynamics at $\mathfrak{D}$.
We will see that the three questions are interrelated, so that understanding 
any one of them provides information about the other two.  

More precisely we are guided by the following qualitative observations, whose judicious 
use leads to quantitate information about bifurcations of the homoclinics orbits.
For numerical continuations we will always consider parameter curves 
starting at $v_0$ and terminating at the critical curve $\mathfrak{D}$.  
This leads to one parameter 
continuation problems, so that bifurcations then occur at isolated parameter points, 
and we can talk about what happens before and after such a bifurcation point.

\begin{itemize}
\item \textbf{Approximating bifurcation parameters from below -- robustness of numerical continuation:} 
suppose that we vary the mass parameters of the system from $v_0$ toward the critical curve $\mathcal{D}$,
and apply a numerical continuation scheme to one of the homoclinic orbits $\gamma_{1,2,3}$.
Then a breakdown in the numerical continuation scheme indicates a possible bifurcation 
of the homoclinic family.  Breakdown is indicated by the loss of invertibility (or poor numerical conditioning) of 
a certain matrix.   The numerical scheme will breakdown before 
the bifurcation, so that the blow up of the condition number provides a useful 
lower bound on the approximate location of
the bifurcation parameter. This is the topic of Section \ref{sec:numCont}.
\item \textbf{Approximating bifurcation parameters from above -- the blue sky test:}  as mentioned above,
the ``tubes'' of planar Lyapunov periodic orbits 
originating at $L_{1,2,3}$ in the triple Copenhagen problem accumulate to the interior 
homoclinic orbits $\gamma_{1,2,3}$.  
This phenomena is robust,
and hence persists for small changes in parameters.
Then studying the limiting behavior of a Lyapunov tube 
provides another qualitative feature that can only 
change at a bifurcation of the homoclinic orbit.  
Numerically locating a 
Lyapunov tube which no longer accumulates to $\gamma_{1,2,3}$
suggests that we have passed a bifurcation of the homoclinic, and provides
a useful geometric mechanism for obtaining upper bounds on the bifurcation parameter.
This is the topic of Section \ref{sec:blueSkies}.
\item \textbf{Bifurcations on the critical curve -- normal form calculations:}  
in some cases the continuation robustness and/or blue sky tests are inconclusive. 
In particular, the tests have a difficult time distinguishing bifurcations which occur near, 
but not on, the critical curve $\mathcal{D}$.  In this case it is helpful to examine the normal form 
at the critical equilibrium solution $L_c$, as it provides local information about connecting orbits.
It can also be useful to numerically study intersections of the center stable/center unstable 
manifolds in the $L_c$ level set of the Jacobi integral. 
This is the topic of Section \ref{sec:normalform}.
\end{itemize}

\noindent Using these numerical techniques in concert provides a novel approach to the qualitative 
study of global continuation and bifurcation properties of connecting orbits in conservative systems.

The remainder of the paper is organized as follows.
In sections \ref{sec:equationsOfMotion} to \ref{sec:bifurcations}
we review the equations of motion, a numerical method for computing critical equilibria/parameter 
sets, results about homoclinic motions in the triple Copenhagen problem, as well as the literature 
on homoclinic bifurcations.
In Section \ref{sec:numCont} we study the $\gamma_{1,2,3}$ families via numerical 
continuation algorithms, while Section \ref{sec:blueSkies} is devoted to blue sky 
catastrophes.  In Section \ref{sec:normalform} we study the dynamics on 
$\mathfrak{D}$ using numerically computed center stable/unstable manifolds
and normal forms.  Our conclusions are summarized in Section \ref{sec:conclusions},
and Appendices \ref{sec:appendixContinuation} and \ref{sec:centerCalcAppendix} provide details on 
multiple shooting/continuation schemes and computing 
the center as well as center stable/unstable manifolds.

\begin{figure}[!t]
\centering
\includegraphics[width=4.5in]{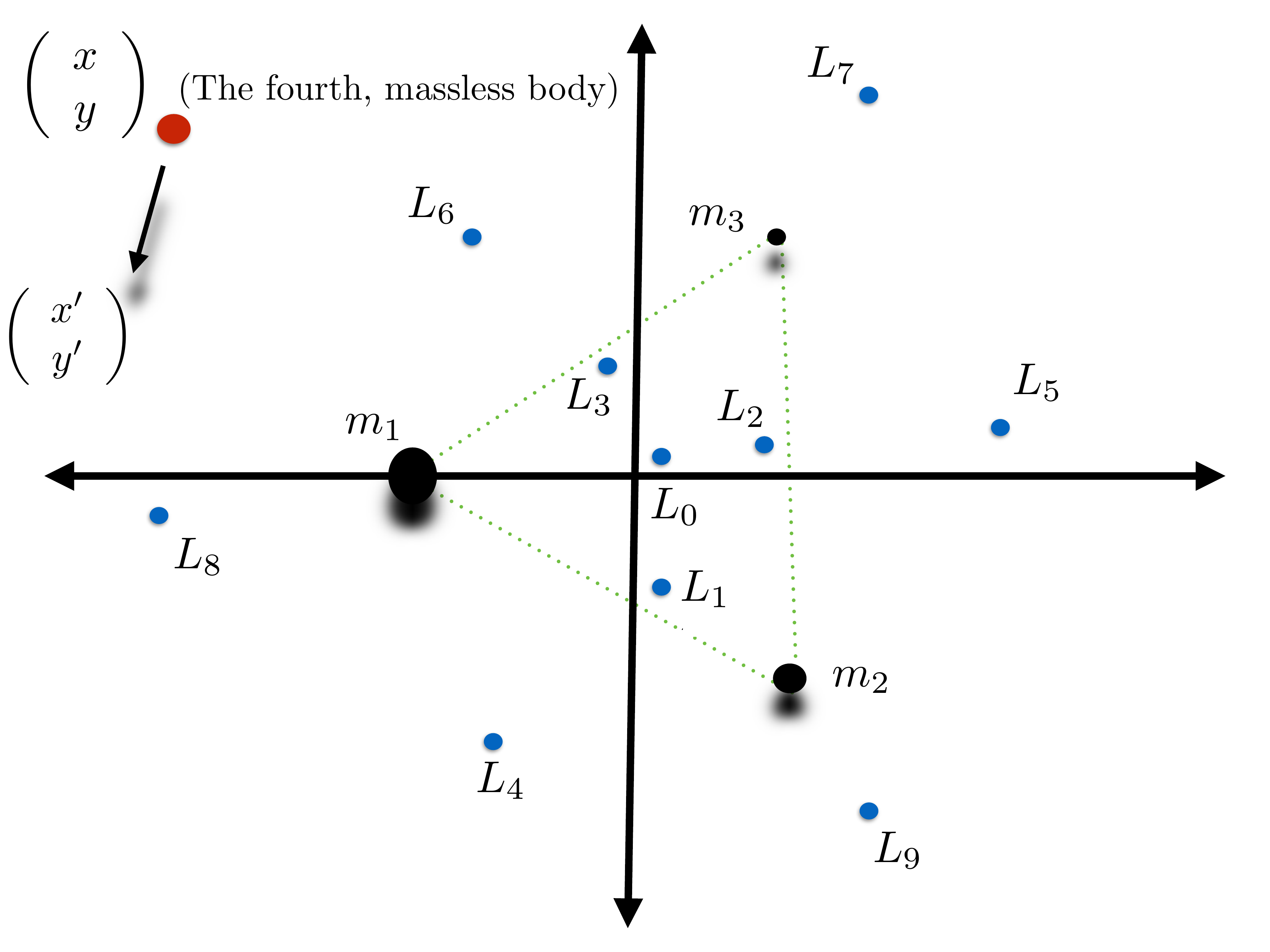}
\caption{\textbf{Configuration space for the CRFBP:} The three 
primary bodies with masses $m_1,m_2,$ and $m_3$ are 
arranged in an equilateral triangle configuration of 
Lagrange, which is  a relative equilibrium solution of the 
three body problem.    After transforming to a
co-rotating frame, we consider the motion of a fourth
massless body.  The equations of motion have 
$8$, $9$, or $10$ equilibrium solutions
(libration points) denoted by $L_j$
for $0 \leq j \leq 9$.  The number of libration points,
and their stability, vary depending on $m_1$, $m_2$, and $m_3$. 
The points $L_{0, 4,5,6}$
have saddle focus stability for some values of the masses.
The other libration points have either saddle $\times$ center
or center $\times$ center stability type for all values of the 
masses.  
}\label{rotatingframe}
\end{figure}

\subsection{CRFBP: equations of motion and basic properties} \label{sec:equationsOfMotion}
Define
\[
K = m_2(m_3 - m_2) + m_1(m_2 + 2 m_3).
\]
The locations $(x_1, y_1)$, $(x_2, y_2)$ and $(x_3, y_3)$ of the 
three primary bodies are given by 
\begin{align*}
x_1 &=   \frac{-|K| \sqrt{m_2^2 + m_2 m_3 + m_3^2}}{K},  
&
y_1 &=   0,  
\\
x_2 &=  \frac{|K|\left[(m_2 - m_3) m_3 + m_1 (2 m_2 + m_3)  \right]}{
2 K \sqrt{m_2^2 + m_2 m_3 + m_3^2} }, 
&
y_2  &=  \frac{-\sqrt{3} m_3}{2 m_2^{3/2}} \sqrt{\frac{m_2^3}{m_2^2 + m_2 m_3 + m_3^2}},
\\
x_3 &=  \frac{|K|}{2 \sqrt{m_2^2 + m_2 m_3 + m_3^2}},
&
y_3 &=  \frac{\sqrt{3}}{2 \sqrt{m_2}} \sqrt{\frac{m_2^3}{m_2^2 + m_2 m_3 + m_3^2}}.
\end{align*}
Let  
\begin{equation} \label{eq:CRFBP_potential}
\Omega(x,y) \bydef
\frac{1}{2} (x^2 + y^2) + \frac{m_1}{r_1(x,y)} + \frac{m_2}{r_2(x,y)} + \frac{m_3}{r_3(x,y)}, 
\end{equation}
where 
\begin{equation} \label{eq:def_r}
r_j(x,y) \bydef \sqrt{(x-x_j)^2 + (y-y_j)^2},  \quad \quad \quad j = 1,2,3,
\end{equation}
and write $\mathbf{x} = (x, \dot x, y, \dot y) \in \mathbb{R}^4$ to denote the
state of the system.   
The equations of motion in the co-rotating frame are given by 
\[
\mathbf{x}' = f(\mathbf{x}),
\]
where
\begin{equation}\label{eq:SCRFBP}
 f(x, \dot x, y, \dot y) \bydef
\left(
\begin{array}{c}
\dot x \\
2 \dot y + \frac{\partial \Omega}{\partial x} \\
\dot y \\
-2 \dot x + \frac{\partial \Omega}{\partial y} \\
\end{array}
\right).
\end{equation}
Observe that $\Omega$, and hence $f$, depend in a complicated 
way on the masses $m_1, m_2, m_3$ through the positions 
$(x_i, y_i)$, $i = 1,2,3$ of the primaries.  

The system conserves the \textit{Jacobi integral}  
\begin{eqnarray} 
E(x, \dot x, y, \dot y) &= 
-\left( {\dot x}^2 + {\dot y}^2 \right) + 2\Omega(x,y).
 \label{eq:energy}
\end{eqnarray}
Assuming that $(m_1, m_2, m_3) \in \mathfrak{S}_I$, the libration points
(equilibrium solutions of Equation \eqref{eq:SCRFBP}) are 
arranged as illustrated in the schematic of Figure \ref{rotatingframe}.
If $(m_1, m_2, m_3) \in \mathfrak{S}_{II}$, then the equilibrium
solutions are as in the schematic, except that $L_0$ and $L_2$ are not
present.

There is a substantial literature on the CRFBP, and in addition to the references
cited above we refer the interested reader also to the works of Baltagiannis and Papadakis 
\cite{MR2845212}, and \'{A}lvarez-Ram\'{i}erz and Vidal \cite{MR2596303} 
where many of the systems basic properties are discussed in detail.
Elementary families of periodic orbits are considered by 
Papadakis in \cite{MR3571218,MR3500916},
and by Burgos-Garc\'{i}a, Bengochea,  and Delgado in 
\cite{burgosTwoEqualMasses,MR3715396}.
A study by Burgos-Garc\'{i}a, Lessard, and Mireles James proves the existence 
of a number of spatial periodic orbits for the CRFBP \cite{jpJaimeAndMe}
(again with computer assistance).
An associated Hill's problem is derived and its periodic orbits are studied 
by Burgos-Garc\'{i}a and Gidea in \cite{MR3554377,MR3346723}.
Murray and Mireles James study transverse homoclinic chaos in the spatial 
CRFBP, for certain vertical Lyapunov families of periodic orbits
attached to the inner and outer libration points \cite{maximeMurray_connectionsVerticalLyap}.

Regularization of collisions are studied by \'{A}lvarez-Ram\'{i}rez, Delgado, and Vidal in 
\cite{MR3239345}.  Chaotic motions were studied numerically by Gidea and Burgos 
in \cite{MR2013214}, and by \'{A}lvarez-Ram\'{i}rez and Barrab\'{e}s in \cite{MR3304062}.
Perturbative proofs of the existence of chaotic motions are
found in the work of She, Cheng and Li
\cite{MR3626383,MR3158025,MR3038224},
and also in the work of 
Alvarez-Ram\'\i rez,  Garc\'{a},  Palaci\'{a}n,  and Yanguas
\cite{chaosCRFBP}.
A computer assisted method of proof for establishing the existence of chaotic 
motions at non-perturbative parameter values was developed by 
Kepley and Mireles James in \cite{MR3906230}.

\subsection{Equilibria: the critical case } \label{sec:criticalNumerics_firstOrder}
In the following discussion we treat $m_1$, $m_2$ as free parameters, 
eliminating the parameter $m_3$ via the equation 
$m_3 = 1 - m_1 - m_2$. To characterize the critical curve we seek 
values of $m_1$ and $m_2$ such that $D f$ is non-invertible at 
one of its Lagrangian points. The following proposition tells us in terms of 
the second derivatives of 
$\Omega$ whether or not $D \dynamics$ is invertible, and is used 
to (numerically) compute the bifurcation curve $\mathfrak{D}$. 
We also recover that a critical libration point always has a double 
zero eigenvalue.  See Figure \ref{fig:m1versusm2}.

\begin{proposition}[Bifurcation value]\label{thm: bifurcationpoint} The linearization 
of $f$ at a libration point $L =(x_0,0,y_0,0)$ has an eigenvalue 
$0$ iff  $\Omega_{xx}(x_0,y_0) \Omega_{yy}(x_0,y_0) = \Omega_{xy}^2(x_0,y_0)$.  
 In particular,
\begin{itemize} 
\item  \textit{If $L_c$ is a critical libration point for the CRFBP,
 then $Df(L_c)$ has eigenvalue $0$ with algebraic multiplicity $2$ and geometric multiplicity 
 $1$.} 
\end{itemize}
The eigenvector and generalized eigenvector corresponding to the eigenvalue $0$ are
\begin{align*}
\bv_0 &\bydef (\Omega_{yy}(x_0,y_0),0, - \Omega_{xy}(x_0,y_0) , 0), \\
\bv_1 &\bydef (\Omega_{xy}(x_0,y_0),\Omega_{yy}(x_0,y_0), 2 - \Omega_{xx}(x_0,y_0), - \Omega_{xy}(x_0,y_0))
\end{align*}
\end{proposition}

We start with a preliminary claim:

\begin{claim}Let $\boldx_0=(x_0,0,y_0,0)$ be a libration point 
of $f$. If $\Omega_{xx}(x_0,y_0) \Omega_{yy}(x_0,y_0) = \Omega_{xy}^2(x_0,y_0)$, 
then both $\Omega_{yy}(x_0,y_0)$ and $4 - \Omega_{xx}(x_0,y_0) - \Omega_{yy}(x_0,y_0)$ 
are non-zero.
\end{claim}

Convincing evidence for the claim is obtained by computing the 
Critical curve as illustrated in Figure \ref{fig:m1versusm2}.
One then plots $\Omega_{yy}(x_0,y_0)$ and 
$4-\Omega_{xx}(x_0,y_0) - \Omega_{yy}(x_0,y_0)$ 
for all $(x_0, y_0)$ in the critical curve and checks that 
they are numerically indeed far from zero. 
A complete proof follows from the results of  
\cite{MR2232439,MR2784870,MR3176322}.

\begin{proof}[of Proposition \ref{thm: bifurcationpoint}]
Writing 
$\Omega_{xx}=\Omega_{xx}(x_0,y_0)$, 
$ \Omega_{yy}= \Omega_{yy}(x_0,y_0)$, 
and $ \Omega_{xy}= \Omega_{xy}(x_0,y_0) $, we have
\begin{align}
p_{Df(L)}(\lambda)	\bydef
\begin{vmatrix}
-\lambda		&1			&0			&0			\\
\Omega_{xx}	&-\lambda	&\Omega_{xy}	&2			\\
0			&0			&-\lambda	&1			\\
\Omega_{xy}	&-2			&\Omega_{yy}	&-\lambda	\\
\end{vmatrix} 
= \lambda^4 + 
\left(4 - \Omega_{xx} - \Omega_{yy} \right) \lambda^2  
+ \Omega_{xx}\Omega_{yy} - \Omega_{xy}^2. \label{eq: CharacteristicPolynomial}
\end{align}
Then $Df(L)$ has an eigenvalue $0$ iff 
$\Omega_{xx}\Omega_{yy} =  \Omega_{xy}^2$. 
From our claim it follows that $0$ is a double 
root as $\Omega_{xx} + \Omega_{yy} \neq 4$.  
Furthermore, from the same claim it also follows that 
$\Omega_{yy} \neq 0$ and hence 
$\bv_0 \neq \bzero$. Direct calculation shows 
$Df(\boldx_0) \bv_0 = \bzero$ and $Df(\boldx_0) \bv_1 = \bv_0$.
\end{proof}

\begin{figure}[!t]
\centering
\includegraphics[width = 1.0 \textwidth]{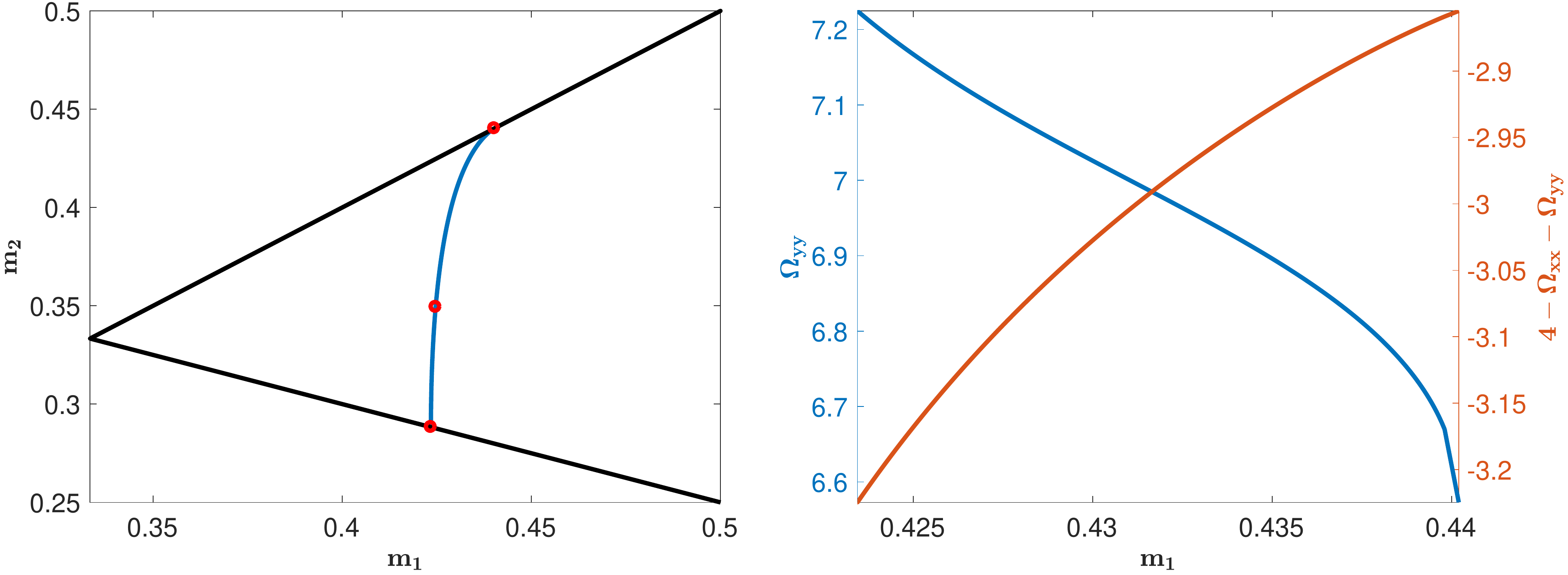}
\caption{\textbf{Critical equilibria:} left frame --  parameter simplex projected into the 
$m_1, m_2$ plane.  The blue curve is the numerically computed critical 
arc $\mathfrak{D}$.  The red dots denote special critical system parameters
considered throughout this work.  The lower red dot depicts the intersection of 
$\mathfrak{D}$ with the $m_2 = m_3$ edge of the simplex.  The upper red dot the 
intersection of $\mathfrak{D}$ with the $m_1 = m_2$ edge.  The middle red dot
is an arbitrarily chosen parameter set on $\mathfrak{D}$ interior to $\mathfrak{S}$.
Right frame -- The functions $\Omega_{yy}$ and $4 - \Omega_{xx} - \Omega_{yy}$ for the critical libration point $L_c$ on $\mathfrak{D}$.}
\label{fig:m1versusm2}
\end{figure}

In the introduction we claimed that at $v_{{\tiny \mbox{pf}}}$ there is a Hamiltonian pitchfork bifurcation instead of a Hamiltonian saddle node bifurcation. That is, we claim that the following bifurcations occur on $\mathfrak{D}$.
\begin{claim}Let $v_c = (m_1,m_2,m_3) \in \mathfrak{D}$ be a critical parameter set.
\begin{itemize}
\item If $m_1 \neq m_2$, the libration points $L_0$ and $L_2$ undergo a Hamiltonian saddle node bifurcation.
\item If $m_1 = m_2$, the libration points $L_0$, $L_2$ and $L_3$ undergo a Hamiltonian pitchfork bifurcation.
\end{itemize}
\end{claim}

We sketch the proof of this claim.
First, we note that, depending on the symmetries of the problem, both the Hamiltonian saddle node 
bifurcations and the Hamiltonian pitchfork bifurcation are co-dimension one. So, for example, one would assume 
that for $v \in \mathfrak{D}$ with neither $m_1 = m_2$ nor $m_2 = m_3$ the bifurcation at $v$ is a 
saddle node bifurcation thanks to lack of symmetry.  Furthermore, as we will see in 
Lemma \ref{thm: centerstablebranchformalseries}, if the constant 
\begin{align*}
\Omega_{xxx} \Omega_{yy}^3 - 3 \Omega_{xxy} \Omega_{yy}^2 \Omega_{xy} 
+ 3 \Omega_{xyy} \Omega_{yy} \Omega_{xy}^2 - \Omega_{yyy} \Omega_{xy}^3  \neq 0,
\end{align*}
at the critical libration point $L_c$, then the dynamics near the critical libration point resembles Figure \ref{fig: branches4247}. We will show numerically that for 
$v_c = (m_1,m_2,m_3) \in \mathfrak{D}$ with $m_1 \neq m_2$ we have
\begin{align*}
\Omega_{xxx} \Omega_{yy}^3 - 3 \Omega_{xxy} \Omega_{yy}^2 \Omega_{xy} 
+ 3 \Omega_{xyy} \Omega_{yy} \Omega_{xy}^2 - \Omega_{yyy} \Omega_{xy}^3 < 0,
\end{align*}
supporting the first part of the claim. 

Moreover, in Section \ref{sec:normalform} we compute the normal form for 
$v_c = (m_1,m_2,m_3) \in \mathfrak{D}$ at the left end point of $\mathfrak{D}$.
That is, at the intersection of $\mathfrak{D}$ with the $m_2 = m_3$ edge.
We then compute the normal form for $v_c = (m_1,m_2,m_3) \in \mathfrak{D}$ when 
$m_1 = 0.4247$, a ``generic point'' in $\mathfrak{D}$. In both cases, we find that the vector field on the center manifold agrees numerically with the normal form of the saddle node bifurcation.
This provides even further support for the first claim.

For the second part of the claim, we analytically show that  
 \begin{align*}
\Omega_{xxx} \Omega_{yy}^3 - 3 \Omega_{xxy} \Omega_{yy}^2 \Omega_{xy} + 3 \Omega_{xyy} \Omega_{yy} \Omega_{xy}^2 - \Omega_{yyy} \Omega_{xy}^3 = 0
\end{align*}
at the libration point $L_c$ for $v_{{\tiny \mbox{pf}}} \in \mathfrak{D}$. Furthermore, in Section 
\ref{sec:normalform} we will also compute the vector field on the center manifold for $v_{{\tiny \mbox{pf}}} \in \mathfrak{D}$ 
which agrees numerically with the normal form of the pitchfork bifurcation, supporting the second part of the claim.

\subsection{Short homoclinic connections in the triple Copenhagen problems} 
\label{sec:copenhagenHomoclinics}

\begin{figure}[!t]
\centering
\includegraphics[width = 1.0 \textwidth]{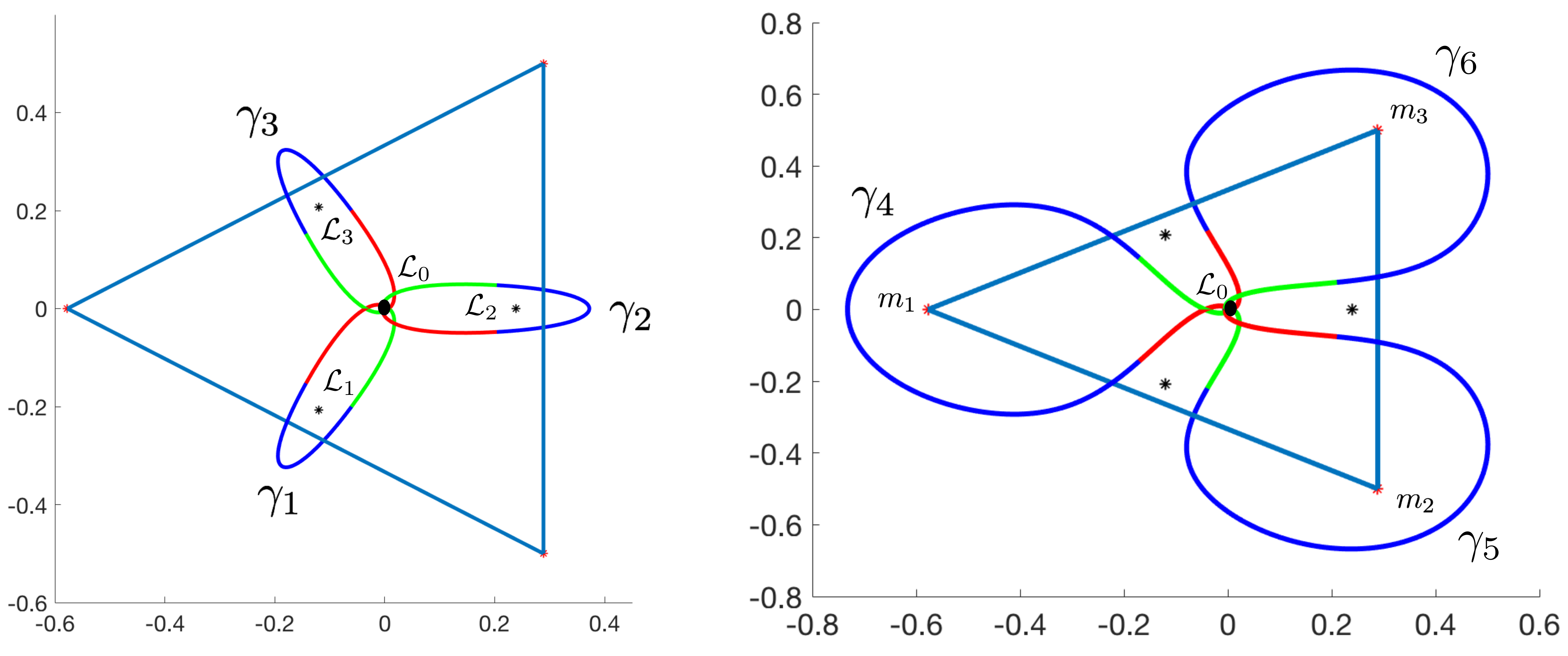}
\caption{\textbf{Fundamental homoclinics at $L_0$:} Left frame -- the three shortest 
homoclinic orbits at $L_0$, which we refer to as $\gamma_{1,2,3}$ 
depending on which libration point they wind around.
Right frame -- the fourth, fifth, and sixth shortest homoclinics at $L_0$, which we refer to as 
$\gamma_{4,5,6}$.  In the triple Copenhagen problem these orbits are related by 
rotational symmetry, however continuation away from equal masses will break this symmetry.
 In both frames the red and green part of the homoclinic is the portion described by the 
 parameterized local unstable and stable manifolds respectively.  The blue portion of the curve 
 is obtained by solving the projected boundary value problem.}
\label{fig:basicHomoclinics}
\end{figure}

Homoclinic connecting orbits for the saddle-focus equilibria in the triple Copenhagen 
problem were studied in detail 
in \cite{MR3919451}.   The main result at $L_0$ is that there are six basic 
homoclinic orbits which appear to organize the full web of connections. 
These are denoted $\gamma_{i} \colon \mathbb{R} \to \mathbb{R}^4$
for $i = 1, \ldots, 6$ and are illustrated in Figure \ref{fig:basicHomoclinics}.
These orbits are the shortest homoclinics, both in terms of arc length and 
time of flight from a local unstable to a local stable invariant manifold.  
The shortest orbits $\gamma_1, \gamma_2, \gamma_3$ each wind once around the 
the libration point $L_1, L_2$ or $L_3$ respectively.  
Indeed, each appears to participate in a blue sky catastrophe with the 
corresponding planar Lyapunov family.

There appear to be infinitely many additional homoclinic orbits ``shadowing'' 
the basic connections in any order we wish.  That is, consider a word 
$\Gamma$ composed of any combination of the letters $\gamma_i$, 
$i = 1, \ldots, 6$.  There is a homoclinic orbits which passes close to 
$\gamma_i$ in the prescribed order.  These results and more are 
discussed in detail in  \cite{MR3919451}.

An important remark is that the results of  \cite{MR3919451} 
show numerically that the homoclinics $\gamma_i$, $i = 1, \ldots, 6$ 
 are transverse in the energy level set of 
$L_0$. Then each of the homoclinic orbits persist for a small change in 
parameter values.  
While some preliminary numerical continuations were 
discussed in \cite{MR3919451}, the calculations were neither 
systematic nor were they taken all the way
to the critical curve $\mathfrak{D}$.

\subsection{Hamiltonian homoclinic bifurcations: classic results} \label{sec:bifurcations}
Homoclinic orbits are fundamental objects of study in dynamical systems theory, 
and there exists a vast literature on their properties, numerical calculation, and 
bifurcations.  Even in the special case of Hamiltonian systems this is a rich 
area and we only recall as much of the theory as pertains directly to the 
present study.  A fantastic overview of this theory with an in depth discussion of the 
literature is found in \cite{MR1605836}.  
We identify five types of global bifurcations involving homoclinic connections in 
Hamiltonian systems.  We state the results for four dimensional vector fields, though 
more general results are found in the references.  

\begin{itemize}
\item \textbf{Type I -- Hamiltonian bi-focus homoclinics:} Suppose that 
$f \colon \mathbb{R}^4 \to \mathbb{R}^4$ is a Hamiltonian vector field 
and that $L \in \mathbb{R}^4$ is an equilibrium solution with complex conjugate
eigenvalues $\pm \alpha \pm i \beta$, $\alpha, \beta > 0$.
Assume that $\gamma \colon \mathbb{R} \to \mathbb{R}^4$ is a transverse homoclinic 
orbit for $L$.  Then there are infinitely many chaotic horseshoes near $\gamma$.
See \cite{MR0442990}. In addition, there is a tube of periodic orbits accumulating to $\gamma$.
See \cite{MR0365628,MR3253906}.  This is the ``blue sky'' catastrophe already 
mentioned above.  
\item \textbf{Type II --Belyakov-Devaney bifurcation:} Suppose that $f(x, \mu)$ is a one parameter
family of Hamiltonian systems, and that for $\mu \in (\mu_0-\epsilon, \mu_0 +\epsilon)$, $L(\mu)$ is 
a libration point for the vector field $f(x, \mu)$.  Suppose that for $\mu < \mu_0$ $L(\mu)$ 
is a bi-focus, and that for $\mu > \mu_0$ $L(\mu)$ has real distinct eigenvalues 
$\pm \alpha,  \pm \beta$, $\alpha, \beta > 0$.
Then $L(\mu_0)$ has real repeated eigenvalues.  Assume that the repeated 
eigenvalues have geometric multiplicity one, and algebraic multiplicity two, and that $\gamma_\mu(t)$ is 
a smooth family of homoclinic connecting orbits for $L(\mu)$.  Then (under some generic 
non-degeneracy assumptions) there are infinitely many homoclinic doubling bifurcations
of $\gamma_{\mu_0}(t)$ at $\mu_0$.  For the precise statement of the theorem and its 
proof see \cite{MR1240679}.
\item \textbf{Type III --Transverse connections to a Hamiltonian saddle node:}
Suppose that $f(x,\mu)$ is a one parameter family of Hamiltonian systems with 
a saddle-node bifurcation at $\mu_0$. To be more precise, let
$L \in \mathbb{R}^4$ denote
the critical libration point and suppose that $L$ has a double zero eigenvalue with 
geometric multiplicity one, algebraic multiplicity two,
and two real eigenvalues $\pm \alpha$.  Moreover, the normal form at $L$ has non-zero 
quadratic term.  Then $L$ has three dimensional center stable and three dimensional 
center unstable manifolds.  The restriction of these three dimensional center stable/unstable
manifolds to the energy level set of $L$ results in a pair of two dimensional invariant manifolds, 
which may intersect transversally in the energy level, giving rise to
a non-degenerate homoclinic orbit $\gamma$.  In general, there will be two families of 
transverse homoclinic orbits which annihilate at $\gamma$, and a family of periodic 
orbits born out of the disappearance of $\gamma$.
The precise statement of the theorem and 
 its proof are found in \cite{MR1967329}.
\item \textbf{Type IV --Degenerate homoclinic orbits in conservative systems:}
this co-dimension one phenomenon 
is studied in \cite{MR1464081,MR2185157}.  
A transverse homoclinic orbit in a one parameter family of Hamiltonian 
systems can be continued until the loss of transversality.
In short, the theorem states that before the bifurcation the generic 
situation is that there is a pair of transverse homoclinic orbits which 
collide and annihilate at criticality.    After the bifurcation the homoclinic families are 
gone.
\item \textbf{Type V-- Degenerate connection to a Hamiltonian saddle node:}
In a two parameter Hamiltonian system a type III homoclinic bifurcation 
can be continued along a one dimension curve in parameter space until 
the critical homoclinic loses transversality.   
In short, the result is that near a 
non transverse critical homoclinic orbit there is a pair of critical transverse homoclinic orbits which 
collide and annihilate.
See \cite{MR1046796,MR1395047} for more complete discussion.
\end{itemize}

In addition, near a Hamiltonian saddle-node or Hamiltonian pitchfork 
bifurcation, normal form analysis provides additional local homoclinic
bifurcations.  For example, near a Hamiltonian saddle node bifurcation there 
are two families of libration points $L_1(\mu)$ and $L_2(\mu)$ which collide
and annihilate at $\mu_0$.  Without loss of generality we have that $L_1(\mu)$
has distinct real eigenvalues and $L_2(\mu)$ has saddle-center stability 
for $\mu < \mu_0$.  Analysis of the normal form shows that there is a family 
of ``small'' homoclinic orbits for $L_1(\mu)$, which wind around $L_2(\mu)$ 
for $\mu < \mu_0$.  This family of homoclinic orbits shrink and disappear 
when  $L_1(\mu)$ and $L_2(\mu)$ collide and annihilate at $\mu_0$.   
See for example \cite{MR1200201,MR1350541}.
Near a Hamiltonian pitch fork a similar normal form analysis 
shows that near the pitch fork there are two families (related by symmetry) 
of ``small'' homoclinic orbits which disappear in the pitch fork bifurcation.  
See again \cite{MR1200201,MR1350541}.

These classical results inform our intuition throughout the remainder of the paper.

\begin{figure}[!t]
\centering
\includegraphics[height = 0.35 \textheight]{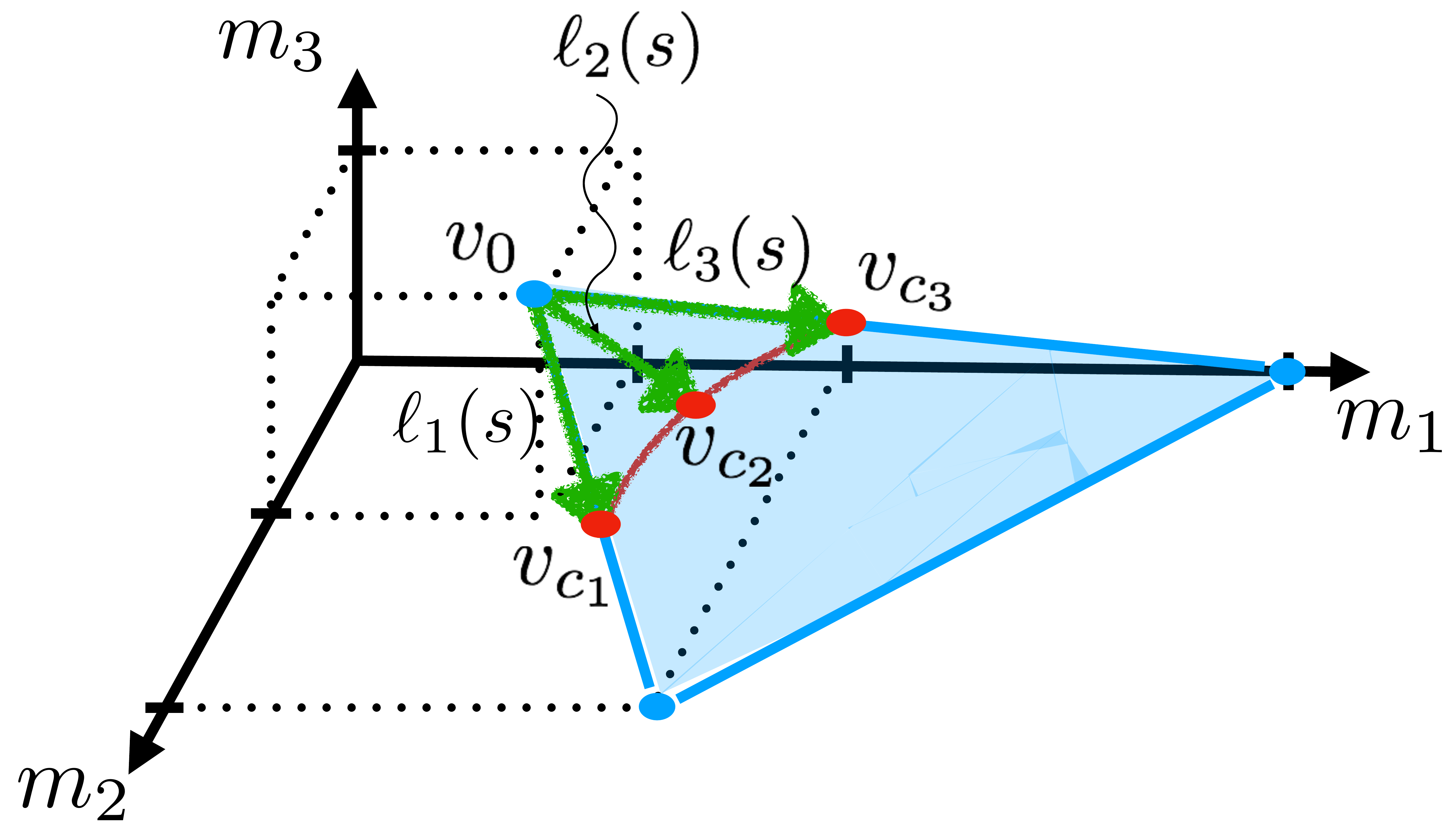} 
\caption{\textbf{Three parameter continuation arcs:} 
consider three critical parameter sets $v_{c_1}, v_{c_2}$, and $v_{c_3}$ on the 
curve $\mathfrak{D}$ and define the lines from $v_0$ to each of these.
We denote the lines by $\ell_1(s)$, $\ell_2(s)$, and $\ell_3(s)$, and 
study the one parameter numerical continuation problem for 
$\gamma_1, \gamma_2$, and $\gamma_3$ on these lines.}
\label{fig:parameterCurves}
\end{figure}

\section{Robustness of the short Copenhagen $L_0$ homoclinics} \label{sec:numCont}
We now consider robustness 
with respect to parameter perturbations
of $\gamma_j$ for $j = 1,2,3$.  Our idea is to 
use classical numerical continuation algorithms for Hamiltonian 
homoclinic connections to study the $\gamma_j$ as the parameters 
of the system are moved away from $v_0 = (1/3, 1/3, 1/3)$. 
Numerical continuation algorithms for periodic and connecting orbits are 
reviewed briefly in Appendix \ref{sec:appendixContinuation} for the sake of completeness, 
however the reader seeking a thorough overview is referred to 
 the classic works of \cite{MR1007358,MR2003792,MR1992054,MR2989589}. 

We fix three lines in parameter space, each starting at $v_0$ and terminating on $\mathfrak{D}$,
so that we obtain three one parameter continuation problems.  
To begin, we numerically compute three critical parameter sets 
$v_{c_1}, v_{c_2}, v_{c_3} \in \mathfrak{D}$ with 
\[
v_{c_1} \approx 
\left( 
\begin{array}{c}
0.440201606048930 \\
0.440201606048930 \\
0.119596787902140
\end{array}
\right), \quad 
v_{c_2} \approx
\left( 
\begin{array}{c}
0.4247   \\
0.349370273506504   \\
0.225929726493496
\end{array}
\right),
\quad 
\mbox{and}
\quad
v_{c_3} \approx
\left( 
\begin{array}{c}
0.423447616433011 \\
0.288276191783495 \\
0.288276191783495
\end{array}
\right).
\]
Here $v_{c_1}$ is on the $m_1 = m_2$, and $v_{c_3}$ on the 
$m_2 = m_3$ parameter edges respectively. The parameters $v_{c_2}$ 
are taken near ``the middle'' of the critical curve, with $m_1 = 0.4247$
fixed somewhat arbitrarily.  
Define the parameter lines 
\[
\ell_{k}(s) = (1-s) v_0 + s v_{c_{k}},
\]
where $k = 1,2,3$.
The critical parameters and parameter curves are 
illustrated schematically in Figure \ref{fig:parameterCurves}.

Since the homoclinic orbits $\gamma_{j}$ for $j = 1,2,3$ are 
transverse in the $L_0$ energy level set of the triple Copenhagen 
problem,  each persists under small changes in the parameters. 
We write $\gamma_{j, k}(s)$ to denote the one parameter family
of homoclinic orbits obtained by parameter continuation of $\gamma_j$ 
along the parameter line $\ell_k(s)$. 
Following the discussion in Section \ref{sec:bifurcations}, 
a homoclinic connection can breakdown/disappear only under one 
of the two following scenarios:
\begin{itemize}
\item (A) loss of transversality, or 
\item (B) disappearance of the underlying equilibrium solution itself.
\end{itemize}
Note that in scenario (A) the equilibrium solution may persist 
after the homoclinic disappears, and that scenario (B) occurs only 
at $\mathfrak{D}$.  So for a given $1 \leq j,k \leq 3$, the question is:
\textit{does that family } $\gamma_{j,k}(s)$ \textit{survive all the way to $\mathfrak{D}$ 
or does it breakdown before?}

The question is made more quantitative as follows.
Since transversality is an open condition, 
there exist $0 < \hat{s}_{j,k} \leq 1$, $j,k = 1,2,3$ so that the 
one parameter family  $\gamma_{j,k}(s)$ 
exists for all $0 \leq s < \hat{s}_{j,k}$.  If $\hat{s}_{j,k} < 1$, then the 
homoclinic family loses transversality and disappears at $s = \hat{s}_{j,k}$.
If $\hat s = 1$, then the homoclinic survives all the way to $\mathfrak{D}$.
In any event, the number $\hat{s}_{j,k}$ provides a measure of the 
robustness of $\gamma_{j,k}$.

The general principle informing our work in this section is that numerical 
continuation algorithms are based on Newton's method, 
and Newton's method must break down near a bifurcation by the implicit 
function theorem.  So, if $\hat{s}_{j,k} < 1$, then the numerical continuation of $\gamma_{j,k}$ 
must break down before we reach this critical value.  
On the other hand, homoclinic orbits can undergo other
types of bifurcations, such as homoclinic doubling \cite{MR965875} or the 
Belyakov-Devaney bifurcation discussed in Section \ref{sec:bifurcations}.
Then the breakdown of numerical continuation provides an indicator that the 
homoclinic is undergoing some kind of bifurcation, and hence provides a
lower bound on the value of $\hat{s}_{j,k}$ and the disappearance of the homoclinic.
We also remark that numerical continuation breaks down near $\mathfrak{D}$ because 
the underlying equilibrium undergoes a bifurcation.

\begin{figure}[!t]
\centering
\includegraphics[height = 0.45  \textheight]{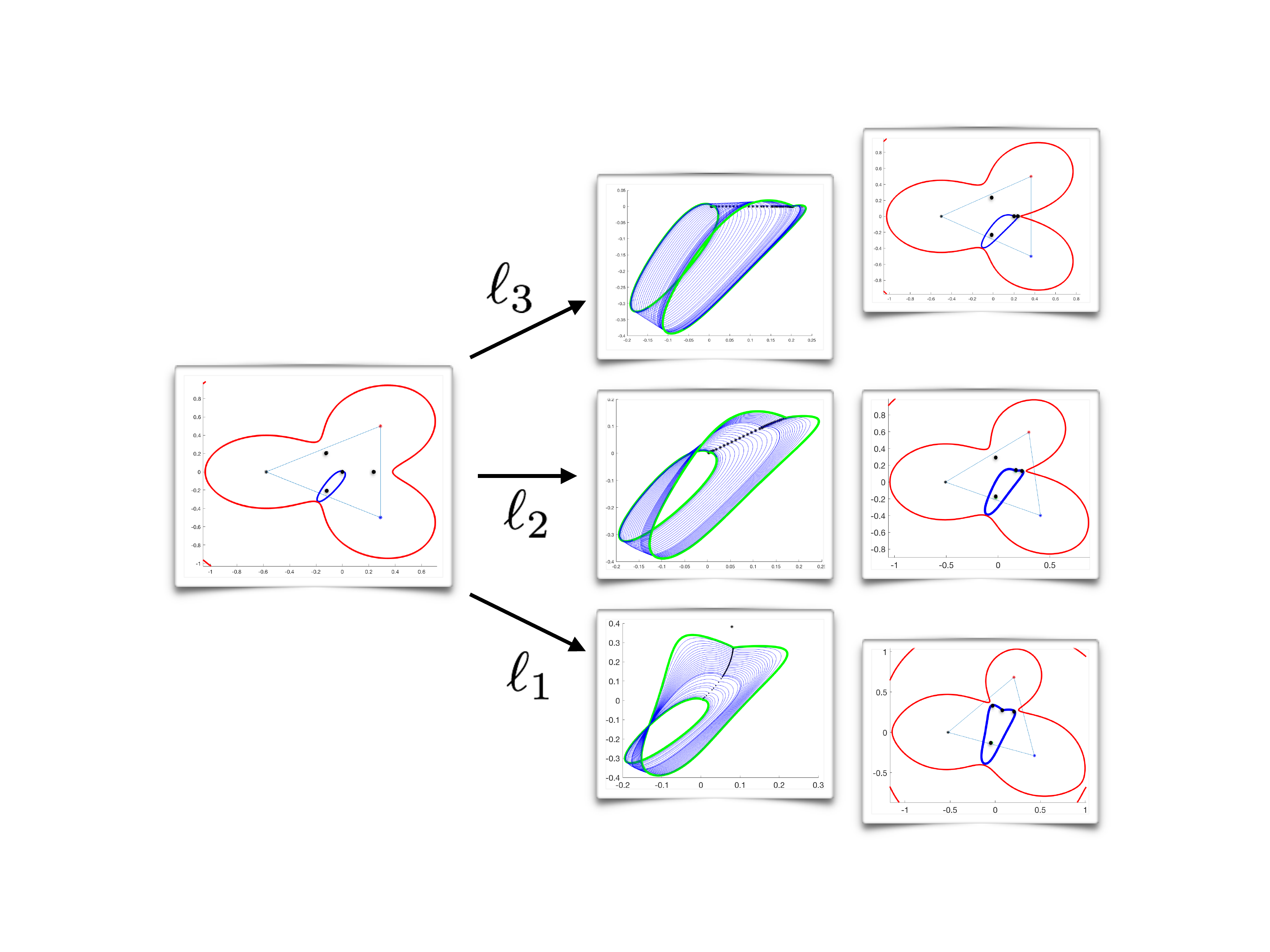} 
\caption{\textbf{Numerical continuation of $\gamma_1$:}  
continuations along $\ell_{1,2,3}(s)$ are successful roughly
$74 \%, 92 \%$, and $97 \%$ of the way to the critical 
curve $\mathfrak{D}$ respectively.
The left frame illustrates the homoclinic orbit $\gamma_1$ (blue
curve) in the triple Copenhagen problem, that is 
when the parameters are $v_0 = (1/3, 1/3, 1/3)$.  The red curve
is the zero velocity curve for the $L_0$ energy level.  The three middle 
frames illustrate the results of numerical continuation of $\gamma_1$ along 
the parameter lines $\ell_{1,2,3}(s)$.  In each case the initial and final 
numerically computer homoclinics are colored green and the 
intermediate homoclinics are colored blue.  The step size in the continuation 
algorithm is chosen adaptively, so that the blue homoclinic 
orbits are not uniformly spaced.  The black dots illustrate the 
numerical continuation of the libration points $L_0$.  The three right frames
illustrate the terminal homoclinic orbits $\gamma_{1, k}(s)$ for 
$s \approx \hat{s}_{1,k}$, $k = 1,2,3$ with the corresponding zero 
velocity curves.  In each of the three frames on the right we see that 
there are four inner libration points.  That is, the numerical
continuation did not reach the critical curve $\mathfrak{D}$. 
}
\label{fig:gamma_1}
\end{figure}

\subsection{Continuation of the $\gamma_1$ family} \label{sec:continuation_gamma1}
Applying the strategy of the previous section starting from $\gamma_1$ 
leads to the following results.  We find that the continuation algorithm 
breaks down near $s = 0.74$ when we continue along the $\ell_{1}(s)$
parameter line, near  $s = 0.9247$ along the $\ell_{2}(s)$ parameter line, 
and near $s = 0.974$ along the $\ell_3(s)$ parameter line. 
The results are illustrated graphically in Figure \ref{fig:gamma_1}, and suggest
that the $\gamma_{1,1}(s)$ family is the least robust, $\gamma_{1,3}(s)$
the most robust, and $\gamma_{1,2}(s)$ is in between.  The 
results suggest also that none of the $\gamma_1$ continuations
survive all the way to the critical curve $\mathfrak{D}$.  Rather, 
in the terminology of Section \ref{sec:bifurcations},
$\gamma_1$ appears to undergo a type IV bifurcation along each of the 
parameter lines.  More precisely we conjecture that the numerical lower bounds
\[
 0.74 < \hat{s}_{1,1}, \quad \quad 
0.9247 < \hat{s}_{1,2} \quad \quad \mbox{and} \quad \quad 
0.974 < \hat{s}_{1,3},
\]
hold along the $\ell_{1,23}(s)$ parameter lines.

\begin{figure}[!t]
\centering
\includegraphics[height = 0.45  \textheight]{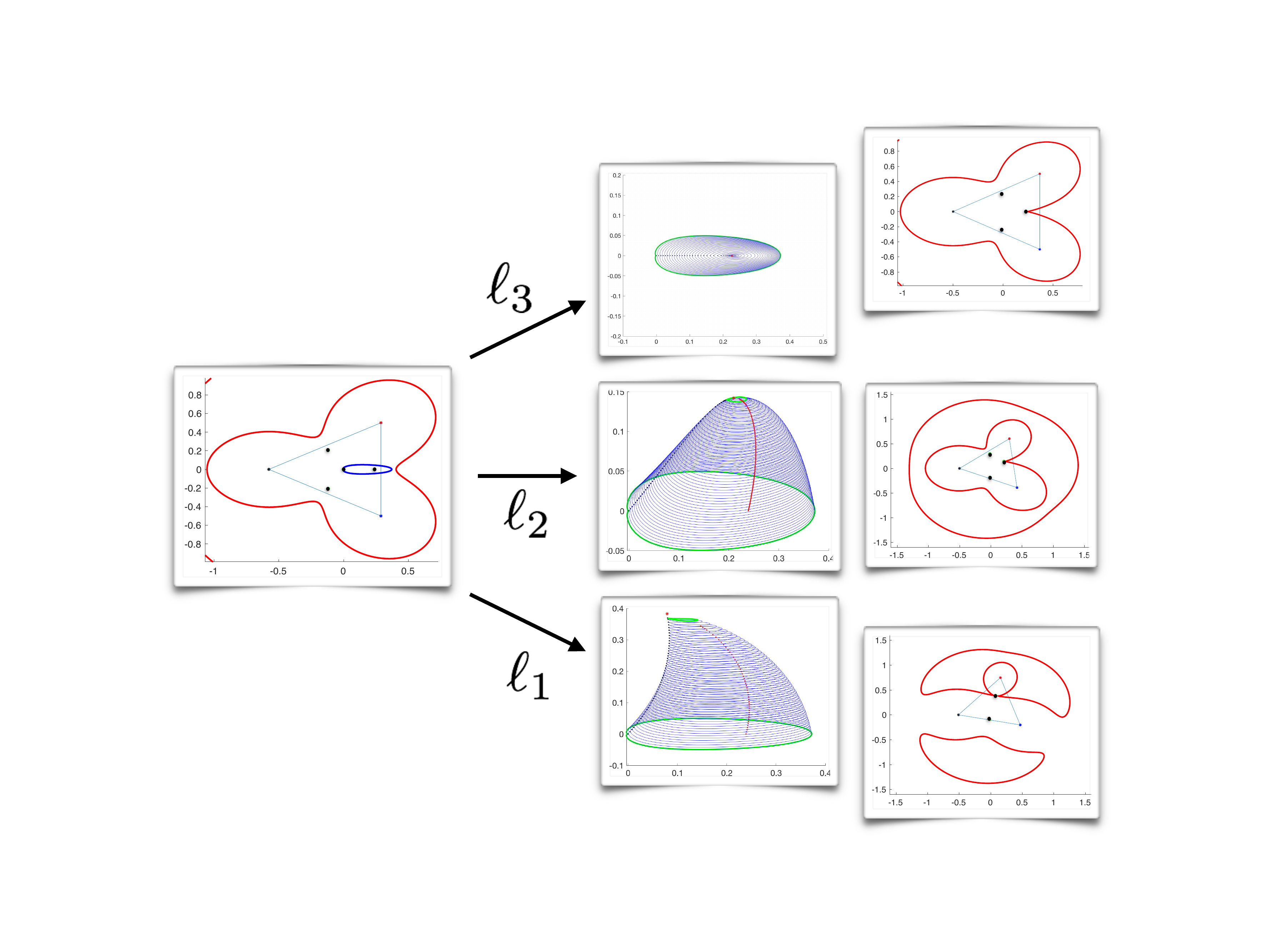} 
\caption{\textbf{Continuation of the $\gamma_2$ family:}
continuations along $\ell_{1,2,3}(s)$ are successful almost
$100 \%$ of the way to the critical 
curve $\mathfrak{D}$.
Objects are colored as described in the caption of 
Figure \ref{fig:gamma_1}.  In each of the middle frames one sees 
the homoclinic orbits shrinking to zero as $L_0$ approaches $L_2$.
Continuation of the $L_2$ family is shown in the bottom and middle 
frames (red curves).  Since the families continue almost all the way to 
$s = 1$ the right frames show the zero velocity curves at $v_{c_{1,2,3}}$
when the homoclinics have vanished.  
}
\label{fig:gamma_2}
\end{figure}

\subsection{Continuation of the $\gamma_2$ family} \label{sec:continuation_gamma2}
Continuation of the $\gamma_2$ family appear to be more straight forward, 
as in every case the continuation succeeds until $s \approx 1$, breaking down only 
when $L_2$ approaches $L_0$ and the homoclinics become very small (distance on the order of
$10^{-4}$) when they become 
difficult to distungish numerically.  
The results are illustrated in Figure \ref{fig:gamma_2}, and appear to indicate that 
the $\gamma_2$ families are 
as robust as possible.  Based on these observations we conjecture that  
\[
  \hat{s}_{2,1} =  \hat{s}_{2,2} = \hat{s}_{2,3} = 1.
\]
That is: along the parameter lines
$\ell_{1,2,3}(s)$, the $\gamma_{2}$ families continue all the way to the critical set 
$\mathfrak{D}$.
Indeed, we will see in Section \ref{sec:normalform}, that this conjecture is further 
supported by normal form calculations at $L_c$.

\begin{figure}[!t]
\centering
\includegraphics[height = 0.45  \textheight]{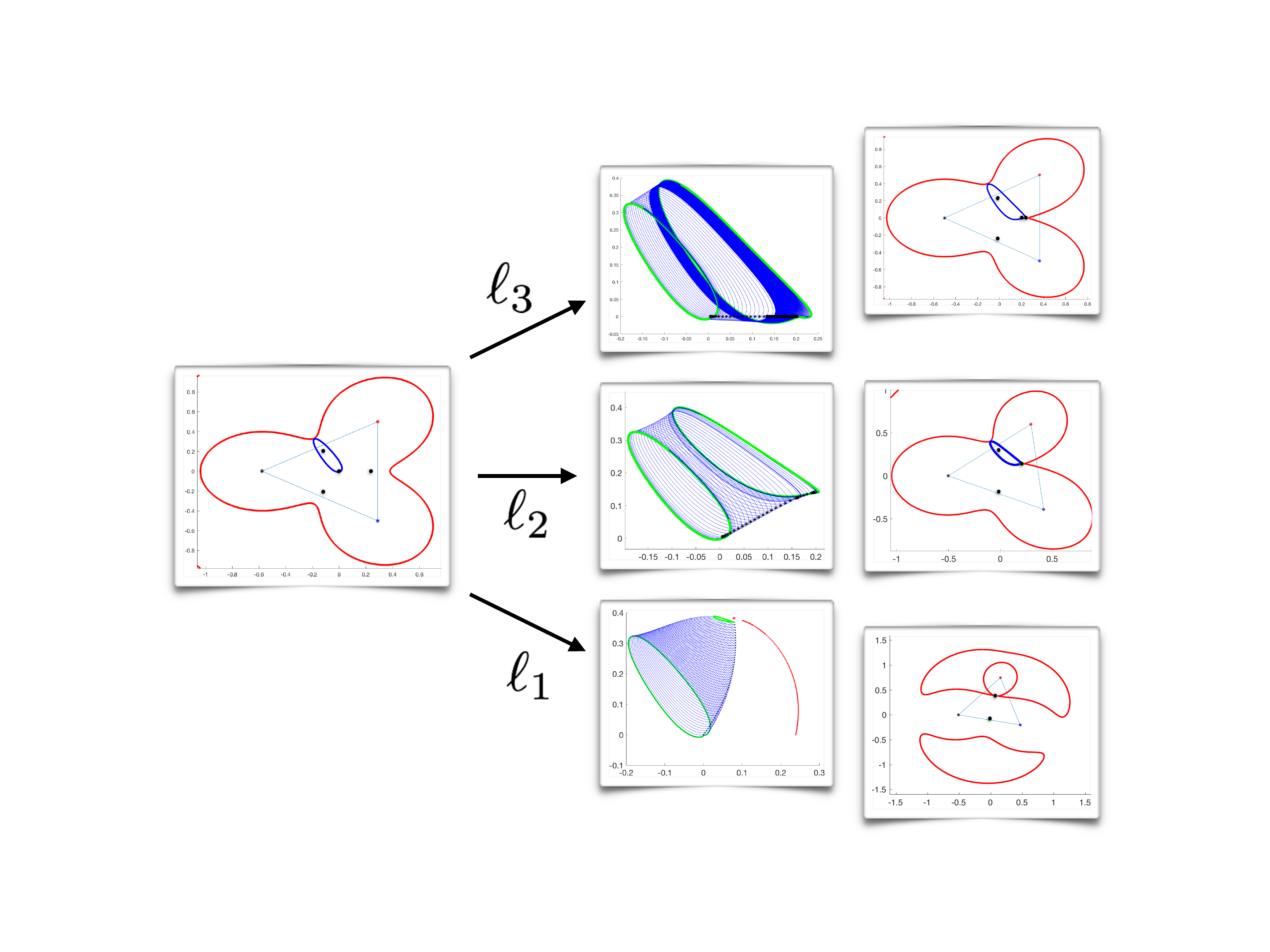} 
\caption{\textbf{Numerical continuation of the $\gamma_3$ basic homoclinic:} 
continuations along $\ell_{1,2,3}(s)$ are successful roughly
$100 \%, 97 \%$, and $97 \%$ of the way to the critical 
curve $\mathfrak{D}$ respectively .
Objects are colored as described in the caption of Figure \ref{fig:gamma_1}.
The left frame illustrates the homoclinic orbit $\gamma_3$ (blue
curve) when the parameters are $v_0 = (1/3, 1/3, 1/3)$.  
The black dots in the middle three frames illustrate the 
numerical continuation of the libration points $L_0$,
while the red dots in the bottom middle frame illustrate the 
continuation of $L_2$.   The top and middle right frames
illustrate the terminal homoclinic orbits $\gamma_{3, k}(s)$ for 
$s \approx \hat{s}_{3,k}$, $k = 1,2$ with the corresponding zero 
velocity curves. Note that in both cases the libration points $L_0$ and
$L_2$ have not quite collided, but that the homoclinics are difficult to 
continue any further.  (In the middle right frame it is difficult to distinguish 
$L_0$ and $L_2$ graphically at this resolution).
The bottom right frame on the other hand 
illustrates the zero velocity curves in the
critical energy level for system parameters $v_{c_3}$, 
where $L_0, L_2, L_3$ have collided, and the homoclinic family appears to have
shrunk to zero.   
}
\label{fig:gamma_3}
\end{figure}

\subsection{Continuation of the $\gamma_3$ family} \label{sec:continuation_gamma3}
In considering the $\gamma_3$ family it is first useful to exploit the symmetries of the 
problems, in order that some of the results already discussed can be recycled.   
For example, when $m_2 = m_3$ the system is symmetric 
about the $x$ axis and that $\gamma_{3, 3}(s)$ family is obtained from the 
$\gamma_{1,3}(s)$ family by reflection.  Then, based on the results presented in 
Section \ref{sec:continuation_gamma1} the $\gamma_3$ family inherit the conjecture
\[
0.974 < \hat{s}_{3,3} = \hat{s}_{1,3},
\]
suggesting again a type IV bifurcation 
in the terminology of Section \ref{sec:bifurcations}.

Similarly, when $m_1 = m_2$
 the system has symmetry about the line through the third primary bisecting the 
 opposite edge of the Lagrangian triangle in phase space.  Thanks to this symmetry 
 the $\gamma_{3, 1}(s)$ homoclinic family is obtained by reflection of $\gamma_{2,1}(s)$
 family, and by the results in Section \ref{sec:continuation_gamma2} we inherit the 
 conjecture that 
 \[
 \hat{s}_{3,1} = \hat{s}_{2,1} = 1. 
 \]
 
Recall also that the critical bifurcation at $v_{c_1}$ is a pitch fork, wherein 
the libration points $L_0$, $L_2$ and $L_3$ collide.   
Then the $\gamma_{3,1}(s)$ family shrinks to zero as $L_3$ approaches $L_0$, 
exactly as the $\gamma_{2,1}(s)$ family shrinks as $L_2$ approaches $L_0$.  
 It follows that $\hat{s}_{3,1} = \hat{s}_{2,1} = 1$, and the $\gamma_{3,1}(s)$ 
 family is as robust as possible.  Indeed this is the picture predicted by the 
 normal form for the pitch fork bifurcation as discussed in Section \ref{sec:bifurcations}.
 
Finally, we report that numerical continuation of the $\gamma_{3,2}$ family succeeds almost $97\%$ of the 
the way to $v_{c_2}$ along the $\ell_{2}(s)$ parameter line.  Based on this we have the 
the bound
\[
0.97 < \hat{s}_{3,2}.
\]
More insight into the fate of the $\gamma_{3,2}$ family is obtained by considering the normal form 
calculations in Section \ref{sec:normalform}.
The results of all three numerical continuations are 
summarized in Figure \ref{fig:gamma_3}.

\begin{remark}[Belyakov-Devaney bifurcations] \label{rem:belDevBif}
It is worth remarking that each of the continuation families $\gamma_{j,k}(s)$ undergoes a 
type II bifurcation before the numerical continuation breaks down.  That is, in each case the underlying 
equilibrium solution $L_0(s)$ changes stability form a bi-focus to a saddle with real distinct eigenvalues 
for $s < \hat{s}_{j,k}$.  This suggests that the Belyakov-Devaney bifurcation is universal for the $\gamma_{1,2,3}$
families.  We note that while this bifurcation effects the computation of the stable/unstable manifolds of $L_0$,
the homoclinics do not break down there, and we can arrange that the numerical continuation 
``jumps over'' the bifurcation.
\end{remark}

\begin{remark}[A point of clarification regarding the ``false'' heteroclinic cycles] \label{rem:fakeOut}
Upon close inspection of  the right three frames of Figure \ref{fig:gamma_1} 
we note that in each of the three cases considered 
$L_2$ appears to lie on the critical $\gamma_1$ curve. This  
could suggest that $\gamma_1$ terminates in a heteroclinic bifurcation, resulting in 
a connection from $L_0$ to $L_2$ and back.
However the proposed heteroclinic cycle is impossible, 
as $L_0$ and $L_2$ are in different energy levels
when $\gamma_1$ is critical.  In fact $L_0$ and $L_2$ are only ever in the same energy 
level when they collide and disappear on the critical curve $\mathfrak{D}$.  By computing  
the tangent vectors along $\gamma_1$ we find that it ``passes over'' $L_2$  
with non-zero velocity, so that the suggestion of a heteroclinic orbit seen in 
Figure \ref{fig:gamma_1} is an effect of projecting into the plane.    
\end{remark}

\section{The blue sky test: continuation of the planar Lyapunov families of $L_{1,2,3}$}
\label{sec:blueSkies}
In this section we discuss calculations which can be used to refine the results of
Section \ref{sec:numCont}.  The idea is based on the fact, already mentioned in the introduction,
that $\gamma_{j}$, $j = 1,2,3$ appear as limits of the planar Lyapunov families 
associated with saddle $\times$ center libration points $L_j$, $j = 1,2,3$.  
Our experience suggests that this relationship is very stable with respect to parameter 
perturbations.  Indeed, let $L_{j,k}(s)$ denote the continuation of $L_j$ along the parameter line 
$\ell_{k}(s)$.  For $0 \leq s \leq \hat{s}_{j,k}$ we always find that 
$\gamma_{j,k}(s)$ is always the limit of the planar Lyapunov family 
associated with $L_{j,k}(s)$, $j,k = 1,2,3$.  
To put it another way, the tube of periodic orbits attached to $\gamma_{j,k}(s)$ 
appears to change its limit behavior only with the loss of transversality and 
disappearance of the homoclinic itself.

In this section we proceeded as if the converse of this statement holds.
That is, when the planar Lyapunov family associated with $L_{j,k}(s)$ does not 
accumulate to a homoclinic at $L_{0, k}(s)$, we take this as an indication that
the $\gamma_{j,k}(s)$ homoclinic family has terminated.  Hence, when we numerically
locate such an $s \in (0,1)$, we assume that $s > \hat{s}_{j,k}$.
Since the continuation based methods of Section \ref{sec:numCont} provide 
lower bounds on $\hat{s}_{j,k}$, the methods based on blue sky catastrophes
developed in this section provide upper bounds and hence numerical enclosures 
of the bifurcation parameter.

Of course this procedure is indicative rather than definitive, as the assumptions are based on 
heuristics rather than mathematically rigorous results.  A change in the terminal 
behavior of a tube indicates only that some bifurcation has occurred
in the homoclinic family, not necessarily that it has disappeared.    
Nevertheless, when used in conjunction with the continuation methods of Section 
\ref{sec:numCont} and the normal form/center manifold analysis of Section 
\ref{sec:normalform} we obtain a compelling narrative describing the global 
dynamics. We return to this point in 
Section \ref{sec:conclusions}.

\subsection{Blue skies for the $\gamma_1$ family}
Recall from Section \ref{sec:numCont} that the $\gamma_{1,1}(s)$ family 
enjoys the lower bound $0.74 < \hat{s}_{1,1}$.  In other words, we are 
fairly confident that family exists along more than $74$ percent 
of the $\ell_1(s)$ parameter line.  We 
apply numerical continuation (with respect to the energy)  
to the planar Lyapunov family associated with $L_{1,1}(0.74)$, and recover the 
homoclinic connection to $L_{0,1}(0.74)$ already computed by mass parameter continuation 
in the previous section.  

Now, taking $s = 0.78$ we numerically continue the 
Planar Lyapunov family associated with $L_{1,1}(0.78)$ and find that the periodic 
orbits do not accumulate to an orbit homoclinic to $L_{0,1}(0.78)$.  Instead, 
we find that we can continue the periodic orbits up to and beyond the energy 
level of $L_{0,1}(0.78)$. The Lyapunov family appears to eventually
accumulate on an orbit which collides 
with the third primary.   This another piece of evidence supporting the claim that 
the $\gamma_{1,1}(s)$ family terminates nearby.  Indeed we
appear to have the bound $\hat{s}_{1,1} < 0.78$.  Combining this with the 
results discussed in Section \ref{sec:continuation_gamma1}, we conjecture that 
\[
\hat{s}_{1,1} \in (0.74, 0.78).
\]

The results of the calculation just discussed are illustrated in 
Figure \ref{fig:gamma_1_blueSkies_m1m2}.  When we perform 
similar calculations for the $\gamma_{1,2}(s)$ and 
$\gamma_{1,3}(s)$ families of homoclinic orbits, 
and combine these with the lower bounds already obtained Section \ref{sec:numCont},
we obtain the enclosures 
\begin{align*}
\hat{s}_{1,2} & \in (0.9, 0.95)   \\
\hat{s}_{1,3} & \in (0.97, 1)
\end{align*}
The results are illustrated in Figures \ref{fig:gamma_1_blueSkies_middleCase} and
\ref{fig:gamma_1_blueSkies_m2m3}.

\subsection{Blue skies for the $\gamma_2$ family}
Since the $\gamma_2$ family appears to disappear 
with $L_2$ in the saddle node bifurcation, the kind of blue sky analysis 
discussed above is not available here.  Indeed, for $s$ near one the 
$L_2$ Lyapunov families do converge to the small homoclinics seen in 
Section  \ref{sec:numCont}, and at the critical energy $L_2$ is gone
so that there are no planar Lyapunov families to study.  
The $\gamma_2$ family is much more amenable to the local 
analysis performed in Section \ref{sec:normalform}.

\subsection{Blue skies for the $\gamma_3$ family}
As already noted in Section  \ref{sec:numCont}, the $\gamma_3$ family 
exhibits somewhat more complicated behavior than the $\gamma_1$ family --
which never persists to $\mathfrak{D}$, and the $\gamma_2$ family --
which always does.  
The behavior of the $\gamma_{3}$ family on the $\ell_1(s)$
and $\ell_3(s)$ parameter lines is forced by symmetry.  
So $\gamma_{3,3}(s)$ has the same terminal behavior as 
$\gamma_{1,3}(s)$, as one is obtained from the other by 
reflection about the $x$ axis.  
From this we have that 
\[
\hat{s}_{3,3} = \hat{s}_{1,3} \in   (0.97, 1), 
\]
thanks to the results for $\gamma_1$ already reported.
Similarly, the $\gamma_{3,1}$ family is related to the 
$\gamma_{2,1}$ family by a reflection.  Then the study of the 
$\gamma_{3,1}$ family is amenable to normal form analysis, 
as already remarked for the $\gamma_2$ family.  

Finally, recall from Section \ref{sec:numCont} that the $\gamma_{3,2}(s)$ family 
continued $97$ percent of the way to $\mathfrak{D}$. 
We remark that the blue sky test for $\gamma_{3,2}(s)$ is inconclusive, 
in the sense that for $s$ values near one on the $\ell_{2}(s)$ parameter curve 
the numerical continuation of the periodic family breaks down near the
$L_c$ energy level. It is unclear what this breakdown indicates -- does the
the periodic family accumulate to a homoclinic at $L_c$?  Or would a more
more computational effort show that the periodic family continues past the 
$L_c$ energy level?  
Again, since the problem occurs near $\mathfrak{D}$, 
the normal form analysis in Section \ref{sec:normalform} 
is seen to resolve the issue.

\begin{figure}[!t]
\centering
\includegraphics[height = 0.28   \textheight]{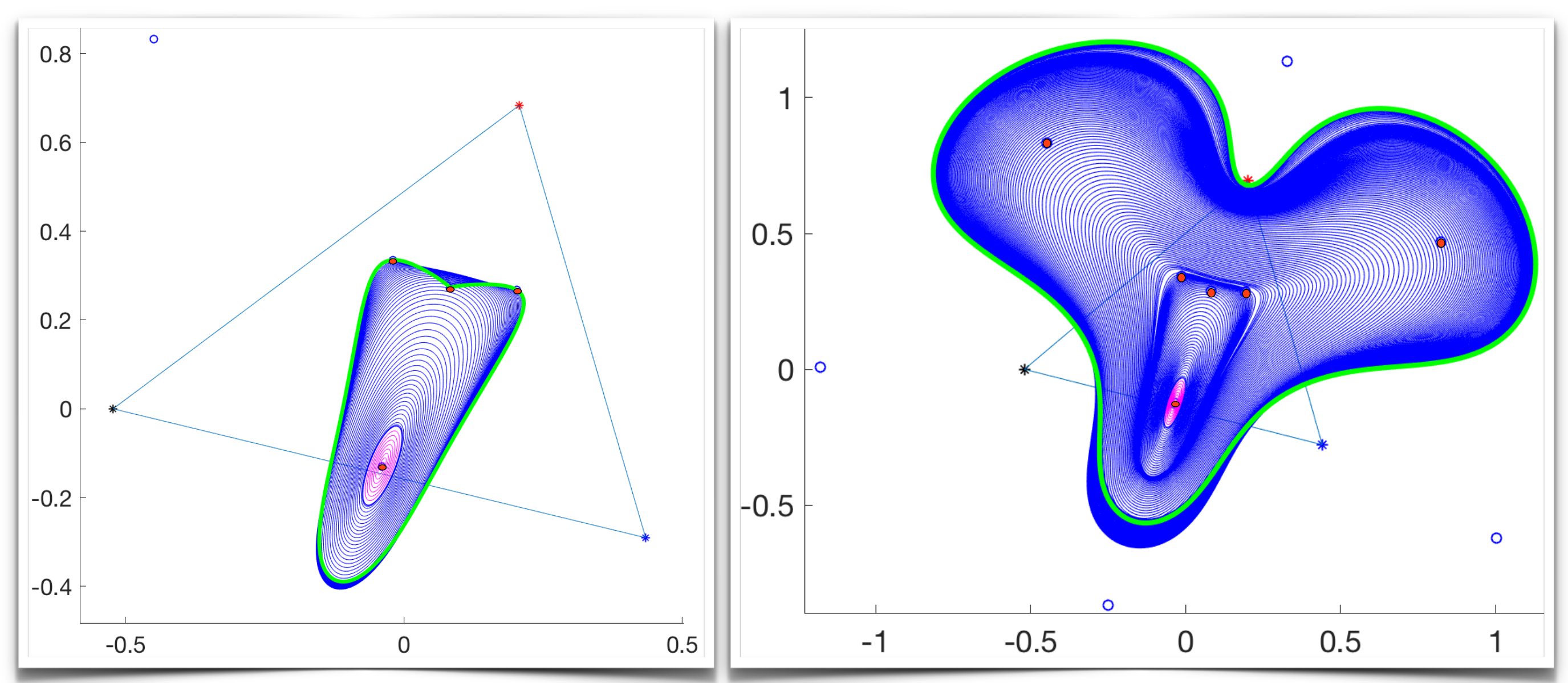} 
\caption{\textbf{Blue sky catastrophes along the $\ell_1(s)$ parameter curve:}
The left frame illustrates the CRFBP with parameter values $\ell_{1}(0.74)$.
We study the blue sky catastrophe for the planar Lyapunov
family of $L_{1,1}(0.74)$, and see that the periodic orbits
accumulate to an orbit homoclinic to $L_{0,1}(0.74)$ -- in fact the 
same homoclinic orbit depicted in the bottom right frame of 
Figure \ref{fig:gamma_1}, computed by numerical continuation 
along the $\ell_1(s)$ parameter curve starting from $\gamma_1$. 
The homoclinic orbit is represented by the
green curve, periodic orbits on the center manifold of $L_{1,1}(0.74)$ are 
represented by the magenta curves, and periodic orbits obtained by 
numerical continuation from the center manifold are represented by 
blue curves. The calculation suggests that $\hat{s}_{1,1} > 0.74$.
The right frame illustrates the planar Lyapunov family of $L_{1,1}(0.78)$, 
again computed by numerical continuation from the center manifold.  
The planar Lyapunov family appears to terminate at a collision with 
the small primary body.  The fact that the Lyapunov family does not 
accumulate to an orbit homoclinic to $L_{0,1}(0.78)$ suggests that 
$\hat{s}_{1,1} < 0.78$.
}
\label{fig:gamma_1_blueSkies_m1m2}
\end{figure}

\begin{figure}[!t]
\centering
\includegraphics[height = 0.28  \textheight]{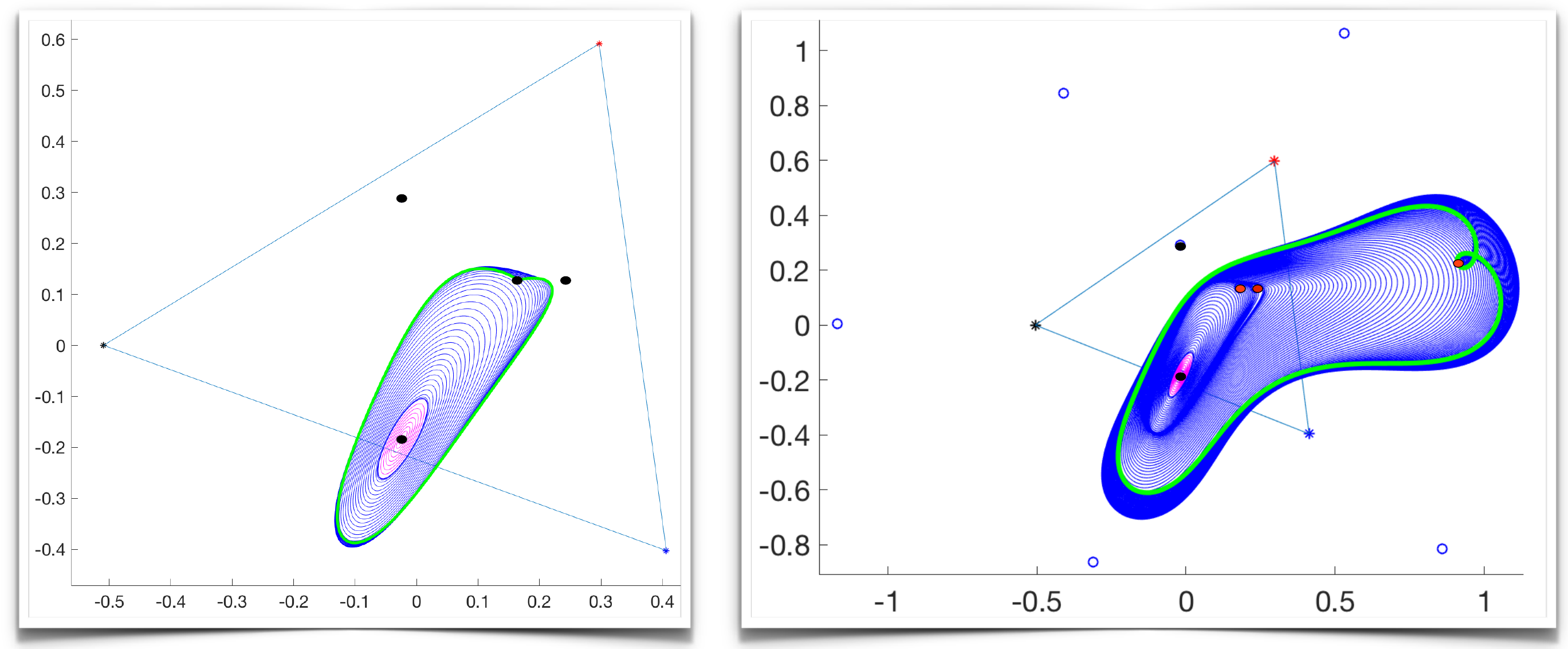} 
\caption{\textbf{Blue sky catastrophes on the $\ell_2(s)$ parameter curve:} 
The left frame illustrates the CRFBP with parameter values $\ell_{2}(0.9)$.
We study the blue sky catastrophe for the planar Lyapunov
family of $L_{1,2}(0.9)$ and reasoning just as in the caption of 
Figure \ref{fig:gamma_1_blueSkies_m1m2}
conclude that $\hat{s}_{1,2} > 0.9$.
Similarly, the right frame illustrates the planar Lyapunov family of $L_{1,2}(0.95)$, 
and suggests that $\hat{s}_{1,2} < 0.95$.
}
\label{fig:gamma_1_blueSkies_middleCase}
\end{figure}

\begin{figure}[!t]
\centering
\includegraphics[height = 0.28   \textheight]{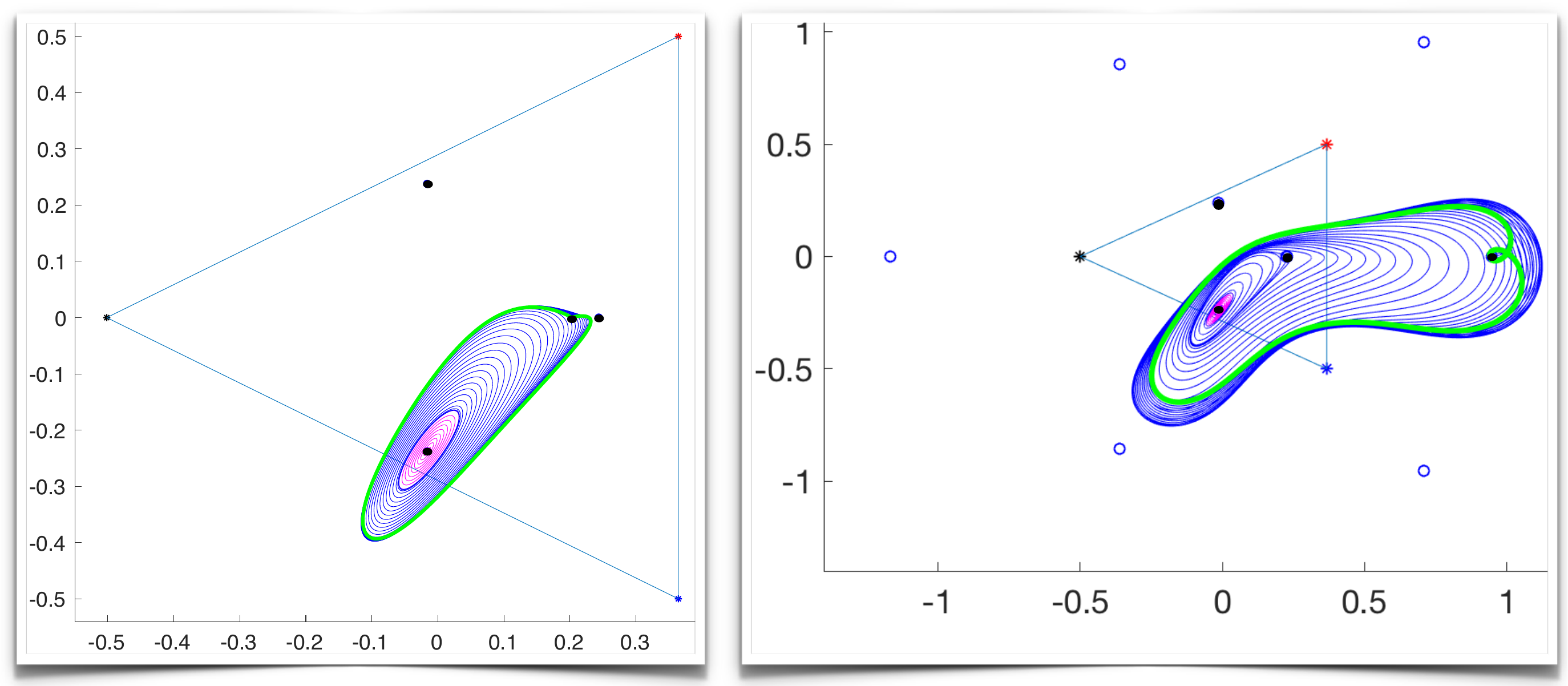} 
\caption{\textbf{Blue sky catastrophes on the $\ell_3(s)$ parameter curve:} 
The left frame illustrates the CRFBP with parameter values $\ell_{3}(0.97)$.
We study the blue sky catastrophe for the planar Lyapunov
family of $L_{1,3}(0.97)$ and reasoning just as in the caption of 
Figure \ref{fig:gamma_1_blueSkies_m1m2}
conclude that $\hat{s}_{1,3} > 0.97$.
Similarly, the right frame illustrates the planar Lyapunov family of $L_{1,3}(1)$, 
and suggests that $\hat{s}_{1,3} < 1$.
}

\label{fig:gamma_1_blueSkies_m2m3}
\end{figure}

\section{Results of center manifold/normal form calculations }\label{sec:normalform}

Instead of numerically following homoclinic orbits or Lyapunov families from the triple Copenhagen 
problem to the bifurcation curve, we can start our calculations at the bifurcation point
itself. That is, we compute the normal form $r$ of the center dynamics along the parameter 
lines $\ell_{1,2,3}(s)$. As we start at the bifurcation point instead of at the triple Copenhagen problem, 
we reverse the orientation on the curves, and consider the parameter lines 
$\hat{\ell}_{1,2,3}(s) = \ell_{1,2,3}(1-s)$ instead. For the saddle node bifurcation, 
we have the normal form
\begin{align*}
r : \mathbb{R} \times \mathbb{R}^2 \to \mathbb{R}^2, (s,x,y) \mapsto (y, s + \alpha_1 x^2),
\end{align*}
for the conjugate vector field $r$ on the center manifold. With the normal form, we can confirm 
that $\hat{s}_{2,k} = 1$ for $k=2,3$, see Figure \ref{fig:Saddle_Node}. By calculating higher 
order terms of the conjugate vector field, we can determine whether the $\gamma_3$ orbit lies 
on the center manifold. The calculated vector fields are given in Table 1. If we truncate the vector field $r$ to third order, its phase portrait does contain two fixed points, but does not contain a homoclinic orbit for small values of $s$. This suggest that if $\gamma_{3,2}$ persists until $\mathfrak{D}$, it does not lie on the local center manifold.

\begin{figure}[!h]\centering
\includegraphics[height = 0.3 \textwidth]{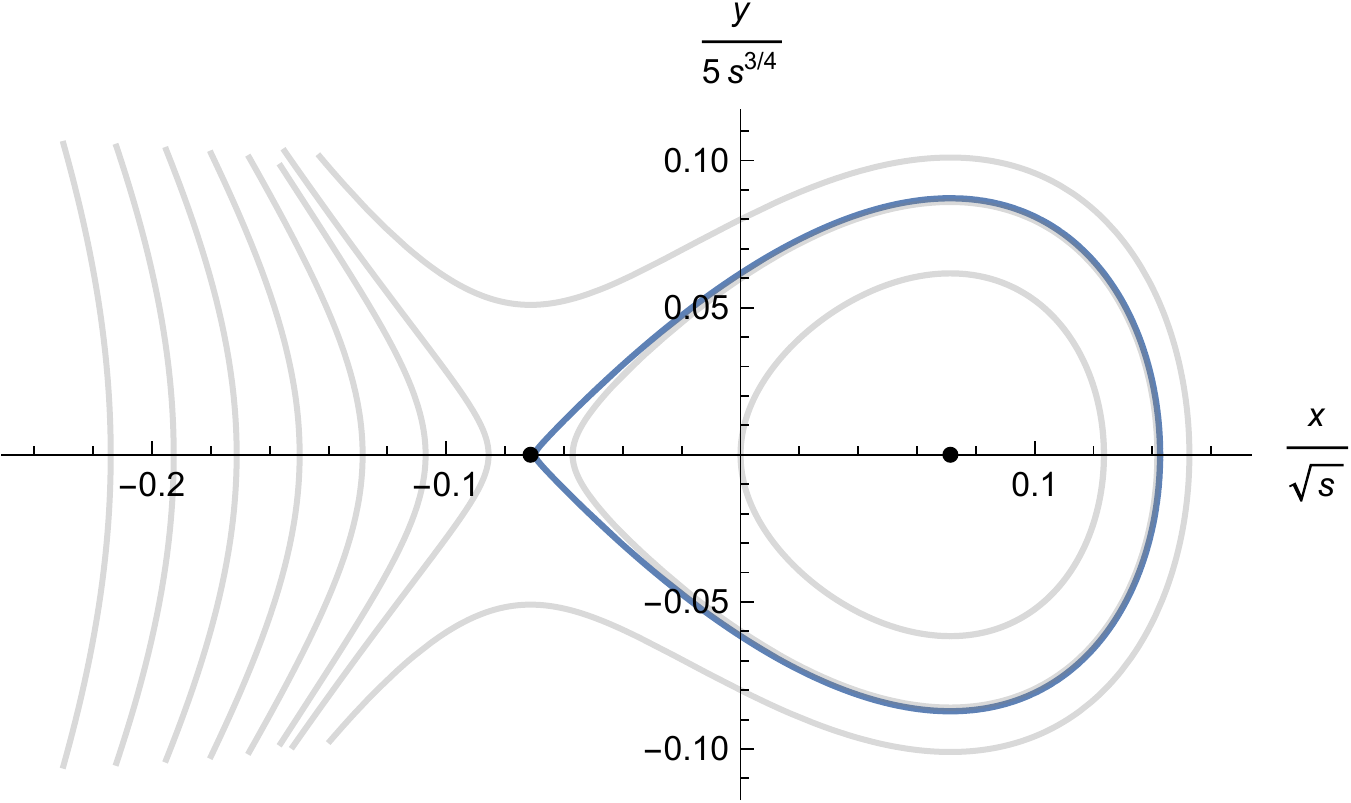}  \hfill
\includegraphics[height = 0.3 \textwidth]{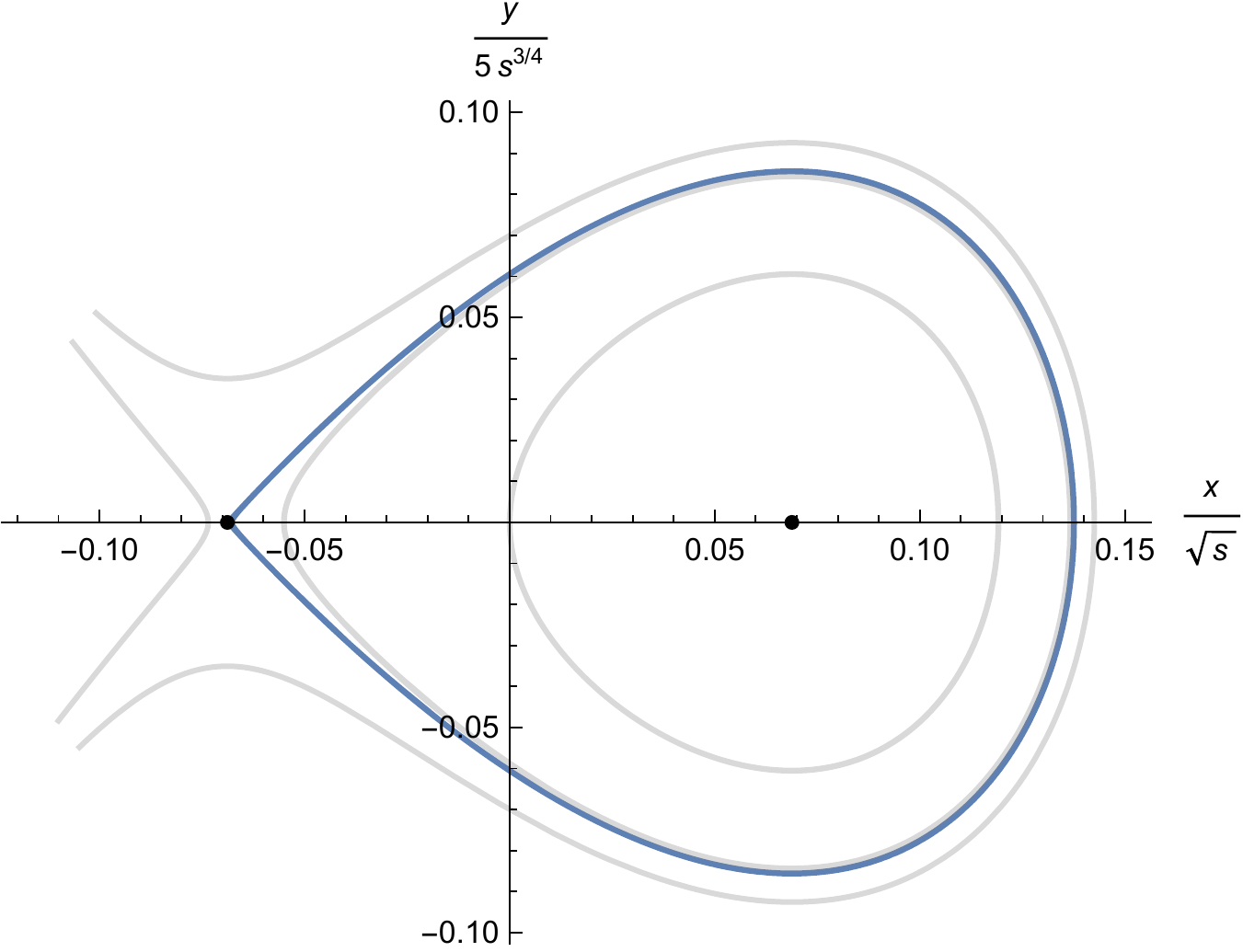}
\caption{\textbf{Small orbits near the saddle node bifurcation:} In the left figure, we plotted several orbits for the normal form at $\hat{\ell}_2(10^{-13})$, with the blue line the approximation for $\gamma_{2,2}(1-10^{-13})$. All orbits on the left side of $L_0$ decrease to $x = - \infty$. In the right figure, we plotted several orbits for the normal form at $\hat{\ell}_3(10^{-13})$, with the blue line the approximation for $\gamma_{2,3}(1-10^{-13})$. Both frames illustrate the conjugate 2 dimensional vector field
on the center manifold}
\label{fig:Saddle_Node}
\end{figure}

For $\hat{\ell}_1(s)$, the normal form of the conjugate vector field is given by
\begin{align*}
r : \mathbb{R} \times \mathbb{R}^2 \to \mathbb{R}^2, (s,x,y) \mapsto (y, \alpha_1 s x + \alpha_2 x^3),
\end{align*}
which will confirm that $\hat{s}_{2,1} =1$ and $\hat{s}_{3,1} =1$, see Figure \ref{fig:PitchFork}. 
Instead of computing the normal form of the co-dimension one  pitchfork bifurcation at $v_{c_1}$, we 
can compute the normal form of the co-dimension two bifurcation at $v_{c_1}$. Then the normal form 
would be, see for instance \cite{BuonoPolzMontaldi},
\begin{align*}
r : \mathbb{R}^2 \times \mathbb{R}^2 \to \mathbb{R}^2, (s,t,x,y) \mapsto (y , t + \alpha_1 s x + \alpha_2 x^3).
\end{align*}
With this normal form, we can recover part of $\mathfrak{D}$ near $v_{c_1}$, and show the 
persistence of $\gamma_3$ until $\mathfrak{D}$ for $v_c$ near $v_{c_1}$, see Figure \ref{fig:extra_homoclinic}. The 
calculated normal forms for both the co-dimension 1 and the co-dimension 2 bifurcation at $v_{c_1}$ 
are found in Table 1.

\begin{figure}[!h]\centering
\includegraphics[height = 0.4 \textwidth]{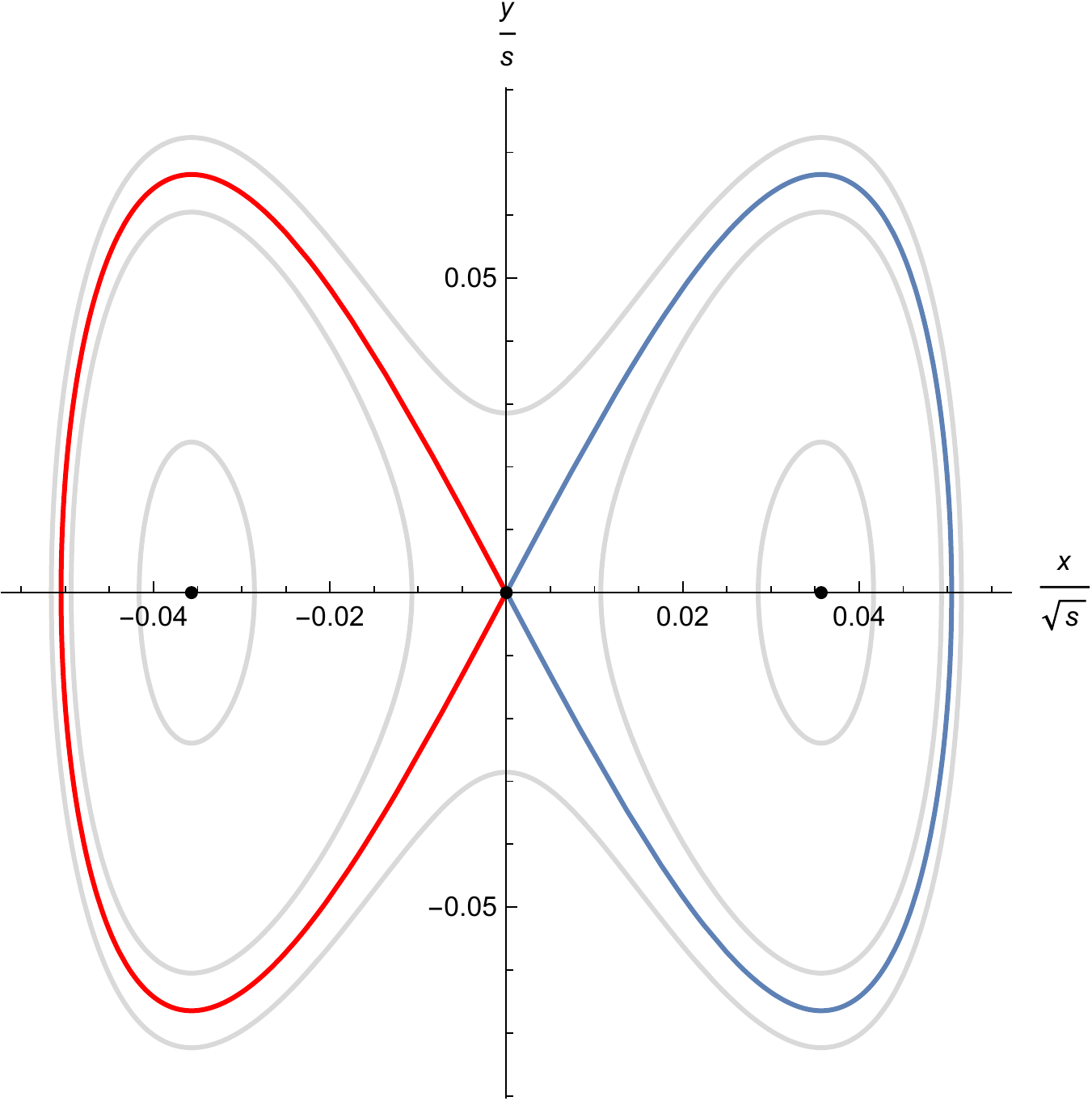} 
\caption{\textbf{Small orbits near the Pitchfork bifurcation:} Several orbits for the normal form 
at $\hat{\ell}_1(10^{-13})$. The red and blue line are approximations for $\gamma_{2,1}(1-10^{-13})$ 
and $\gamma_{3,1}(1-10^{-13})$. The frame illustrates the conjugate 2 dimensional vector field
on the center manifold.}
\label{fig:PitchFork}
\end{figure}

\begin{figure}[!h]
\centering
\includegraphics[width = 1 \textwidth]{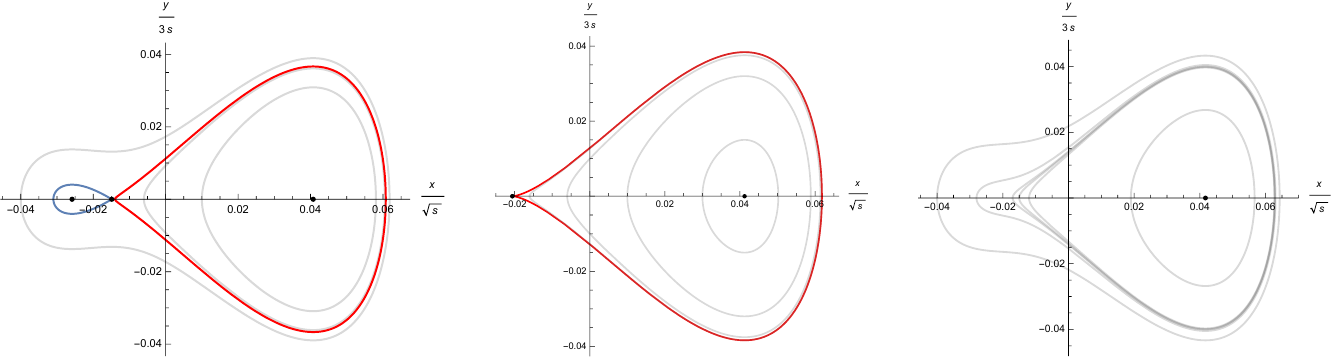}
\caption{\textbf{Persistence of $\gamma_3$:} We show the persistence of 
$\gamma_3$ until the saddle node bifurcation for $v_c$ near $v_{{\tiny \mbox{pf}}}$. 
In the left frame, we have 
$(m_1,m_2,m_3) \in  \mathfrak{S}_I$ near $v_{{\tiny \mbox{pf}}}$ but away from the 
$m_1 = m_2$ edge. We see the small homoclinic orbit $\gamma_2$ and the large 
homoclinic orbit $\gamma_3$ illustrated by blue and red curves respectively. In the middle frame 
we have $(m_1,m_2,m_3) \in \mathfrak{D}$ near $v_{{\tiny \mbox{pf}}}$. Here $L_0$ and $L_2$ have 
collided at $L_c$, yet we still see $\gamma_3$, which is illustrated by the red curve. 
In the right frame, we have $(m_1,m_2,m_3) \in  \mathfrak{S}_{II}$ near $v_{{\tiny \mbox{pf}}}$ 
but away from the $m_1 = m_2$ edge. As $L_c$ has disappeared, $\gamma_3$ has become a 
periodic orbit around $L_3$. All frames illustrate the conjugate 2 dimensional vector 
field on the center manifold.}\label{fig:extra_homoclinic}
\end{figure}

\begin{table}
\centering
\begin{tabular}[]{c | l | c | c | c }
& Masses & $\alpha_1$ & $\alpha_2$ &  Third order terms \\ & & & &  in the y-derivative  \\ \hline
\multirow{2}{*}{Pitchfork} & $v_{c_1}$ co-dimension 1 & $6.93442$ & $-5441.04$ & -- \\
& $v_{c_1}$ co-dimension 2 & $6.93442$ & $- 5441.04$ & -- \\ \hline
\multirow{2}{*}{Saddle node} 
&$v_{c_2}$ & $- 196.451$ & -- &  $20818.5 x^3$  \\
& $v_{c_3}$ & $ - 211.138$ & -- & $24498.1 x^3$
\end{tabular}\label{tab:Normal_Forms}
\caption{\textbf{Normal forms:} The calculated constants in the normal forms at $v_{c_k}$ for $k=1,2,3$.}
\end{table}

\subsection{Separation of $\mathfrak{D}$ and the $\gamma_3$ family} \label{sec:seperationOfD_and_D3}

Following our work in Sections \ref{sec:numCont} and \ref{sec:blueSkies}, we want to better understand
the robustness of the $\gamma_3$ family.  This is especially delicate off the $m_1 = m_2$ and 
$m_2 = m_3$ edges of the parameter simplex.  Symmetry considerations -- combined with the 
the local analysis of  the previous section -- show that at and near the $m_1 = m_2$ edge
the $\gamma_3$ family is completely robust,  surviving all the way to the critical 
curve $\mathfrak{D}$.  
On the other hand, symmetry considerations combined with the continuation and blue sky 
tests of Sections \ref{sec:numCont} and \ref{sec:blueSkies} suggest that 
along and near the $m_2 =  m_3$ edge, the $\gamma_3$ family 
terminates before the $\mathfrak{D}$ curve.  Taken together, this suggests that 
$\gamma_3$ must exhibit some transitional behavior, some co-dimension two 
bifurcation, along $\mathfrak{D}$.

We attempt to resolve this picture as follows.
\begin{enumerate}
\item\label{step1} We use Newton's method to find $m_2$, $x_0$ and $y_0$ for fixed $m_1$ such 
that $(m_1,m_2 ,1-m_1-m_3 ) \in \mathfrak{D}$, with corresponding libration point is $L_c = (x_0,0,y_0,0)$.,
\item\label{step2} We choose $\centerconj(x,y) = (x,y)$ and find the Taylor polynomials for the center manifold, the center stable/unstable manifolds, and the (un)stable branch on the center manifold, as shown
 in Lemmas \ref{thm: centerconjugacyformalseries} to \ref{thm: centerstableconjugacyformalseries} 
up to order $15$ using radial derivatives,
\item\label{step3} We numerically find the region where the conjugacy equation \eqref{eq: centerstableconjugacy} for the center stable/unstable manifold has an error of order $10^{-15}$ and the stable and unstable branches $\stablebranch$ and $\unstablebranch$ are numerically correct for the Taylor polynomials 
of step \ref{step2},
\item\label{step6} We take $t_s$ and $t_u$ such that $\stablebranch(t_s)$ and $\unstablebranch(t_u)$  are in the region of step \ref{step3}. We then find homoclinic orbits by numerically integrating part of the unstable 
fiber of $\unstablebranch(t_u)$ until the stable fiber of $\stablebranch(t_s)$ is reached.
We check numerically that the manifolds have a transverse intersection, indicating that 
we are at a type III global bifurcation.
\end{enumerate}
Here the function $k_c$ is a choice we have to make to find the center manifold at $v_c \in \mathfrak{D}$. We refer the reader to Appendix \ref{sec:centerCalcAppendix} for background on the center manifold, and Lemma \ref{thm: centerconjugacyformalseries} in particular to see why we have to choose $k_c$. Furthermore, the branches $\stablebranch$ and $\unstablebranch$ parameterize the (un)stable orbit on the center manifold, see Figure \ref{fig: branches4247}.

\begin{figure}[!ht]
\centering
\begin{tabular}{cc}
\includegraphics[width = 0.45 \textwidth]{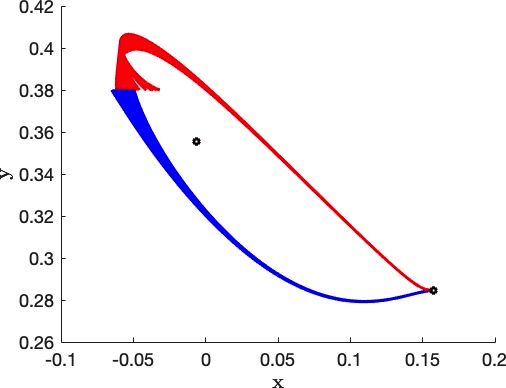} & \includegraphics[width = 0.45 \textwidth]{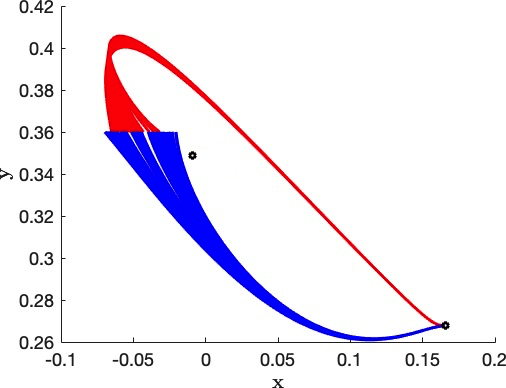} \\
$m_1 = 0.430$ & $m_1 = 0.429$ \\
\includegraphics[width = 0.45 \textwidth]{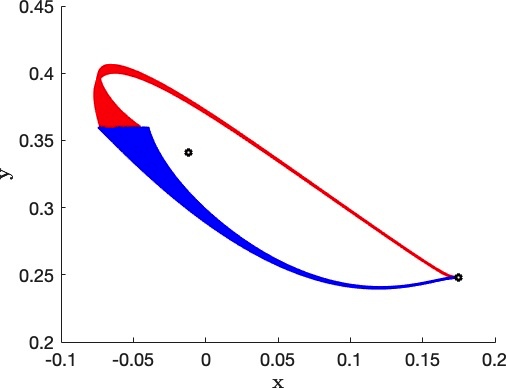} &  \includegraphics[width = 0.45 \textwidth]{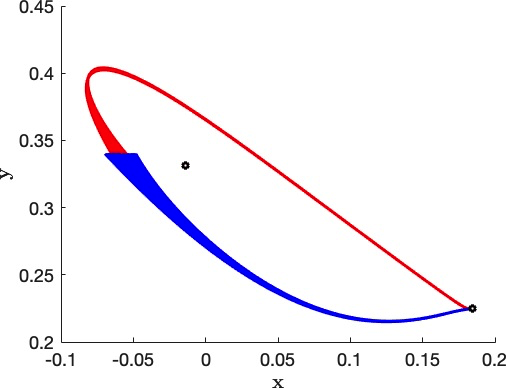} \\
$m_1 = 0.428$ & $m_1 = 0.427$ \\ 
\multicolumn{2}{c}{
\includegraphics[width = 0.45 \textwidth]{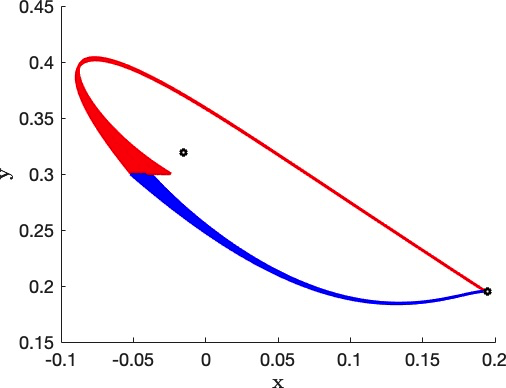} } \\
\multicolumn{2}{c}{
$m_1 = 0.426$ } 
\end{tabular}
\caption{\textbf{Short homoclinics around $L_3$ on $\mathfrak{D}$:} For different values of $m_1$ we find a short homoclinic 
orbit for $L_c$ at parameter values $(m_1,m_2 ,1-m_1-m_2) \in \mathfrak{D}$. The red surface is the backward flow of part of 
the edge of stable manifold, and the blue surface is the forward flow of part of the edge of the unstable manifold. In all 5 cases, 
there is a transverse intersection between the stable and unstable manifold on the $y = \text{constant}$ level set with $\dot{y}$ positive,
and each of the orbits has a shape suggesting they are continuations of $\gamma_3$ in the triple Copenhagen 
problem.}
\label{fig:SmallHomocliniconD}
\end{figure}

For $m_1 \in \{ 0.426 ,0.427,0.428, 0.429,0.430\}$, we compute the homoclinic for $L_c$ around $L_3$.  See Figure \ref{fig:SmallHomocliniconD}.
For all five values of $m_1$, we find that the homoclinic orbit exists, and it leaves the unstable manifold close to the center unstable branch. On the other hand, the homoclinic orbits enters the  stable manifold further away from the stable branch on the center manifold 
when we decrease $m_1$. To be more precise, for all values of $m_1$ we look where on the stable fiber of $\stablebranch(t_s)$ the homoclinic orbit enters 
the stable manifold. As $m_1$ decreases, we see that the orbit enters the stable manifold further away from $\stablebranch(t_s)$ on its stable fiber. Equivalent, the homoclinic orbit enters the stable fiber of $\stablebranch(t)$ at fixed distance from $\stablebranch(t)$ for decreasing values of $t$ as $m_1$ decreases.
We conjecture
 that the $\gamma_3$ family split from the critical curve $\mathfrak{D}$ when the homoclinic orbit around $L_3$ enters the stable manifold along the stable direction.

\begin{figure}[!ht]
\centering
\begin{tabular}{cc}
\includegraphics[width = 0.45 \textwidth]{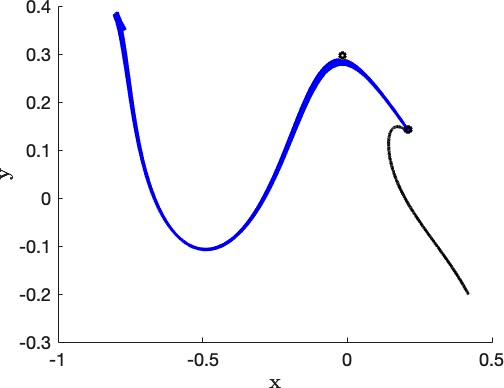} 
& \includegraphics[width = 0.45 \textwidth]{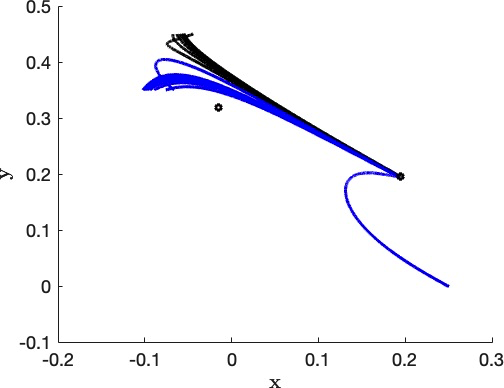} \\
\end{tabular}
\caption{\textbf{Qualitative behavior of the stable manifold:} left frame plots the backward integration of 
the edge of the stable manifold for parameter values $v_{c_2} \in \mathfrak{D}$. Right frame plots the 
backward integration of the edge of the stable manifold for the parameter values $(m_1,m_2 ,1-m_1-m_2) \in \mathfrak{D}$ for $m_1 = 0.426$.}
\label{fig:Dichotomy}
\end{figure}

In the right frame of Figure \ref{fig:Dichotomy}, we integrated points on the boundary of the stable manifold for $m_1 = 0.426$, and we see three different kinds of behavior. 
The single blue orbit in the lower half of the figure is the backward integration of the stable branch on 
the center manifold, and all other orbits lie on a stable fiber. Near the equilibrium, the orbits on the stable fibers decay exponentially fast towards the stable branch on the center manifold. In other words, in backward time the orbits on the stable fibers diverge exponentially fast away from the stable branch. This explains why the stable branch on the center manifold lies isolated from the orbits on the stable fibers. The orbits ending on the stable fibers 
exhibit two different kinds of behavior. The black orbits all come towards $L_3$ from the right, make 
a bend before $L_3$, and end up in the stable manifold of $L_c$. The blue orbits all come from the 
left of $L_3$, passing above $L_3$ before making a bend towards the stable manifold. Due to this dichotomy, in particular the blue orbits passing above $L_3$, we might 
find a homoclinic orbit around $L_3$ when $m_1 = 0.426$. In fact, we find such a homoclinic orbit  which is shown in Figure \ref{fig:SmallHomocliniconD}. In particular, in Figure \ref{fig:SmallHomocliniconD} we only plot a small region of the stable manifold in which the homoclinic orbit lies.

On the contrary, in the left figure of Figure \ref{fig:Dichotomy}, we integrated several points on the boundary of the 
stable manifold for the parameter set $v_{c_2}$. Here the black line is the backward integration of 
the stable branch of the center manifold. All other orbits end up on a stable fiber and exhibit the 
same qualitative behavior. They come towards $L_3$ from the left, but pass underneath $L_3$ and 
finally end up in the stable manifold.

Since we do not see the dichotomy nor orbits passing above $L_3$ for the parameters $v_{c_2}$ that we did see for $m_1 = 0.426$,  we qualitatively rule out that there is a short homoclinic orbit around $L_3$ for the parameter 
values $v_{c_2}$. Any homoclinic around $L_3$ will occur on longer time scales than the short homoclinics
in the $\gamma_3$ family.  This shows that $\hat{s}_{3,2} < 1$.

\subsection{Additional homoclinic at $\mathfrak{D}$} \label{sec:largeHomoclinics}
We now expand on the ideas of the previous section, to study longer homoclinic orbits
at $L_c$.
Again we consider symmetric and non-symmetric cases.  In the symmetric case 
 step 1 will be modified, as for $v_{c_3}$ we want to impose that $m_2 = m_3$, 
instead of fixing $m_1$ a priori.  For $v_{c_1}$, we have shown at the beginning of Section \ref{sec:normalform} 
that the bifurcation is a supercritical pitchfork bifurcation: thus there will be no (un)stable solution branch on 
the center manifold. In fact, we show that the constant $\mathscr{E}$ from 
Lemma \ref{thm: centerstablebranchformalseries} is zero at $v_{c_1}$. As a consequence, 
the stable and unstable manifold of $L_{{\tiny \mbox{pf}}}$ are both of dimension $1$, and we
expect no transverse intersection between them. Therefore, we only look for homoclinics for
$L_c$ at parameter values $v_{c_2}$ and $v_{c_3}$.

\subsection{The pitchfork bifurcation at $v_{c_1}$} \label{sec:pitchfork}

To show that the constant $\mathscr{E}$ from Lemma 
\ref{thm: centerstablebranchformalseries} is $0$, we find the symmetry of $\dynamics$. For notational convenience, we apply the translation 
$T(x,y) = (x - (x_1 + x_2)/2, y - (y_1 +y _2 )/2)$ to our dynamical system and hence 
the positions of the planets become
\begin{align*}
\hat{x}_1 &=  \frac{-1}{ 4\constant} ,& \hat{y}_1 &= \frac{ \sqrt{3}(1-2m_1)}{4\constant}, \\
\hat{x}_2 &= \frac{1}{4 \constant}, & \hat{y}_2 &=  \frac{-\sqrt{3}(1-2m_1)}{4 \constant} ,\\
\hat{x}_3 &= \frac{3 - 6 m_1 }{4 \constant} ,& \hat{y}_3 &= \frac{\sqrt{3}}{4 \constant}.
\end{align*}
Here $\constant = \sqrt{m_2^2 + m_2 m_3 + m_3^3 } = \sqrt{ 1 - 3m_1 + 3 m_1^2}$ when $m_1 = m_2$. 
Now we see that when $m_1 = m_2$ we have that $(\hat{x}_1,\hat{y}_1)$, and therefore also $(\hat{x}_2,\hat{y}_2)$, 
is perpendicular to $(\hat{x}_3,\hat{y}_3)$. Let $\Theta$ be the rotation matrix such that the 
positive $\hat{y}$-axis is rotated onto the 
normalized vector $(\hat{x}_3,\hat{y}_3)$.  
Then, $\Theta$ is given by the formula
\begin{align*}
\Theta (x,y) = 2  x \mathbf{\hat{x}}_2 +  \frac{2y}{\sqrt{3}}\mathbf{\hat{x}}_3
\end{align*}
Note that $\mathbf{\hat{x}}_1 = - \mathbf{\hat{x}}_2$ and that $\mathbf{\hat{x}}_2$ and 
$\mathbf{\hat{x}}_3$ are perpendicular.  Then, the term $\kap \mathbf{\hat{x}}_3$ in the first norm is due to the fact  that the translation $T$ can be written as $T(\boldx) = \boldx + (1- 2 m_1 ) \mathbf{\hat{x}}_3$,
\begin{align*}
\Omega( T^{-1}(\Theta (-s, t) ))  &= \frac{1}{2} \| -2 s \mathbf{\hat{x}}_2 + \frac{2t}{\sqrt{3}} \mathbf{\hat{x}}_3  \kap \mathbf{\hat{x}}_3 \|_2^2+ \sum_{i=1}^3 \frac{m_i}{\| - 2 s \mathbf{\hat{x}}_2 + \frac{2t}{\sqrt{3}} \mathbf{\hat{x}}_3 - \mathbf{\hat{x}}_i\|_2} \\
 			&= \frac{1}{2} \| 2 s \mathbf{\hat{x}}_2 + \frac{2t}{\sqrt{3}} \mathbf{\hat{x}}_3 \kap \mathbf{\hat{x}}_3 \|_2^2 + \frac{m_3}{\| -2 s \mathbf{\hat{x}}_2 + \frac{2t}{\sqrt{3}} \mathbf{\hat{x}}_3 - \mathbf{\hat{x}}_3\|_2} \\
			& \quad \quad + \frac{m_1}{\| - 2 s \mathbf{\hat{x}}_2 + \frac{2t}{\sqrt{3}} \mathbf{\hat{x}}_3 - \mathbf{\hat{x}}_1\|_2}  + \frac{m_2}{\| - 2 s \mathbf{\hat{x}}_2 + \frac{2t}{\sqrt{3}} \mathbf{\hat{x}}_3 - \mathbf{\hat{x}}_2\|_2}  \\
			&= \frac{1}{2} \| 2 s \mathbf{\hat{x}}_2 + \frac{2t}{\sqrt{3}} \mathbf{\hat{x}}_3  \kap \mathbf{\hat{x}}_3 \|_2^2 + \frac{m_3}{\| 2 s \mathbf{\hat{x}}_2 + \frac{2t}{\sqrt{3}} \mathbf{\hat{x}}_3 - \mathbf{\hat{x}}_3\|_2} \\
			& \quad \quad + \frac{m_2}{\|  2 s \mathbf{\hat{x}}_2 + \frac{2t}{\sqrt{3}} \mathbf{\hat{x}}_3 - \mathbf{\hat{x}}_2\|_2}  + \frac{m_1}{\|  2 s \mathbf{\hat{x}}_2 + \frac{2t}{\sqrt{3}} \mathbf{\hat{x}}_3 -\mathbf{\hat{x}}_1\|_2}  \\
			&= \Omega(T^{-1}(\Theta(s,t)))
\end{align*}
We define $\Phi(s,t) \bydef \Omega(T^{-1}(\Theta(s,t)))$, and our bifurcation point lies in translated coordinates on the line $\mathbf{\hat{x}}_3$, 
i.e.\ $\Omega(x_0,y_0) = \Phi(0,t)$ for some $t \in \mathbb{R}$. From the symmetry $\Phi(-s,t) = \Phi(s,t)$ we obtain
\begin{align*}
\Phi_{s}(0,t) &= - \Phi_{s}(0,t) = 0,  \\
\Phi_{st}(0,t) &= - \Phi_{st}(0,t) = 0, \\
\Phi_{sss}(0,t) &= - \Phi_{sss}(0,t) = 0.
\end{align*}
Let $\lambda = \sqrt{3} (1 - 2 m_1) = \sqrt{3} m_3$.
Using the chain rule we have
\begin{align}
2M \Phi_{s} (0,t) &= \Omega_x - \lambda \Omega_y = 0, \label{eq:rightendpointderivativeszero1}\\
4 M^2 \Phi_{st} (0,t) &= \lambda \Omega_{xx} + (1 - \lambda^2) \Omega_{xy} - \lambda \Omega_{yy} = 0, \label{eq:rightendpointderivativeszero2} \\
8M^3 \Phi_{sss} (0,t) &= \Omega_{xxx} - 3 \lambda \Omega_{xxy} + 3 \lambda^2 \Omega_{xyy} - \lambda^3 \Omega_{yyy} = 0. \label{eq:rightendpointderivativeszero3}
\end{align}
Furthermore, since we are at a bifurcation point, we also have $\Omega_{xx} \Omega_{yy} = \Omega_{xy}^2$, hence Equation \eqref{eq:rightendpointderivativeszero2} becomes
\begin{align*}
\lambda \Omega_{xx}^2 + ( 1 - \lambda^2) \Omega_{xy}\Omega_{xx} - \lambda \Omega_{xy}^2 = 0.
\end{align*}
Therefore, we either have $\Omega_{xx} = \lambda \Omega_{xy}$, thus also 
$\Omega_{xy} = \lambda \Omega_{yy}$, or $\Omega_{xx} = -\lambda^{-1} \Omega_{xy}$, and also 
$\Omega_{xy} = -\lambda^{-1} \Omega_{yy}$. To establish that we have $\Omega_{xx} = \lambda \Omega_{xy}$ instead of $\Omega_{xx} = -\lambda^{-1} \Omega_{xy}$, we use Newton's method to find the root of 
\begin{align*}
(\Omega_x( \Theta(0,t)) ,  \Omega_{xx}( \Theta(0,t))  - \lambda \Omega_{xy}( \Theta(0,t))  ),
\end{align*}
which defines the bifurcation parameters $v_{c_1}$ -- 
and libration point $L_{c_1} = (\Theta(0,t)_1,0,\Theta(0,t)_2,0)$. To see this, from Equation \eqref{eq:rightendpointderivativeszero1}
 it follows that $\Omega_y = 0$, i.e.\ $\Theta((0,t))$ is indeed a fixed point. Furthermore, from Equation \eqref{eq:rightendpointderivativeszero2} it also follows that $\Omega_{xy} = \lambda \Omega_{yy}$, 
 and hence $\Omega_{xx} \Omega_{yy} = \Omega_{xy}^2$, thus $\Theta((0,t))$ is not only a fixed point, but the 
 linearization of the vector field at $\Theta((0,t))$ has a double eigenvalue $0$.

Exploiting again that
$\Omega_{xx} = \lambda \Omega_{xy} = \lambda^2 \Omega_{yy}$, 
the constant $\mathscr{E}$ becomes 
\begin{align*}
\mathscr{E} &= \Omega_{xxx} \Omega_{yy}^3 - 3 \Omega_{xxy} \Omega_{yy}^2 \Omega_{xy} + 3 \Omega_{xyy} \Omega_{yy} \Omega_{xy}^2 - \Omega_{yyy} \Omega_{xy}^3 \\
			&= \left( \Omega_{xxx}  - 3 \lambda \Omega_{xxy}  + 3 \lambda^{2} \Omega_{xyy}  - \lambda^{3} \Omega_{yyy} \right) \Omega_{yy}^3 \\
			&= \frac{\Omega_{yy}^3}{8 M^3} \Phi_{sss}(0,t) \\
			&= 0.
\end{align*}
Hence, Lemma \ref{thm: centerstablebranchformalseries} cannot be applied to find 
(un)stable solution branches on the center manifold, which further supports the claim 
that $L_c$ undergoes a pitchfork bifurcation at $v_{c_1}$.

\subsection{Generic saddle node bifurcation at $v_{c_2}$} \label{sec:near_v2}

In our numerical scheme, we use Newton's method to find a root of
\begin{align*}
(\Omega_x, \Omega_y , \Omega_{xx} \Omega_{yy} - \Omega_{xy}^2),
\end{align*}
where we fix $m_1 = 0.4247$, and consider $m_2$ as the only parameter.  This results in the bifurcation parameters $v_{c_2}$. In Figure \ref{fig: branches4247} we plot the Taylor polynomials of the stable 
and unstable branch, together with two orbits starting close to the stable branch. This allows us to 
check the branches and take $t_s$ and $t_u$ in step \ref{step6} as large as possible. The results 
of the numerical integration done in step \ref{step6} is illustrated in Figure \ref{fig: homoclinic4247} .

\begin{figure}[!h]
\centering
\includegraphics[height = 0.35  \textheight]{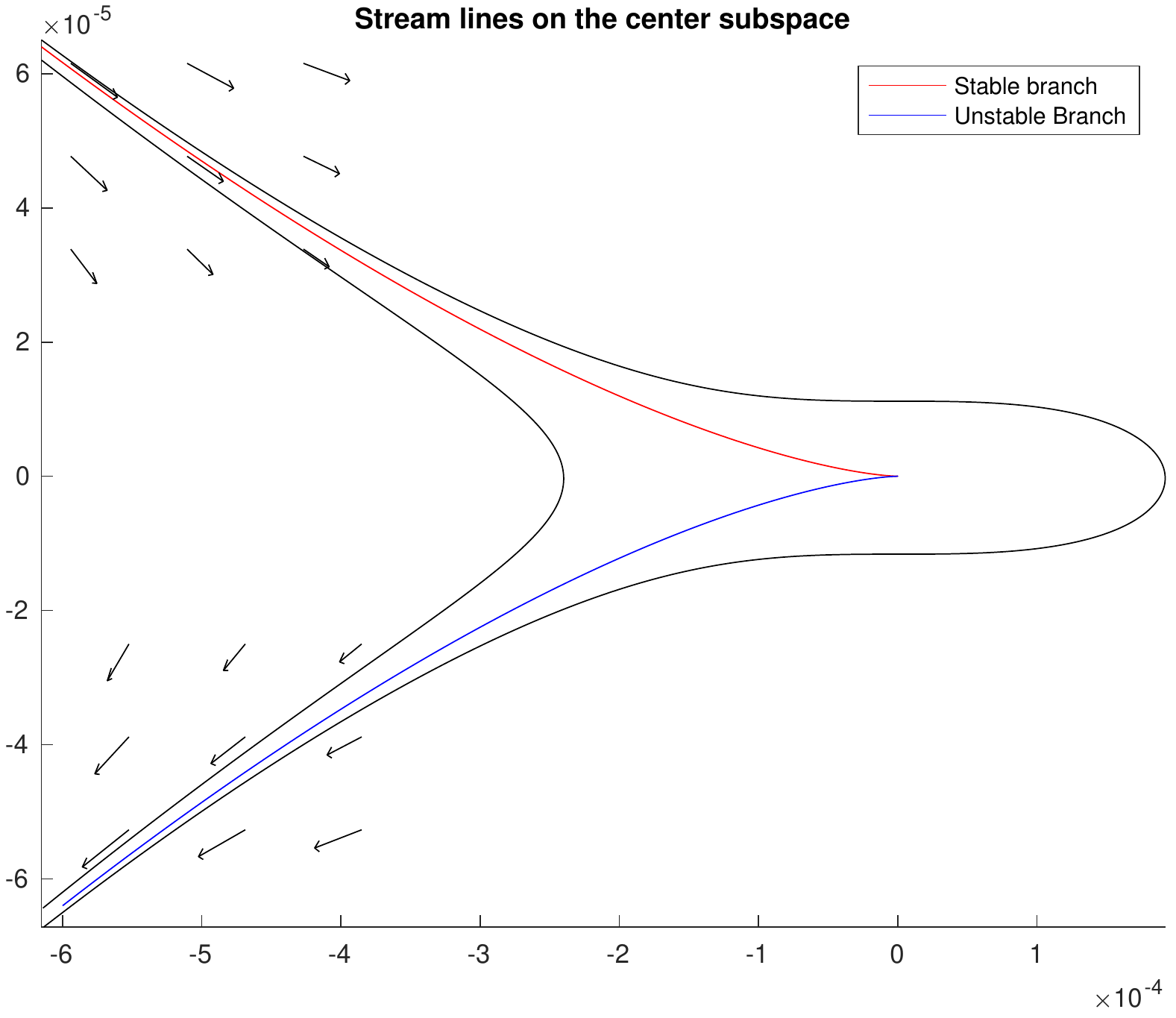} 
\caption{The approximation of the solution branches on the center manifold for the
parameters $v_{c_2}$ on the bifurcation curve.}
\label{fig: branches4247}
\end{figure}

\begin{figure}[!t]
\begin{tabular}{cc}
\includegraphics[width = 0.45  \textwidth]{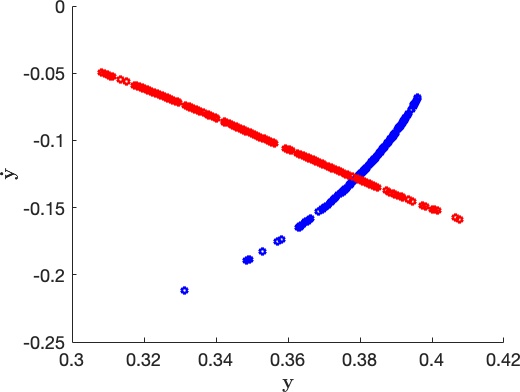}  & 
\includegraphics[width = 0.45  \textwidth]{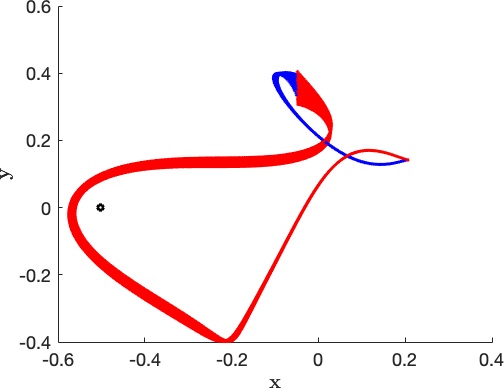}  \\
\includegraphics[width = 0.45  \textwidth]{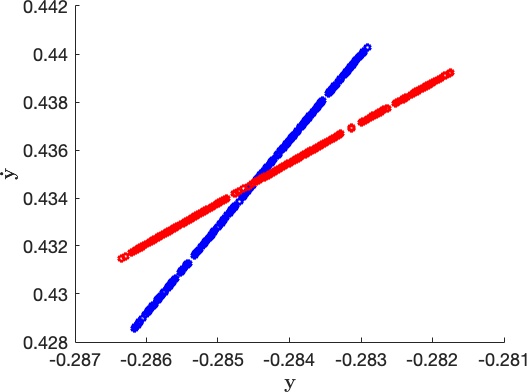}  & 
\includegraphics[width = 0.45  \textwidth]{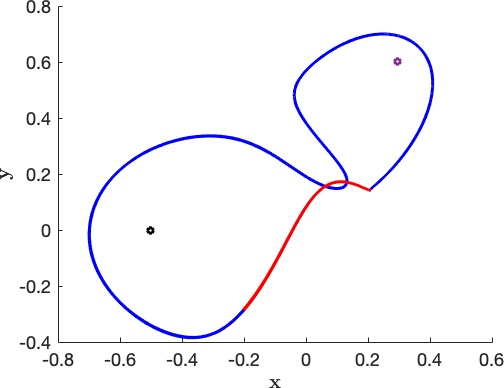}  \\
\end{tabular}
\caption{\textbf{Numerical integration of the stable and unstable manifold at $v_{c_2}$:} We found 
two ''large'' homoclinic orbits of $L_c$ by numerically integrating the (un)stable manifold. The red 
surface is the backward flow of part of the edge of stable manifold, and the blue surface is the 
forward flow of part of the edge of the unstable manifold. In the top left figure, we plot the 
intersection of the forward and backward flow and the line $x = -0.05$ with positive $x$-derivative. 
As there is a transverse connection in the $(y,\dot{y})$-plane, we find the enclosure of a 
homoclinic orbit, which is the top right figure.
In the bottom left figure, we plot the intersection of the forward and backward flow and the 
line $x = -0.2$ with positive $x$-derivative. As there is a transverse connection in the 
$(y,\dot{y})$-plane, we find the enclosure of a homoclinic orbit, which is the bottom right figure. }
\label{fig: homoclinic4247}
\end{figure}

\subsection{The non-generic saddle node bifurcation at $v_{c_3}$} \label{sec:nonGeneric}

To obtain the bifurcation parameters $v_{c_3}$, we will exploit that the planar circular restricted four body problem has symmetry on the edge $m_2 = m_3$. We have $\Omega(x,y) = \Omega(x,-y)$, and thus also $\Omega_x(x,-y) = \Omega_x(x,y)$, $\Omega_y(x,-y) = -\Omega_y(x,y)$, and $\Omega_{xy}(x,-y) = -\Omega_{xy}(x,y)$. In particular, we have $\Omega_y(x,0) = 0$ and $\Omega_{xy}(x,0) = 0$. We can therefore in step \ref{step1} use Newton' method to find the root of 
\begin{align*}
(\Omega_x(x,0), \Omega_{xx}(x,0))
\end{align*}
along $\ell_3(s)$ to obtain the bifurcation parameters $v_{c_3}$ and bifurcation point $L_c  = (x_0,0,0,0)$.

The symmetry also motivates us to define the linear map 
$A :(x,\dot{x},y,\dot{y}) \mapsto (x , - \dot{x}, - y, \dot{y})$, 
yielding the symmetry $\dynamics(A\boldx) = - A \dynamics(\boldx)$. 
So, any orbit that starts on the plane $(x,0,0,\dot{y})$ has its backward orbit given by $x(-t) = A x(t)$. 
In other words, an orbit that connects the unstable boundary and the plane $(x,0,0,\dot{y})$ is a symmetric
homoclinic orbit for the critical equilibrium: we do not have to integrate the stable boundary backwards. 
We show that the unstable manifold has an intersection with $(x,0,0,\dot{y})$, and we obtain the 
homoclinic orbits in Figure \ref{fig: homoclinic4234} by flipping the orbits along the symmetry axis. 

\begin{figure}[!t]
\begin{tabular}{cc}
\includegraphics[width = 0.47 \textwidth]{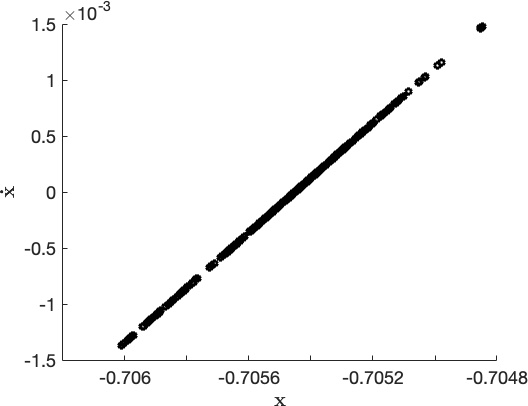} & 
\includegraphics[width = 0.47  \textwidth]{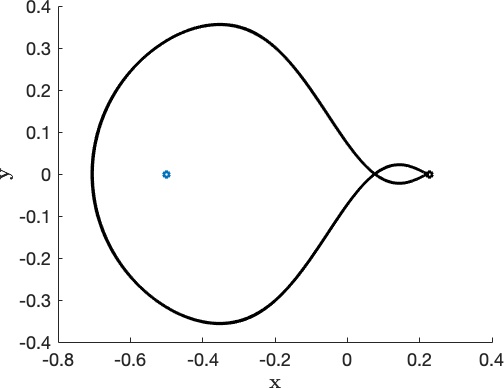} \\
\includegraphics[width = 0.47  \textwidth]{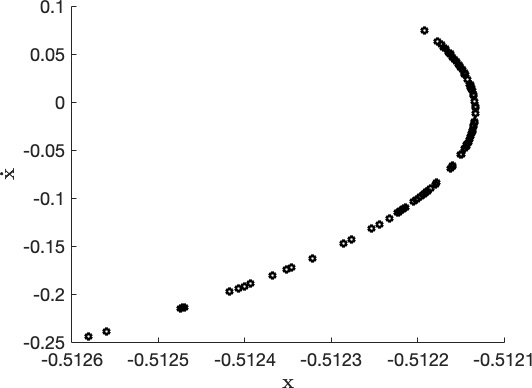} & 
\includegraphics[width = 0.47  \textwidth]{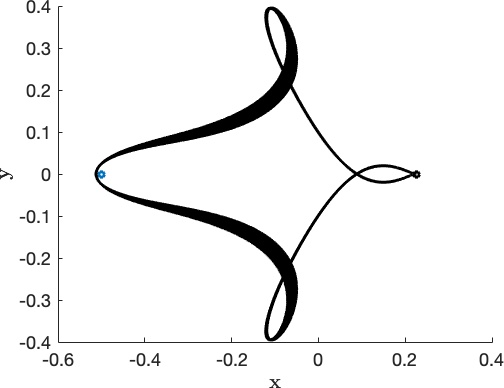} \\
\includegraphics[width = 0.47  \textwidth]{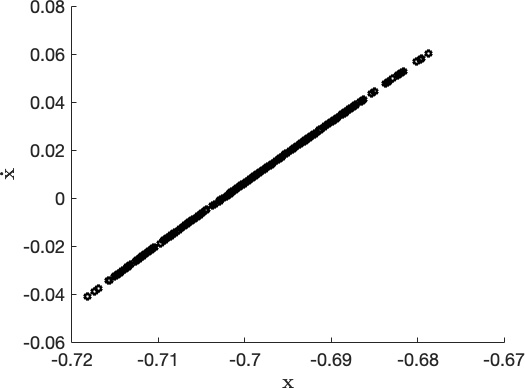} & 
\includegraphics[width = 0.47  \textwidth]{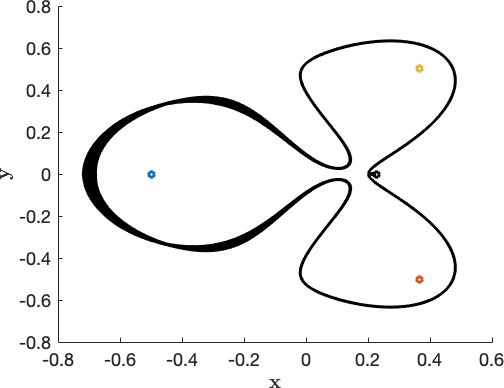} \\
\end{tabular}
\caption{\textbf{Numerical integration of the stable and unstable manifold at $v_{c_3}$:} We found three ''large'' homoclinic orbits of $L_c$ by numerically integrating the unstable manifold. In the left figures, we plot the intersection of forward flow of part of the edge of the unstable manifold and the line $y =0$.  In all three figures we see that in the $(x,\dot{x})$-plane we have a transverse intersection with the line $\dot{x} = 0$. Hence due to symmetry we obtained a region in which a homoclinic orbit lies. In the right figures we plotted the corresponding region in which a homoclinic orbit lies.}
\label{fig: homoclinic4234}
\end{figure}

These results suggest a rich transverse homoclinic orbit 
structure at $\mathfrak{D}$.  Not all of the homoclinic orbits found at 
the critical curve have shapes reminiscent of the basic triple Copenhagen 
homoclinics $\gamma_{1,2,3,4,5,6}$.  For example, the middle right frame of 
Figure \ref{fig: homoclinic4234} illustrates a homoclinic orbit with no 
apparent analogue in the triple Copenhagen problem.


\section{Conclusions} \label{sec:conclusions}
Informed by the numerical explorations discussed in the main text of the paper,
we propose the following conjectures concerning the global dynamics of the 
CRFBP.  First, and based on the observation in Remark \ref{rem:belDevBif}, 
we conjecture that each of the short triple Copenhagen homoclinic orbits 
undergoes a type II bifurcation.

\begin{conjecture}
Any parameter continuation of the triple Copenhagen 
homoclinic orbits $\gamma_{1,2,3}$ toward the critical curve $\mathfrak{D}$
results in a Belyakov-Devaney bifurcation before the termination of the family.
\end{conjecture}

Recall that when the $\gamma_1$ family of homoclinic orbits is continued 
along one of the three parameter lines $\ell_{1,2,3}(s)$, the 
numerical continuation breaks down before the critical curve.  
The $\gamma_1$ family exhibits the least robustness with respect to 
parameter continuation along these three parameter lines
and we conjecture that this 
behavior is general.  

\begin{conjecture}
The $\gamma_1$ family \textit{never} continues to the critical curve $\mathfrak{D}$.
\end{conjecture}

The $\gamma_2$ family on the other hand exhibits the most robustness, and we find that
we are able to continue along each of the three parameter lines almost until 
the homoclinic orbits shrink to points.  Moreover, we confirmed by considering the 
normal form that when $L_2$ is close to $L_0$ 
there is a short homoclinic orbit for $L_0$ which winds around $L_2$.
We conjecture that this is the general picture.

\begin{conjecture}
The $\gamma_2$ family \textit{always} continues to the critical curve $\mathfrak{D}$, 
where it vanishes with $L_0$ and $L_2$ in the saddle node bifurcation.
\end{conjecture}

The behavior of the $\gamma_3$ family is the most complicated.  
For parameters on the $m_2= m_3$ edge of the parameter simplex the problem 
has symmetry about the $x$-axis.  Then the behavior of $\gamma_3$ on the 
$\ell_3(s)$ line mirrors the 
behavior of $\gamma_1$, which breaks down before $\mathfrak{D}$.  
Similarly, for parameters
on the $m_1 = m_2$ edge of the parameter simplex the problem has symmetry about the line 
through the third primary bisecting the edge of the triangle connecting the first to 
the second primary.  In this case the 
behavior of $\gamma_3$ on the $\ell_1(s)$ line mirrors the behavior 
$\gamma_2$, which continues all the way to $\mathfrak{D}$ where it 
disappears.  Moreover, the normal form calculation suggests that there is 
a ``short'' homoclinic around $L_3$ for parameter values near the pitch-fork
bifurcation, and our numerics confirm this, with a shape suggestive of the $\gamma_3$ family.
We conjecture that the discussion above tells the full story.

\begin{conjecture}
The $\gamma_3$ family of homoclinics continues to the critical curve $\mathfrak{D}$
when $m_2 \approx m_3$ and does not reach $\mathfrak{D}$ when $m_1 \approx m_2$.
There a single parameter value on $\mathfrak{D}$ separating these behaviors.
\end{conjecture}

Indeed we propose a little more.  
Let $\mathfrak{D}'$ denote the curve in parameter space where the stability of 
$L_0$ changes from bi-focus to a saddle with real distinct eigenvalues.  
Let $\mathfrak{D}_1$ denote the set of points in $\mathfrak{S}$ where the 
$\gamma_1$ family loses transversality and breaks down in a type IV bifurcation.  
We claim that $\mathfrak{D}_1$ is an analytic simple closed curve from 
the $m_1 = m_2$ edge to the $m_2 = m_3$ edge which does not intersect
$\mathfrak{D}$.  Define $\mathfrak{D}_{2}$ and $\mathfrak{D}_3$ analogously.
We claim that $\mathfrak{D}_2 = \mathfrak{D}$ and that $\mathfrak{D}_3$
coincides with $\mathfrak{D}$ near the $m_1 = m_2$ edge, separates 
at a single parameter value where there is a type V bifurcation,
and then intersects with $\mathfrak{D}_1$ only at the $m_2 = m_3$ parameter edge. 
We conjecture that the separation of $\mathfrak{D}$ and $\mathfrak{D}_3$ occurs at 
$v_c = (m_1,m_2,m_3) \in \mathfrak{D}$ with $m_1 \in (0.425,0.426)$. We also claim 
that $\mathfrak{D}'$ lies entirely to the left of $\mathfrak{D}_1$, 
so that each of the $\gamma_{1,2,3}$ families undergoes a Belyakov-Devaney 
bifurcation on $\mathfrak{D}'$.  A graphical depiction of the 
conjectures is given in Figure \ref{fig:parameterSimplexFinal}.

\begin{figure}[!t]
\centering
\includegraphics[width=4.5in]{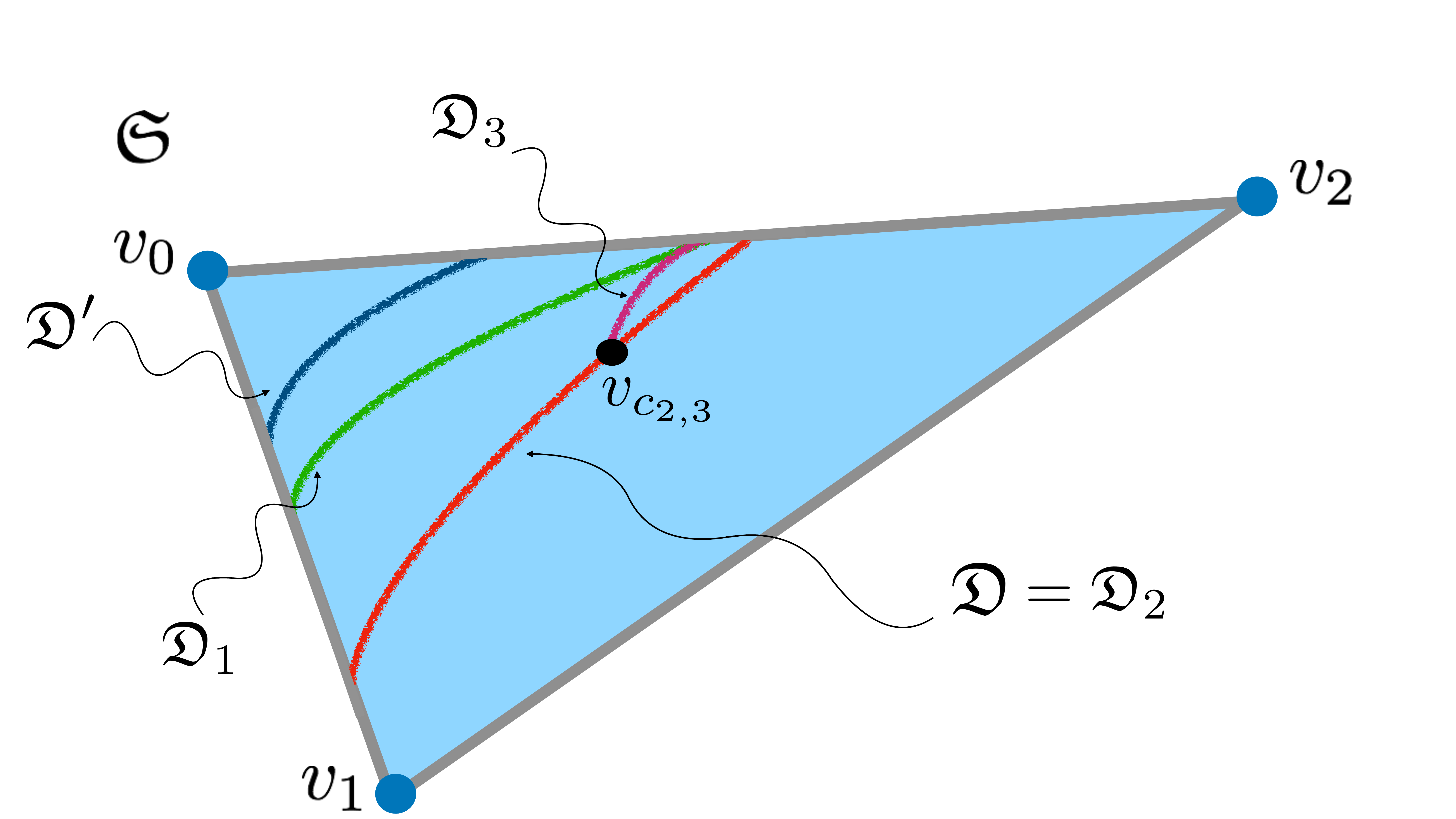}
\caption{\textbf{Conjectured critical curves for the $\gamma_{1,2,3}$ families:} 
conjectured critical lines for the $\gamma_{1,2,3}$ families of homoclinic orbits.
}\label{fig:parameterSimplexFinal}
\end{figure}

Resolving these conjectures would seem to require substantial additional work.  
Moreover, since the mathematically rigorous characterization of the equilibrium set 
given in \cite{MR2232439,MR2784870,MR3176322} required the use of computer assisted
methods of proof, it seems likely resolution of the conjectures would require similar methods.
We remark that one parameter families of connecting orbits have been studied using 
computer assisted methods of proof, see for example \cite{maximeBridge}.  
Computer assisted proofs of bifurcations of connecting orbits for maps were studied 
in \cite{MR3927299}, and there is reason to believe that these techniques 
could be extended to differential equations.  We further mention that 
computer assisted methods of proof have been devised for studying 
center manifolds \cite{capinski_roldan} in celestial mechanics problems.
Finally note that validated numerical methods for infinite dimensional multi-parameter
continuation problems have been developed \cite{MR3464215}, so that many of the techniques
needed for establishing the conjectures above exist.  Combining and extending
existing methods to resolve the proposed conjectures would represent a 
substantial leap in the state-of-the art of computer assisted methods of proof 
for global analysis of nonlinear systems.

There exists a substantial literature on
 numerical methods for studying degenerate connecting orbits.
We refer for example to  works of 
\cite{MR2020992,MR1456497,MR1205453,MR1283943,MR1404122,MR1900655,MR1314079,MR1830909}.
Since these works are based on implicit function theory/Newton's method applied to appropriate
systems of constraints, it is reasonable to suppose that such methods are amenable to the kind of 
a-posteriori analysis that underlies the papers discussed in the previous paragraph exploit.  In other words,
 the theoretical framework needed for framing computer assisted proofs of the critical and degenerate
 connecting orbits already exists.  
 
 Another concrete open question is suggested by the $\gamma_1$ family, which appears to
 always terminates in a type $IV$ global bifurcation.  If this is indeed the case then the 
 $\gamma_1$ family must collide with another homoclinic family and disappear
 at $\mathfrak{D}_1$, begging the question: \textit{what is the other homoclinic family 
 participating in this bifurcation?}  Or, if our guess is mistaken and $\gamma_1$ does not terminate 
 in a type $IV$ bifurcation, so that there is no second family of homoclinics terminating at $\mathfrak{D}_1$, 
 then what is happening there?   Put another way, if $\mathfrak{D}_1$ is not a type $IV$ 
 bifurcation curve then what is it? We have not yet begun to 
 explore this very natural question.

Before concluding, we remark that another interesting project would be to generalize the current 
study to the $\gamma_{4,5,6}$ families. We did not undertake this work for two reasons.
First our work on the $\gamma_{1,2,3}$ family already requires a substantial number of pages to 
describe.  But more importantly, the $\gamma_{4,5,6}$ families do not participate in blue 
sky catastrophes with any Lyapunov families of periodic orbits.  Rather, they are related to certain periodic orbits
around the primaries coming from the rotating Kepler problem. This complicates the use of the blue sky tests 
conducted in Section \ref{sec:blueSkies}.  Nevertheless, a follow-up study on the robustness
of the $\gamma_{4,5,6}$ families could be an interesting project.

\section{Acknowledgments}
The Authors would like to thank J.B. van den Berg and Bob Rink for many 
helpful discussions and much encouragement as this work progressed.   
The authors owe sincere thanks to two anonymous referees who carefully 
read the submitted version of the manuscript, and made a number of 
insightful comments and corrections which improved the final printed version 
of the paper. W. Hetbrij was partially supported by NWO-VICI grant 639033109. 
J.D. Mireles James was partially supported by NSF grant DMS-1813501.

\appendix

\section{Numerical continuation of periodic and connecting orbits in conservative systems}  \label{sec:appendixContinuation}
Before discussing continuation schemes for two point boundary value problems
we first introduce a little notation, and review some basic concepts from 
the elementary theory of differential equations.  A comprehensive reference for the 
material discussed in this section is the book of \cite{MR2224508}.

Suppose that for each $s \in [a, b] \subset \mathbb{R}$ the set $\Omega_s$ is an open subset 
of $\mathbb{R}^{d}$, and let $f_s \colon \Omega_s \times [a,b] \to \mathbb{R}^{d}$ be a 
one parameter family of smooth vector fields, depending smoothly on the parameter $s$. 
Suppose further that for each $s \in [a, b]$ there is a smooth function $E_s \colon \Omega_s \to \mathbb{R}$ having that 
\[
\nabla E_s(\mathbf{x}) \cdot f_s(\mathbf{x}) = 0,
\]  
for all $\mathbf{x} \in \Omega_s$.  It follows that the function $E_s$ is a conserved quantity for the 
vector field $f_s$, in the sense that if $\gamma \colon [0, T] \to \Omega_s$ is a solution of the differential equation 
\[
\gamma'(t) = f_s(\gamma(t)), 
\]
then 
\[
E_s(\gamma(t)) = E_s(\gamma(0)),
\]
for all $t \in [0, T]$.  To see this, simply differentiate to obtain 
\begin{align*}
\frac{d}{dt} E_s(\gamma(t)) & = \nabla E_s(\gamma(t)) \cdot \gamma'(t) \\
& = \nabla E_s(\gamma(t)) \cdot f(\gamma(t)) \\
& = 0, 
\end{align*}
as $\gamma(t) \in \Omega_s$ for $ t \in [0, T]$,  and $E_s$ is constant on solution curves
as desired.

Next we introduce the notion of the flow map generated by $f_s$.
Indeed, suppose that $x_0 \in \Omega_s$.  Since $f_s$ is differentiable (and hence locally Lipschitz) there 
exists a maximal $-\infty \leq \omega_{x_0} < 0 < \tau_{x_0} \leq \infty$ so that 
the solution of the initial value problem 
\[
x'(t) = f_s(x(t)), \quad \quad \quad x(0) = x_0
\]
exists and is unique for all $t \in [\omega_{x_0}, \tau_{x_0}]$.  We write 
\[
x(t) = \phi(x_0, t, s),  
\]
for $t \in [\omega_{x_0}, \tau_{x_0}]$ to denote the local flow map.
The map has the property that 
\[
\phi(x_0, t_1 + t_2) = \phi(\phi(x_0, t_1), t_2),
\]
for all $t_1, t_2 \in [\omega_{x_0}, \tau_{x_0}]$ which have
that $t_1 + t_2 \in [\omega_{x_0}, \tau_{x_0}]$.

It is a classic theorem of differential equations that $\phi$ is a 
smooth function of $s \in (a,b)$, and of $x_0 \in \Omega_s$ and of 
$t \in (\omega_{x_0}, \tau_{x_0})$, and we would like to 
compute the partial derivatives.
Since $x(t) = \phi(x_0, t, s)$ is a solution curve for the differential equation,
we have that 
\begin{equation} \label{eq:timePartialFlow}
\frac{\partial}{\partial t} \phi(x_0, t, s) = f(\phi(x_0, t, s)), \quad \quad \quad t \in (\omega_{x_0}, \tau_{x_0}).
\end{equation}
The derivatives with respect to $x_0$ and $s$ are more involved, but satisfy
the so called \textit{variational equations}.  
More precisely, let 
\[
M(t) = D_x \phi(x_0, t, s).
\]
Then $M(t)$ satisfies the non-autonomous, homogeneous, linear matrix differential equation 
\[
M'(t) = D_x f_s(\phi(x_0, t, s)) M(t), \quad \quad \quad M(0) = \mbox{Id},
\]
the equation of first variation with respect to initial conditions.  
Similarly, let 
\[
\eta(t) = \frac{\partial}{\partial s} \phi(x_0, t, s).
\]
Then $\eta(t)$ satisfies the non-autonomous, 
inhomogeneous, linear differential equation 
\[
\eta'(t) = D_x f_s(\phi(x_0, t, s)) \eta(t) + \frac{\partial}{\partial s} f_s(\phi(x_0, t, s)) 
\quad \quad \quad \eta(0) = 0,
\]
the equation of first variation with respect to the parameter $s$.
Of course if $f$ has other parameters, then partial derivatives of $\phi$ with 
respect to these parameters are found the same way.

Suppose now that $s \in [a,b]$ and $x_0 \in \Omega_s$ are fixed and that 
$0 < \tau \leq \tau_{x_0}$ is a fixed ``step size''.  In practice the quantities 
$\phi(x_0, \tau, s)$, $M(\tau)$, and $\eta(\tau)$
are computed simultaneously by numerically integrating the system of equations 
\begin{equation}\label{eq:varEqs}
\begin{array}{ll}
x'(t) & = f_s(x(t)) \\
M'(t) & = D_x f_s(x(t)) M(t) \\
\eta'(t) & = D_x f_s(x(t)) \eta(t) + \frac{\partial}{\partial s} f_s(x(t))
\end{array}
\end{equation}
with initial data $x(0) = x_0$, $M(0) = \mbox{Id}$, and $\eta(0) = 0$ up to time $\tau$. 
The computations carried out in the present work we 
exploit the standard ``off the shelf'' MatLab integration scheme 
known as \verb|rk45| to numerically solve these 
initial value problems when necessary.

\subsection{Periodic orbits of conservative systems} \label{sec:periodicOrbitsNum}
In this section we work at a fixed parameter value in $s \in [a,b]$ and hence 
suppress the dependance of $f$ and $\phi$ on $s$ all together.
So, let $f = f_s \colon \Omega \to \mathbb{R}^{d}$ denote our conservative vector field with 
$\Omega \subset \mathbb{R}^{d}$ open and
$E = E_s \colon \Omega \to \mathbb{R}$  the conserved quantity.  

It is well known (and much exploited in the present work) that periodic orbits of a conservative vector field 
appear in one parameter ``tubes'' parameterized by the conserved quantity, or equivalently -- by period. 
See for example the discussion of periodic orbits in the book of \cite{MR3642697}.
So, one locally isolates a particular periodic orbit in the ``tube'' by fixing a target period $T > 0$
as a problem parameter. Let $\tau_1, \ldots, \tau_K > 0$ 
have that $\tau_1 + \ldots + \tau_K  = 1$.  A periodic orbit
is equivalent to a solution $x_1, \ldots, x_K \in \Omega$ of the the system of 
equations
\begin{align*}
\phi(x_1, \tau_1 T) &= x_{2} \\
\phi(x_2, \tau_2 T) &= x_{3} \\
 & \vdots \\
\phi(x_K, \tau_K T) &= x_1.  
\end{align*}

Our strategy is to solve the system of equations using a Newton method, 
in which case it is important to consider only systems with isolated zeros.
Yet this system of equations cannot have a unique solution, as any phase shift of a 
periodic orbit is again a periodic orbit.   
To isolate, we introduce a Poincare phase condition.  That is, we require that the 
periodic orbit crosses a fixed co-dimension one affine subspace (or ``plane'')
$P \subset \mathbb{R}^{d}$ at time zero.  Let $y_0 \in \mathbb{R}^{d}$ be a point in the 
desired plane $P$ and $v \in \mathbb{R}^{d}$ be the direction vector of the plane.  Then 
$x_1 \in \mathbb{R}^{d} \in P$
if and only if 
\[
(x_1 - y_0)^T \cdot v = 0.
\]

When we append this equation to the system the resulting augmented system
has one more equation than unknowns.  The system needs to be balanced
by either eliminating an equation (which can be done using the conserved quantity)
or by introducing an ``unfolding parameter'' $\alpha \in \mathbb{R}$ -- the approach
taken in the present work.  An unfolding parameter is an artificial variable which balances 
the system,  and which should have the special property that it must be
zero at a periodic solution (this will be made precise below).

Deciding how to insert a new parameter into the problem is 
a delicate question, yet for conservative
vector fields there is a canonical choice.  
We describe the classic approach, as developed in
references \cite{MR1992054,MR2003792}, and 
define the unfolded family of vector fields by 
\[
f_\alpha(x) = f(x) + \alpha  \nabla E(x)^T,
\]
where $\alpha \in \mathbb{R}$ is the new unfolding parameter.
Suppose now that $\gamma \colon [0, T] \to \Omega$ is 
a periodic solution of 
\[
\gamma'(t) = f_\alpha(\gamma(t)).
\]
If $\nabla E(\gamma(t))$ not identically zero on $[0, T]$, then 
$\alpha = 0$, and - in fact -- $\gamma$ is a period $T$ 
solution of $\gamma' = f(\gamma)$.

To see this, consider the real valued function 
\[
g(t) = E(\gamma(t)).
\]
We have that 
\begin{align*}
g'(t) &= \nabla E(\gamma(t)) \gamma'(t) \\
& = \nabla E(\gamma(t)) f_\alpha(\gamma(t)) \\
& = \nabla E(\gamma(t)) f(\gamma(t)) + \alpha \nabla E(\gamma(t)) \nabla E(\gamma(t))^T \\
& = \alpha \| \nabla E(\gamma(t)) \|^2,
\end{align*}
as $f$ conserves $E$.
Since $\gamma$ has period $T$, we have that $\gamma(0) = \gamma(T)$ and hence 
\[
0  = E(\gamma(T)) - E(\gamma(0)) = g(T) - g(0).
\]
But 
\begin{align*}
 g(T) - g(0) & = \alpha \int_0^T \| \nabla E(\gamma(t)) \|^2 \, dt \\
 & = 0
\end{align*}
if and only if $\alpha = 0$, as $\|\nabla E(\gamma(t))\|$ is not identically zero.  

Now, denote by $\phi(x, t, \alpha)$  the local flow generated by $f_\alpha$, and note that the 
partial derivative of $\phi$ with respect to $\alpha$ is computed by solving the 
equation of first variation with respect to $\alpha$ as discussed in the 
previous section.  Define the map 
$F \colon  \Omega^K \times \mathbb{R} \subset \mathbb{R}^{d K + 1} \to \mathbb{R}^{d K + 1}$ by 
\begin{equation} \label{eq:perOrbMap}
F_T(x_1, x_2, x_3, \ldots, x_{K-1}, x_K, \alpha) = 
\left(
\begin{array}{c}
\phi(x_1, \tau_1 T, \alpha) - x_2 \\
\phi(x_2, \tau_2 T, \alpha) - x_3 \\
\vdots \\
\phi(x_{K-1}, \tau_{K-1} T, \alpha) - x_K \\
\phi(x_K, \tau_K T, \alpha)  - x_1 \\ 
(x_1 - y_0)^T \cdot v    
\end{array}
\right) 
\end{equation}
and note that a zero of $F_T$ is an orbit of period $T$ for $f$.
Let $\mathbf{x} = (x_1, \ldots, x_K, \alpha) \in  \mathbb{R}^{d K + 1}$.
The number of equations matches the number of unknowns, and the system is 
amenable to the Newton method
\[
\mathbf{x}_{j+1} = \mathbf{x}_j  + \Delta_j,
\]
where $\Delta_j$ is a solution of the linear system 
\[
DF_T(\mathbf{x}_j) \Delta_j = F_T(\mathbf{x}_j).
\]
The derivative involves only partial derivatives of the local flow, which 
are computed by solving variational equations \eqref{eq:varEqs}.  This is referred to as
a \textit{multiple shooting} scheme for the periodic orbit.

\begin{remark}
Some mechanical systems have the property that $\nabla E(x) = 0$ if and only if $f(x) = 0$.  
This happens for example when $f$ is a Hamiltonian vector field and $E$ is the Hamiltonian.
In this case $\nabla E(\gamma(t))$ is not identically zero if and only if $\gamma(t)$ is 
not constant, making this non-degeneracy hypothesis especially easy to check.
We also remark that for some mechanical systems it is possible to simplify the 
unfolded vector field $f_\alpha$.
For example if $f$ is a conservative system derived from Newton's laws, then 
one gets the same result by adding $\alpha$ times a linear dissipative term to the system 
rather than using the canonical gradient term.  
See \cite{MR2003792}.
\end{remark}

\subsection{Continuation with respect to period in conservative systems} \label{sec:continuationAppendix}
Note that the period $T > 0$ appears as a continuation parameter in the map $F_T$ defined in 
Equation \eqref{eq:perOrbMap}.  Then, 
supposing that $\mathbf{x}_0$ has $F_T(\mathbf{x}_0) = \mathbf{0}$ and that $DF_T(\mathbf{x}_0)$ is 
non-singular, we have by the implicit function theorem that 
there is a smooth branch of solutions of $\mathbf{x}(\beta)$
defined for $\beta \in (T-\epsilon, T+\epsilon)$, with $\mathbf{x}(T) = \mathbf{x}_0$.  Moreover, since 
\[
F_{\beta}(\mathbf{x}(\beta)) = 0,  \quad \quad \quad \mbox{for } \beta \in (T-\epsilon, T+\epsilon), 
\]
we have, after taking the derivative with respect to $\beta$, that  
\[
D F_{\beta}(\mathbf{x}(\beta)) \frac{d}{d \beta}\mathbf{x}(\beta) + \frac{\partial}{\partial \beta} F_{\beta}(\mathbf{x}(\beta)) = 0.
\]
Let $\mathbf{x}'(T) = \mathbf{v}$.  Then $\mathbf{v}$ solves the equation
\[
D F_{T}(\mathbf{x}_0)\mathbf{v} = - \frac{\partial}{\partial \beta} F_{T}(\mathbf{x}_0),
\]
giving the linear approximation of the branch through $\mathbf{x}_0$.
We note that the right hand side consists of terms of the form 
\[
\frac{\partial}{\partial \beta} \phi(x_j, \tau_j T, \alpha) = \tau_j f_\alpha(\phi(x_j, \tau_jT, \alpha),
\]
see Equation \eqref{eq:timePartialFlow}.  These expressions depend only on the unfolded vector field.  

Choosing a $\beta \neq T$,  we take 
\[
\mathbf{x}_1 = \mathbf{x}_0 + (\beta -T) \mathbf{v}, 
\]
as the first order approximation of $\mathbf{x}(\beta)$, corresponding to 
the periodic orbit with period $\beta \in (T-\epsilon, T+\epsilon)$
in a nearby energy level.  
Letting $\mathbf{x}_1$ serve as the initial guess for a zero of 
$F_{\beta}$, we run Newton's method and converge to a new solution -- 
for $|\beta-T|$ small enough.  

The process can be continued until the curve $\mathbf{x}(\beta)$ undergoes a bifurcation.
Such bifurcations are indicated by the singularity of the derivative $DF_T$, and are used 
to detect interesting phenomena in the main body of the paper.
We remark that the appearance of a saddle node bifurcation is not dynamically important 
in this context, as it indicates only that frequency does not vary monotonically within the
``energy tube''.  The tube itself is the object of interest.

We continue through such saddle node bifurcations using ``psudo-arclengh continuation,''
 a small modification of the algorithm just discussed.  More complete discussion of  
numerical algorithms based on psuco-arclength continuation 
is found in the book of \cite{MR2071006}.  See also \cite{MR3930602,matcont,MR1900655}.

\subsection{Homoclinic orbits in conservative systems}
Suppose that for $s_0 \in [a,b]$, $p_0 \in \Omega_{s_0}$ is a hyperbolic equilibrium solution 
of $f_{s_0}$.   
Suppose in addition that $d_u, d_s$ are respectively 
the dimension of the unstabe/stable eigenspaces and that
$d_u + d_s = d$.  
Since $p_0$ is a hyperbolic equilibrium solution,
there exists an $\delta > 0$ and 
$p_s \colon (s_0 - \delta, s_0 + \delta) \to \mathbb{R}^d$
so that $p(0) = p_0$ and 
\[
f_s(p(s)) = 0 \quad \quad \quad \mbox{for all } s \in (s_0 - \delta, s_0 + \delta).
\]

Moreover, we can choose $\delta > 0$ so that the dimension of the 
unstable/stable eigenspaces attached to $p(s)$ are $d_u$ and $d_s$ respectively 
for $s \in (s_0 - \delta, s_0 + \delta)$.
Since $f_s$ depends smoothly on $s \in [a,b]$, 
the parameterizations of the local unstable and stable manifolds 
depend smoothly on $s$ as well.  
Let $D_{u,s}$ be respectively the unit balls in
 $\mathbb{R}^{d_{u,s}}$ and suppose that 
$P_s \colon D_u \times  (s_0 - \delta, s_0 + \delta) \to \mathbb{R}^d$
and $Q_s \colon D_s \times  (s_0 - \delta, s_0 + \delta) \to \mathbb{R}^d$
are smooth parameterizations of the local unstable and stable 
manifolds attached to $p(s)$.    
 In particular, assume that for $s \in (s_0-\delta, s_0 + \delta)$,
 $w \in D_u \subset \mathbb{R}^{d_u}$, and 
 $z \in D_s \subset \mathbb{R}^{d_s}$ we have that 
 \[
 \lim_{t \to - \infty} \phi(P(w, s), t, s) = p(s)
 \quad \quad \quad \mbox{and} \quad \quad \quad 
 \lim_{t \to \infty} \phi(Q(z, s), t, s) = Q(s).
\]

Much like the case of a periodic orbit discussed above, we have that 
an infinitesimal shift of a homoclinic orbit segment is itself a homoclinic 
orbit segment.   Then a phase condition is required if we want to isolate a 
solution of the system.  We proceeded as follows, and define appropriate 
sections in the domains of $P$ and $Q$ respectively.  That is,
we take $w \colon [-1, 1]^{d_u-1} \to D_u$ and $z \colon [-1,1]^{d_s -1} \to D_s$
parameterizations of co-dimension one surfaces in $D_u$ and $D_s$.
We refer to the parameterized surfaces as ``phase surfaces'' in 
$D_u$ and $D_s$, and define 
\[
\hat P(\sigma, s) = P(w(\sigma), s)
\]
and 
\[
\hat Q(\theta, s) = Q(z(\theta), s).
\]


Let $s \in (s_0-\delta, s_0 + \delta)$,
$K \in \mathbb{N}$ and $\tau_1, \ldots, \tau_K > 0$ have $\tau_1 + \ldots + \tau_K = 1$.
We seek $x_1, \ldots, x_K \in \Omega_{s}$, $\sigma \in [-1,1]^{d_u-1}$, $\theta \in [-1, 1]^{d_s-1}$, 
and $T > 0$   so that 
\begin{align*}
\hat{P}(\sigma, s)  &=  x_1 \\
\phi(x_1, \tau_1 T, s) & = x_2 \\
\phi(x_2, \tau_2 T, s) & = x_3 \\
& \vdots \\
\phi(x_{K-1}, \tau_{K-1} T, s) & = x_K \\
\phi(x_K, \tau_K T, s) & = \hat{Q}(\theta, s).
\end{align*} 
A solution of this system of equations has that each of the
points $\hat{P}(\sigma, s)$, $x_1$, $\ldots$, $x_K$, and $\hat{Q}(\theta, s)$
lies on the same homoclinic orbit to $p(s)$.  That is, the homoclinic orbit 
is generated by flowing forward or backward any of these points.
It is worth noting that $T > 0$ is the ``time of flight'' from the unstable 
to the stable phase surface.  

Noting that $(\sigma, \theta) \in \mathbb{R}^{d_u + d_s - 2} = \mathbb{R}^{d-2}$
and counting variables, we see that $(\sigma, \theta, x_1, \ldots, x_K, T) \in \mathbb{R}^{dK+d-1}$
while the system of equations takes values in $\mathbb{R}^{dK+d}$.
Again -- and in direct analogy with the situation encountered when discussing periodic orbits -- 
we have one too few unknowns.  Just as in the periodic orbit case 
the system can be balanced by either adding an ``unfolding parameter'',
or by exploiting the conserved quantity to eliminate an equation.

Extending the ideas described  in the previous section, 
we once again employ the classic unfolding parameter
technique of \cite{MR1992054,MR2003792}, and
define 
\[
f_{s,\alpha} (x) = f_s(x) + \alpha \nabla E_s(x)^T.
\]
Let $\phi(x, t, s, \alpha)$ denote the associated flow.  
Arguing exactly as in the periodic case, we have that if $\gamma \colon [0, T] \to \mathbb{R}^d$
is an orbit with initial conditions $\gamma(0) = P(\sigma, s)$ and terminal conditions
$\gamma(T) = Q(\sigma, s)$, then $\alpha = 0$.  This relies on the fact that 
$P$ and $Q$ parameterize unstable/stable manifolds for the conservative vector 
field $f_s$, hence lie in the level set of the equilibrium solution.  That is
\[
E_s(P(\theta, s)) = E_s(Q(\sigma, s)) = E_s(P(s)),
\]
for all  $\sigma \in [-1,1]^{d_u-1}$, $\theta \in [-1, 1]^{d_s-1}$.

Then for fixed $s \in (s_0 - \delta, s_0 + \delta)$
a zero of the map $F_s \colon [-1,1]^{d_u-1} \times [-1, 1]^{d_s-1} 
\times \Omega_s^K \times \mathbb{R} \times \mathbb{R} \subset \mathbb{R}^{dK + d}
\to \mathbb{R}^{dK+d}$ defined by 
\begin{equation}\label{eq:homoclinicOpEq}
F_s(\sigma, \theta, x_1, \ldots, x_K, T, \alpha) = 
\left(
\begin{array}{c}
\hat{P}(\sigma, s)  - x_1 \\
\phi(x_1, \tau_1 T, s, \alpha) - x_2 \\
\phi(x_2, \tau_2 T, s, \alpha) - x_3 \\
 \vdots \\
\phi(x_{K-1}, \tau_{K-1} T, s, \alpha) - x_K \\
\phi(x_K, \tau_K T, s, \alpha) - \hat{Q}(\theta, s)
\end{array}
\right),
\end{equation}  
has $\alpha = 0$ and that the orbit of any of the points $x_1, \ldots, x_K$ is homoclinic
for $p(s)$.

Let $\mathbf{x} = (\sigma, \theta, x_1, \ldots, x_K, T, \alpha)$ denote the independent 
variable for $F_s$.  If $F_s(\mathbf{x}_0) \approx 0$ then we define the Newton sequence 
\[
\mathbf{x}_{n+1} = \mathbf{x}_n + \Delta_n,
\]
where $\Delta_n$ is the solution of the linear equation 
\[
DF_s(\mathbf{x}_n) \Delta_n = -F_s(\mathbf{x}_n).
\]
This approach is a multiple shooting scheme for homoclinic orbits in conservative systems. 
The iteration converges if $\mathbf{x}_0$ is close enough to a homoclinic orbit segment. 
Again, the derivative of $F_s$ involves derivatives of $P, Q$, and $\phi$
where the derivatives of $\phi$ with respect to $x$, $\alpha$ are computed 
by solving variational equations \eqref{eq:varEqs}. The derivative of $\phi$ with respect to $T$
is as given in Equation \eqref{eq:timePartialFlow}, and depends only on the unfolded vector field.

 \begin{remark}[Approximating the parameterizations $P$ and $Q$] \label{rem:parmMethInMultShoot}
 In practice it is very common to exploit the first order 
 approximations of $P$ and $Q$ by their eigenspaces, and indeed this is 
 the approach taken in many of the classic references 
\cite{MR1007358,MR1404122,MR1205453,MR2020992,MR1314079,MR1830909,MR2989589,MR1205453}. 
On the other hand, there exist well developed methods for numerically 
computing higher order approximations of $P$ and $Q$.  We exploit such methods in 
the present work as they (A) provide improved numerical stability/condition numbers 
in the boundary value problems defining the homoclinic orbits (just as higher order
numerical integration schemes provide improvements over the standard Euler method), and (B)
higher order methods are necessary for the center manifold calculations exploited  
in the critical calculations.  Using higher order methods throughout represents a 
more unified approach.  
For the high order approximation of $P$ and $Q$ in the present work we utilize 
numerical implementations, based on the parameterization method   
  \cite{MR1976079,MR1976080,MR2177465}, and developed in 
  \cite{MR3906230,shaneAndJay,MR3919451}.
 \end{remark}

 \subsection{Continuation of a conservative homoclinic with respect to a parameter}
 Suppose that $\mathbf{x}_0$ has $F_{s_0}(\mathbf{x}_0) = 0$, and that $DF_{s_0}(\mathbf{x}_0)$ is non-singular.
 Then, by the implicit function theorem there exists an $\epsilon > 0$ and a smooth branch of zeros 
 $\mathbf{x}(s)$ so that $F_{s}(\mathbf{x}(s)) = 0$ for $s \in (s_0 - \epsilon, s_0 + \epsilon)$, and $\mathbf{x}(s_0) = \mathbf{x}_0$.
 Differentiating with respect to $s$ leads to 
 \[
 D F_s(\mathbf{x}(s)) \frac{d}{d s} \mathbf{x}(s) + \frac{\partial}{\partial s} F_s(\mathbf{x}(s)) = 0.
 \]
 Letting $\mathbf{x}'(s_0) = \mathbf{v}$, we see that $\mathbf{v}$ solves the linear equation 
 \[
 DF_s(\mathbf{x}_0) \mathbf{v} = - \frac{\partial}{\partial s} F_s(\mathbf{x}_0).
 \]
Note that the right hand side consists of partial derivatives of the flow with respect to 
$\alpha$, and that these are found by solving variational equations.  

For $s \in (s_0-\epsilon, s_0+\epsilon)$ with $|s - s_0| \ll 1$, define 
\[
\mathbf{x}_1 = \mathbf{x}_0 + (s - s_0) \mathbf{v}, 
\]  
as the first order approximation of $\mathbf{x}(s)$ at $s_0$.  Then $\mathbf{x}_1$
is the first order approximation of a homoclinic orbit segment to $p(s)$.
We take $\mathbf{x}_1$ as our initial guess for a zero of $F_s(\mathbf{x})$ and 
once again iterate the Newton scheme.  
Again, this process can be iterated until $\mathbf{x}(s)$ undergoes a bifurcation.  

One technical difficulty is that this procedure requires the computations of  
derivatives of $P$ and $Q$ with respect to the parameter $s$.  If $P$ and $Q$ are found
using the parameterization method as discussed in Remark \ref{rem:parmMethInMultShoot}, 
then the partial derivatives with respect to parameter can be computed using the variational 
equations developed in \cite{JDMJ01}.  On the other hand, it is also possible to compute 
the partial derivatives using finite differencing schemes.  For exampe if we know $P(\sigma, s_0)$
and $P(\sigma, s_1)$ then we have
\[
\frac{\partial}{\partial s} P(\sigma, s_0) \approx \frac{P(\sigma, s_1) - P(\sigma, s_0)}{s_1 - s_0}.
\]
Since the parameterizations at the new parameter $s_1$ have to be computed in order to define 
the map $F_{s_1}$, this differencing does not incur any additional cost.  In the present work we find
that this differencing scheme is sufficient for our purposes.

\section{Formal series calculation of the center stable/unstable manifolds} \label{sec:centerCalcAppendix}
The calculations discussed in Section \ref{sec:blueSkies} provide supporting evidence
for the claims made in Section \ref{sec:numCont} in cases when the homoclinic 
orbits do not shrink in size to zero.  In these cases, failure of the numerical continuation 
algorithms is taken as an indication that the homoclinic orbits undergo a bifurcation.
On the other hand, when the connecting orbit shrinks to zero size the continuation 
algorithm eventually fails for other reasons: essentially it becomes increasingly difficult 
to distinguish the homoclinic orbit form the equilibrium solution itself.
This is the case for the $\gamma_{2, k}(s)$, $k = 1,2,3$ and the $\gamma_{3,1}(s)$
homoclinic families, which we conjectured survive all the way to the critical curve $\mathfrak{D}$.

Fortunately, small amplitude homoclinic orbits can be studied by completely different -- and 
much more local -- center manifold methods.  We now turn to a method for 
computing the center manifold of a critical libration point when the system has parameter 
values on (or near) the critical curve $\mathfrak{D}$.
We employ a novel parameterization method recently 
developed by van den Berg, Rink, and the first author \cite{wouterCenterManifoldsII}.  

The parameterization method is a functional analytic framework for studying invariant 
manifolds developed by Cabr\'{e}, Fontich, and de la Llave \cite{MR1976079,MR1976080,MR2177465}.
We refer the interested reader also to the comprehensive recent book on the 
subject by Haro, Figueras, Mondelo, Luque, and Canadell \cite{mamotreto}.
The parameterization method was used extensively in the work of 
\cite{MR3906230,MR3919451}, and also in the recent work of
Murray and the second author on homoclinic dynamics for planar and spatial 
periodic orbits in the CRTBP
\cite{chebManifolds,maximeMurray_connectionsVerticalLyap}.

Of course, computational methods for center manifolds, and related methods
for numerical calculation of normal forms, have a long 
history in the Celestial Mechanics literature. Other 
method than those we choose could have been used here.  
The literature is substantial and a thorough review of the literature is beyond
the scope of the present work.  
We refer the interested reader to the works of  
\cite{MR1054714,MR1636119,MR1705705,MR1700577,MR1720891,MR1630282,MR2268196,MR2670181}.
We also refer to the books \cite{MR1867240,MR1881823,MR1878993,MR1875754} 
for extensive computational treatment of invariant manifolds and their use in 
space mission design.
Much more detail on center manifolds for parabolic equilibria with applications to
Celestial Mechanics is found in \cite{MR3642259,MR2276478,MR2030148}.
This interested reader may want to consult the lecture notes 
 \cite{1990mmcm.conf..285S}.

Turning to the parameterization method, let us establish some notation.
We focus here on the case that the phase space is four dimensional, 
as this is the only case encountered in the present work, however most of the 
results reported here generalize as discussed in the references above.
Let $\Omega \subset \mathbb{R}^4$ be an open subset and suppose that 
$f \colon \Omega \to \mathbb{R}^4$ is a smooth and conservative vector field. 
Consider a connected open set $U \subset \mathbb{R}^d$ with $d = 1,2,3$, 
a one-to-one map $K \colon U \to \mathbb{R}^4$, and a vector field 
$r \colon \mathbb{R}^d \to \mathbb{R}^d$.  If $K$ satisfies the infinitesimal 
invariance equation 
\begin{equation} \label{eq:invEq}
DK(\mathbf{y}) r(\mathbf{y}) = f(K(\mathbf{y})),  \quad \quad \quad \mathbf{y} \in U,
\end{equation}
then the image of $K$ is locally invariant under the flow generated by $f$.
In fact, $K$ lifts orbits of $\mathbf{y}' = r(\mathbf{y})$ to orbits of $f$.
If the vector field $f$ is inflowing/outflowing on the boundary of $K(U)$, then the image 
of $K$ is forward/backward invariant.
The main idea of the parameterization method is to study appropriate versions of 
Equation \eqref{eq:invEq} in various important situations of interest.  
The geometric meaning of Equation \eqref{eq:invEq} is illustrated in Figure \ref{fig:parmSchematic}.

\begin{figure}[!t]
\centering
\includegraphics[scale = 0.25]{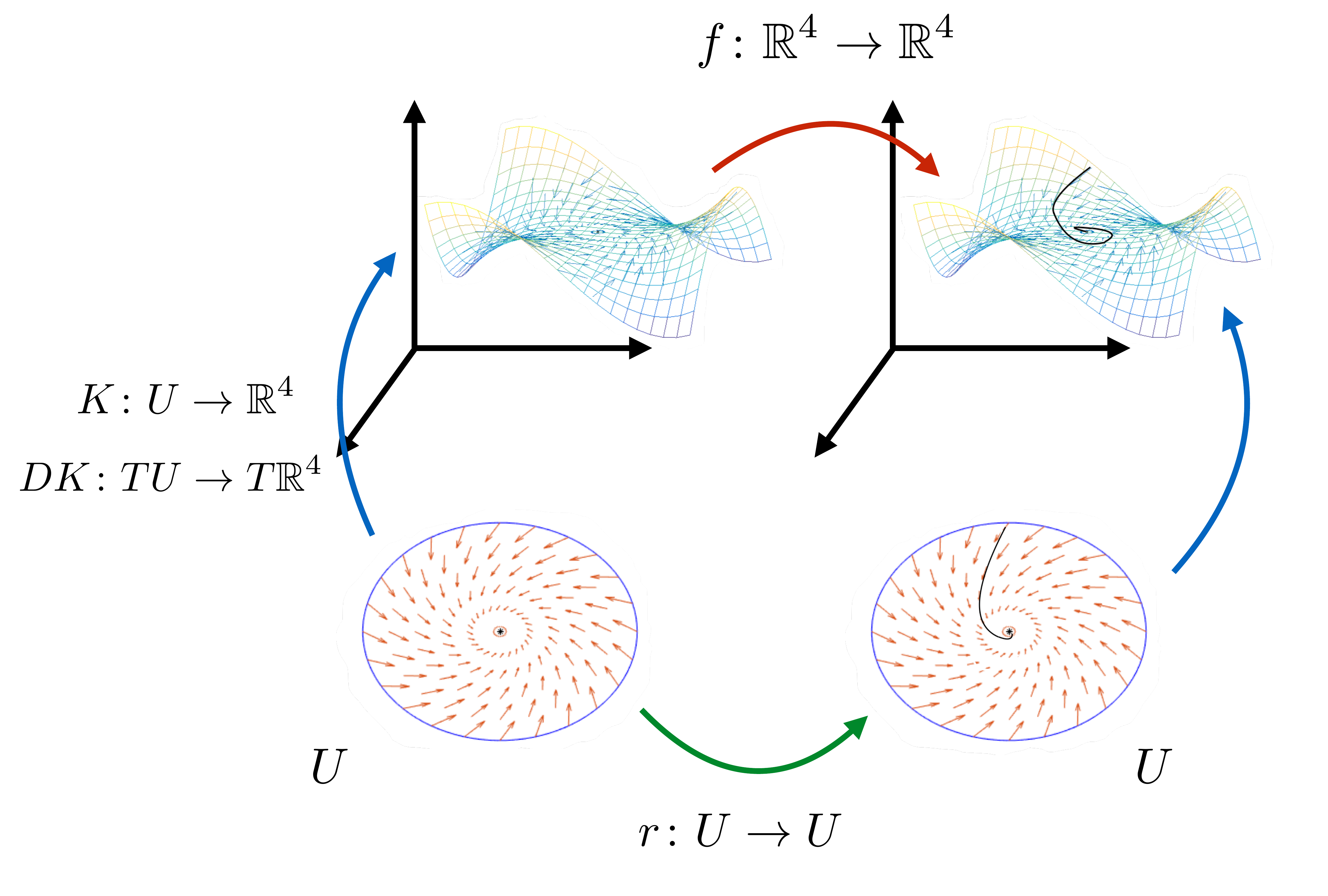}
\caption{\textbf{The geometric meaning of Equation \eqref{eq:invEq}:} 
We are interested in a chart map $K$ parameterizing an invariant manifold patch.
The parameterization method singles out special charts satisfying 
an infinitesimal conjugacy. The  idea is that $K$ functions in 
two complementary ways.  On one hand, $K$ embeds $U$ into the phase space, hence is 
a chart for a manifold patch.  On the other 
hand, $DK$ maps a vector field $r$ defined in $U$ into the tangent 
space of this manifold.  If the two vector fields 
--$f$ restricted to the image of $K$, and the push forward of $r$ by $DK$ 
are equal (as required by 
Equation \eqref{eq:invEq}) -- then $K$ maps orbits of $r$ to 
orbits of $f$ and the manifold patch is locally invariant.}
\label{fig:parmSchematic}
\end{figure}

We now describe a version of the parameterization method 
for parabolic equilibrium solutions. That is, we consider situations where 
$Df(\mathbf{x}_0)$ has one or more zero eigenvalues.
In the present paper we are primarily interested in the fold 
bifurcation, though the pitchfork is considered briefly as well.

The technique used here is developed in \cite{wouterCenterManifoldsII}, and has not 
been applied in the context of Celestial Mechanics until now.
Because of this novelty we discuss the method in somewhat more detail than in the previous two 
sections, where we reviewed material appearing already in the literature.  
Returning to Equation \eqref{eq:invEq}, the idea is that we must now solve simultaneously for the 
embedding $K$ of the center manifold and the model/inner dynamics $r$.

As we have seen in Section \ref{sec:normalform}, at the Saddle node bifurcation there exists a solution branch which is non-linear stable, as well as a solution branch which is non-linear unstable. After we compute the embedding $K$, we exploit the Jacobi integral $E$ to find those branches.
If $\branch(t)$ is a characterization of the stable solution branch for the conjugate vector field $r$ on the center subspace, we have that $\conjugacy(\branch(t))$ is a solution branch of the original differential equation $\dot{\boldx}(t) = \dynamics(\boldx(t))$ by the construction of $\conjugacy$. Thus the energy $E$ is 
constant along $\conjugacy( \branch (t))$. Furthermore, 
we have $\lim_{t \to  \infty} \conjugacy( \branch (t)) = \boldx_0$, hence
\begin{align*}
E_0 \bydef E(\boldx_0) = E(\lim_{t \to  \infty} \conjugacy( \branch (t))) = \lim_{t \to   \infty }E( \conjugacy( \branch (t))) = E( \conjugacy( \branch (t)))  && \text{ for all $t \in \mathbb{R}$}.
\end{align*}
If $\branch(t)$ is a characterization of the unstable branch, we consider the limit of $t$ going to $ -\infty$ to again find that
\begin{align*}
E_0 \bydef E(\boldx_0) = E(\lim_{t \to - \infty} \conjugacy( \branch (t))) = \lim_{t \to -  \infty }E( \conjugacy( \branch (t))) = E( \conjugacy( \branch (t)))  && \text{ for all $t \in \mathbb{R}$}.
\end{align*}
Hence the (un)stable branch is found by solving the energy equation
\begin{equation}\label{eq: stablebranch}
E \circ \conjugacy \circ \branch - E_0 = 0, 
\end{equation}
using a power matching scheme.

\subsection{Center Stable/Unstable Manifolds} \label{sec:centerStableUnstable}

After finding the stable and unstable branches in the center manifold, we want to construct the center 
(un)stable manifolds. For convenience, we consider only the center stable manifold 
in the following sections.  The center unstable manifold is similar. We consider connected open sets $\centersubspace \subset \mathbb{R}^2$  and $\centerstablesubspace \subset \mathbb{R}^3$ for the role of $U$ in the invariance equation \eqref{eq:invEq} for the center manifold and center stable manifold, respectively. So, for the center stable manifold we seek a solution of the conjugacy equation
\begin{align*}
 D \centerstableconj(\boldy,z) 
\centerdynamics'(x,y,z) =\dynamics \circ \centerstableconj(x,y,z) && \text{ for all $(x,y,z) \in \centerstablesubspace$.}
\end{align*} 
We know again from \cite{wouterCenterManifoldsII} that such a conjugacy and conjugate dynamical system exists. 
In this case, the conjugacy $\centerstableconj$ is between the center stable subspace and the center 
stable manifold, and the vector field $\centerdynamics'$ is the conjugate vector field on the center stable 
manifold. The center subspace and manifold are naturally embedded in the center stable subspace and 
manifold respectively. Hence, we require that $\centerstableconj(\boldy,0) = \conjugacy(\boldy)$ for all 
$\boldy =(x,y) \in \centersubspace$.  It is too much to ask that the conjugate dynamics on 
the center stable subspace are uncoupled, i.e. 
$\centerdynamics'(\boldy,z) = ( \centerdynamics(\boldy), \stabledynamics(z))$ cannot hold to all orders.
Instead, as we will see in Lemma \ref{thm: centerstableconjugacyformalseries}, we impose that the conjugate vector field 
 in the center direction is uncoupled and the vector field on the stable fibers is given by $\stabledynamics(x)z$, 
 where $x$ is the first coordinate on the center subspace, up to arbitrary order.  Consider then the  
 conjugacy equation
\begin{align}
 D \centerstableconj(\boldy,z) ( \centerdynamics(\boldy) , \stabledynamics(x) z) = \dynamics \circ \centerstableconj(\boldy,z)&& \text{ for all $(\boldy,z) \in \centerstablesubspace$ and $\boldy = (x,y)$,} \label{eq: centerstableconjugacy}
\end{align} 
subject to the constraint $\centerstableconj(\boldy, 0 ) = \conjugacy(\boldy)$.

\subsection{Formal Series Calculations} \label{sec:formalSeries}

Since center manifolds are in general not analytic, we search for solutions of 
in the space of $C^n$.  Nevertheless, we use formal power series 
for computational convenience.  For notational convenience, we apply a translation and coordinate transformation to move $L_c$ to the origin and have $D\dynamics(L_c)$ in Jordan normal form. To find the solution branches $\branch(t)$ we will need to know the explicit coordinate transformation. In Proposition \ref{thm: bifurcationpoint} we found the eigenvectors $\bv_0$ and $\bv_1$ of the center subspace. The stable and unstable eigenvectors are given by
\begin{align*}
\bv_{\pm} \bydef \left( \Omega_{xy} + 2 \lambda_{\pm} , 
\lambda_{\pm} \left( \Omega_{xy} + 2 \lambda_{\pm} \right),  
\Omega_{yy} - 4 , \lambda_{\pm} 
\left(  \Omega_{yy} - 4 \right) \right),
\end{align*}
where $\lambda_{\pm} \bydef \mp\sqrt{-4 + \Omega_{xx} + \Omega_{yy}}$ are the stable and unstable eigenvalues. We define the coordinate transformation $\mathcal{C} \bydef (\bv_0, \bv_1, \bv_+ , \bv_-)$, and redefine $\dynamics(\boldx) \bydef \mathcal{C}^{-1} f(\mathcal{C} \boldx + \boldx_0)$.

\begin{lemma}\label{thm: centerconjugacyformalseries} 
For every formal series $\centerconj: \mathbb{R}^2 \to \mathbb{R}^2$ given by
\begin{align}
\centerconj\begin{pmatrix} x \\ y \end{pmatrix} \bydef \begin{pmatrix} x \\ y \end{pmatrix} 
+ \sum_{n=2}^\infty  \sum_{\substack{(i,j) \in \mathbb{N}^2 \\ i + j = n}} \begin{pmatrix} 
a_{i,j} x^i y^j \\ b_{i,j} x^i y^i \end{pmatrix},
\end{align}
there exist formal series $\hyperbolicconj: \mathbb{R}^2 \to \mathbb{R}^2$ and 
$\centerdynamics : \mathbb{R}^2 \to \mathbb{R}^2$ given by
\begin{align}
\hyperbolicconj \begin{pmatrix} x \\ y \end{pmatrix} &\bydef  \sum_{n=2}^\infty 
\sum_{\substack{(i,j) \in \mathbb{N}^2 \\ i + j = n}} \begin{pmatrix} \alpha_{i,j} x^i y^j \\ 
\beta_{i,j} x^i y^i \end{pmatrix}, \\
\centerdynamics \begin{pmatrix} x \\ 
y \end{pmatrix} &\bydef \begin{pmatrix} y \\ 
0 \end{pmatrix} + \sum_{n=2}^\infty \sum_{\substack{(i,j) \in 
\mathbb{N}^2 \\ i + j = n}} \begin{pmatrix} \gamma_{i,j} x^i y^j \\ \varepsilon_{i,j} x^i y^i \end{pmatrix},
\end{align}
such that $\conjugacy \bydef ( \centerconj, \hyperbolicconj)$ and $\centerdynamics$
  solve Equation \eqref{eq:invEq}. Furthermore, instead of fixing the constants $(b_{i,j})_{(i,j) \in \mathbb{N}^2}$ and solving for the constants $(\gamma_{i,j})_{(i,j) \in \mathbb{N}^2}$ among others, we could fix the constants $(\gamma_{i,j})_{(i,j) \in \mathbb{N}^2}$ and solve for $(b_{i,j})_{(i,j) \in \mathbb{N}^2}$.
\end{lemma}
\begin{proof}
We define the homogeneous polynomials
\begin{align*}
P_{\centerconj}^n \begin{pmatrix} x \\ y \end{pmatrix} &\bydef  
\sum_{\substack{(i,j) \in \mathbb{N}^2 \\ i + j = n}} \begin{pmatrix} a_{i,j} x^i y^j \\ 
b_{i,j} x^i y^i \end{pmatrix}, \\
P_{\hyperbolicconj}^n \begin{pmatrix} x \\ y \end{pmatrix} &\bydef  
\sum_{\substack{(i,j) \in \mathbb{N}^2 \\ i + j = n}} 
\begin{pmatrix} \alpha_{i,j} x^i y^j \\ \beta_{i,j} x^i y^i \end{pmatrix}, \\
P_{\centerdynamics}^n \begin{pmatrix} x \\ y \end{pmatrix} 
&\bydef \sum_{\substack{(i,j) \in \mathbb{N}^2 \\ i + j = n}} 
\begin{pmatrix} \gamma_{i,j} x^i y^j \\ \varepsilon_{i,j} x^i y^i \end{pmatrix},
\end{align*}
and we will prove that we can recursively define $P_{\hyperbolicconj}^n$ and 
$P_{\centerdynamics}^n$ in terms of $P_{\centerconj}^m$, $P_{\hyperbolicconj}^m$, 
and $P_{\centerdynamics}^m$  for $m < n$. It is clear that $\hyperbolicconj$ and $\centerdynamics$ satisfy Equation \eqref{eq:invEq} up to first order.
Now, assume that $P_{\hyperbolicconj}^m$ and $P_{\centerdynamics}^m$ for $m < n$ are 
such that $\conjugacy$ and $\centerdynamics$ satisfy Equation \eqref{eq:invEq} up to order 
$n -1$. Then the $n$-th order terms of the difference between the right and left hand side of Equation \eqref{eq:invEq} are given by
\begin{align}
			 &\sum_{\substack{(i,j) \in \mathbb{N}^2 \\ i + j =n}} \begin{pmatrix}  
			 b_{i,j} x^iy^j - \gamma_{i,j} x^i y^j  - i a_{i,j} x^{i-1} y^{j+1} \\
			-\varepsilon_{i,j} x^i y^j  - i b_{i,j} x^{i-1} y^{j+1}\\
			\lambda_+ \alpha_{i,j} x^i y^j  - i \alpha_{i,j} x^{i-1} y^{j+1}\\
			\lambda_- \beta_{i,j} x^i y^j  - i \beta_{i,j} x^{i-1} y^{j+1}
			\end{pmatrix} - \mathcal{P}^n \begin{pmatrix} x \\ y \end{pmatrix} \nonumber \\
			=& \begin{pmatrix*}[l]  \left( b_{n,0}  - \gamma_{n,0}  \right) x^n \\
			 -\varepsilon_{n,0} x^n\\
			 \lambda_+ \alpha_{n,0}  x^n\\
			 \lambda_- \beta_{n,0}    x^n\end{pmatrix*} + 
			\sum_{\substack{(i,j) \in \mathbb{N}^2 \\ i + j =n \\ 
			(i,j) \neq (n,0)}} \begin{pmatrix*}[l]  
			\left( b_{i,j}  - \gamma_{i,j}   - (i+1) a_{i+1,j-1} \right) x^{i} y^{j} \\
			\left( -\varepsilon_{i,j}   - (i+1) b_{i+1,j-1} \right) x^{i} y^{j}\\
			\left( \lambda_+ \alpha_{i,j}   - (i+1) \alpha_{i+1,j-1} \right) x^{i} y^{j}\\
			\left( \lambda_- \beta_{i,j}   - (i+1) 
			\beta_{i+1,j-1} \right) x^{i} y^{j}
			\end{pmatrix*} - \mathcal{P}^n \begin{pmatrix} x \\ y 
			\end{pmatrix}.
			\label{eq: nthorderconjugacy}
\end{align}
Here $\mathcal{P}^n$ consists of the $n$-th order terms of $\dynamics(\conjugacy^{<n}(\boldy)) - D\conjugacy^{<n}(\boldy) \centerdynamics^{<n}(\boldy)$ where $\conjugacy^{<n} = \sum_{m=1}^{n-1} \left( P_{\centerconj}^m , P_{\hyperbolicconj}^m \right)$ and $\centerdynamics^{<n} = \sum_{m=1}^{n-1} P_{\centerdynamics}^m$. For part of the computational implementation of obtaining $\mathcal{P}^n$ we refer to Appendix \ref{sec:numericalCenterManifold}.
In particular, $\mathcal{P}^n$ is an expression in terms of the constants $a_{i,j}$, $b_{i,j}$, $\alpha_{i,j}$, $\beta_{i,j}$, $\gamma_{i,j}$ and $\varepsilon_{i,j}$ for $i+j <n$. Hence we find unique $\alpha_{i,j}$, $\beta_{i,j}$, $\gamma_{i,j}$ and $\varepsilon_{i,j}$ such that 
\eqref{eq: nthorderconjugacy} vanishes
 starting from $(i,j) = (n,0)$. In particular, we see that if $\gamma_{i,j}$ is fixed instead of $b_{i,j}$, 
 we can make 
 \eqref{eq: nthorderconjugacy} vanish for $\alpha_{i,j}$, $\beta_{i,j}$, $b_{i,j}$ and $\varepsilon_{i,j}$
  instead.
\end{proof}

Now that we have a solution to Equation \eqref{eq:invEq}, we want to find the stable branch in the center 
subspace. Recall that we have applied the coordinate transformation
\begin{align*}
\mathcal{C} = \left( \begin{array}{cccc}
\Omega_{yy} &  \Omega_{xy}  & \multirow{4}{*}{$\bv_+$} &  \multirow{4}{*}{$\bv_-$}\\
0 & \Omega_{yy} \\
- \Omega_{xy} & 2 - \Omega_{xx} \\
0 & - \Omega_{xy}
\end{array}\right),
\end{align*}
and a translation by $\boldx_0$ on $\dynamics$. Hence the conserved quantity becomes $\energy = E (\mathcal{C} \boldx + \boldx_0)$, with the Jacobi integral $E(x, \dot x, y, \dot y) = 
-\left( {\dot x}^2 + {\dot y}^2 \right) + 2\Omega(x,y)$.

\begin{lemma}\label{thm: centerstablebranchformalseries}
There exists a formal series $\branch : \mathbb{R} \to \mathbb{R}^2$ given by 
\begin{align}
\branch(t) \bydef \begin{pmatrix} 0 \\ t^3 \end{pmatrix} + \sum_{n=2}^\infty \begin{pmatrix} 
  c_n t^n \\  0 \end{pmatrix}
\end{align} 
which solves Equation \eqref{eq: stablebranch} if the constant 
\begin{align*}
\mathscr{E} \bydef \Omega_{xxx} \Omega_{yy}^3 - 3 \Omega_{xxy} 
\Omega_{yy}^2 \Omega_{xy} + 3 \Omega_{xyy} 
\Omega_{yy} \Omega_{xy}^2 - \Omega_{yyy} \Omega_{xy}^3
\end{align*}
is non-zero.
\end{lemma}
\begin{proof}
We first note that 
\begin{align*}
\energy_0 \bydef \energy(\bzero) = \energy(\conjugacy(\bzero)) 
= \energy(\conjugacy(\branch(0))) =
\energy \circ \conjugacy \circ \branch (0),
\end{align*}
thus the constant term of Equation \eqref{eq: stablebranch} vanishes. For the higher order
 terms we want to find the Taylor expansion of $\energy$. Since both $\Omega_x$ and 
 $\Omega_y$ vanish, we see that $\energy$ has no linear terms in its Taylor expansion. 
 We assume that $\branch(t) = (O(t^2), t^3)$, hence we have 
 $\conjugacy( \branch (t))= (O(t^2), O(t^3), O(t^4), O(t^4) )$. Thus we must show that the 
 coefficient of $\boldx_1^2$ in the expansion of $\energy$ is $0$. We have by the chain rule, we write $\energy_{i_1,\dots,i_m}$ to denote the partial derivative of $\energy$ with respect to $x_{i_1}$,\dots, $x_{i_m}$,
\begin{align*}
\energy_{1,1} &= 2 \Omega_{xx} \Omega_{yy}^2  
+ 4\Omega_{xy} \Omega_{yy} \left( - \Omega_{xy} \right) + 2
 \Omega_{yy} \left( - \Omega_{xy} \right)^2 = 0.
\end{align*}
Therefore, the expansion of $\energy \circ \conjugacy \circ \branch - \energy_0$ is at least of order $5$. To rule 
out terms of order $5$, we must show that the coefficient of $\boldx_1\boldx_2$ in the expansion 
of $\energy$ is $0$. In fact, we have the more general result
\begin{align*}
\energy_{1,i} &= 2 \Omega_{xx} \Omega_{yy} \mathcal{C}_{1,i}  + 2 \Omega_{xy} \left( \Omega_{yy} \mathcal{C}_{3,i} - \Omega_{xy} \mathcal{C}_{1,i} \right) + 2 \Omega_{yy} \left( - \Omega_{xy} \right) \mathcal{C}_{3,i} = 0,
\end{align*}
for $i = 2,3,4$. Thus the leading term of $\energy \circ \conjugacy \circ \branch - \energy_0$ is of order 
$6$, and its coefficient is determined by the coefficients of $\boldx_2^2$ and $\boldx_1^3$ 
in the expansion of $\energy$. We have
\begin{align*}
\energy_{2,2} &= 2 \Omega_{xx} \Omega_{xy}^2 + 4
\Omega_{xy} \Omega_{xy} \left( 2 - \Omega_{xx} \right) + 2 \Omega_{yy} \left( 2- \Omega_{xx} \right)^2 - 2 \left( \Omega_{yy}^2 + \left( - \Omega_{xy} \right)^2 \right) \\
			&= 2 \Omega_{yy} \left( 4 - \Omega_{xx} - \Omega_{yy} \right), \\
\energy_{1,1,1} 
&=2 \left( \Omega_{xxx} \Omega_{yy}^3 -
 3 \Omega_{xxy} \Omega_{yy}^2 \Omega_{xy} +
  3 \Omega_{xyy} \Omega_{yy} \Omega_{xy}^2 - 
  \Omega_{yyy} \Omega_{xy}^3 \right).
\end{align*} 
From the proof of Proposition \ref{thm: bifurcationpoint}, it follows that both $\Omega_{yy}$ and 
$4 - \Omega_{xx} - \Omega_{yy}$ are non-zero, hence $\energy_{2,2} \neq 0$.
 By assumption, $\frac{1}{2} \energy_{1,1,1} = \mathscr{E} \neq 0$. Hence, the leading order 
 of the expansion of $\energy \circ \conjugacy \circ \branch - \energy_0$ is 
\begin{align*}
\frac{1}{2} \energy_{2,2} t^6  +  \frac{1}{6} \energy_{1,1,1} c_2^3 t^6.
\end{align*}
Since we want that $\energy \circ \conjugacy \circ \branch - \energy_0 = 0$, we must have  $\frac{1}{2} \energy_{2,2} + \frac{1}{6} \energy_{1,1,1} c_2^3 =0$, which uniquely determines $c_2 = \sqrt[3]{- 3 \energy_{2,2}/\energy_{1,1,1}} \neq 0$, with the convention 
$\sqrt[3]{-x} = - \sqrt[3]{x}$ for $x \ge 0$. We can now recursively find $c_n$ for $n \ge 3$.
 If we have found $c_m$ for $m \le n$ such that  $\energy \circ \conjugacy \circ \branch - \energy_0$ 
 vanishes for all order up to $t^{n+4}$, the $n+5$-th order is given by
\begin{align*}
\frac{1}{2}\energy_{1,1,1} c_2^2 c_{n+1} t^{n+5} + \mathcal{P}_{n+5}.
\end{align*}
Here $\mathcal{P}_{n+5}$ is the $n+5$-th order term of $\energy(\conjugacy(\branch^{<n}(t))) - \energy_0$ where $\branch^{<n} (t) = \left( \genfrac{}{}{0pt}{}{0}{t^3}\right) + \sum_{m=2}^{n-1} \left( \genfrac{}{}{0pt}{}{c_mt^m}{0}\right)$. Hence 
$\mathcal{P}_{n+5}$ only depends $c_m$ for $m \le n$. 
Hence $c_{n+1}$ is uniquely determined 
and $\energy \circ \conjugacy \circ \branch - \energy_0$ vanishes for all order up to $t^{n+5}$. 
Thus $\branch : \mathbb{R} \to \mathbb{R}^2$ solves Equation \eqref{eq: stablebranch}. 
\end{proof}

Depending on the sign of $t$, $\branch(t)$ is part of the unstable or stable solution branch. We will see that the stable branch is given by $\stablebranch \bydef \branch |_{t > 0}$ and the unstable branch is  $\unstablebranch \bydef \branch |_{t < 0}$. Finally, we show 
 that we can find a unique expansion for the center stable manifold.

\begin{lemma}\label{thm: centerstableconjugacyformalseries}
There exists formal series $\centerstableconj : \mathbb{R}^3 \to \mathbb{R}^4$ and 
$\stabledynamics : \mathbb{R} \to \mathbb{R}$ given by
\begin{align}
\centerstableconj \begin{pmatrix} x \\ y \\ z \end{pmatrix} 
&=\begin{pmatrix} x \\ y \\ z \\ 0 \end{pmatrix} +  
\sum_{n=2} \sum_{\substack{(i,j,k) \in 
\mathbb{N}^3 \\ i + j + k = n }} \begin{pmatrix} a_{i,j,k} x^iy^j z^k \\
b_{i,j,k} x^i y^j z^k \\
\alpha_{i,j,k} x^i y^j z^k \\
\beta_{i,j,k} x^i y^j z^k 
\end{pmatrix}, \\
\stabledynamics (x) &= \lambda_+ + \sum_{n=2}^\infty \zeta_{n} x^{n-1},
\end{align}
which solve Equation \eqref{eq: centerstableconjugacy}. Furthermore, we demand that  
$a_{i,j,0}$, $b_{i,j,0}$, $\alpha_{i,j,0}$ and $\beta_{i,j,0}$ are given by the constants $a_{i,j}$, $b_{i,j}$, $\alpha_{i,j}$ and $\beta_{i,j}$ from Lemma \ref{thm: centerconjugacyformalseries} respectively. This uniquely determines $\centerstableconj$ and 
$\stabledynamics$ except for $\alpha_{0,k-1,1}$ for $k \ge 1$.
\end{lemma}
\begin{proof}
We define the homogeneous polynomial
\begin{align*}
\mathcal{P}_{\centerstableconj}^n \begin{pmatrix} x \\ y \\ z 
\end{pmatrix} &\bydef \begin{pmatrix} \mathcal{P}_{\centerconj}^n \\
 \mathcal{P}_{\hyperbolicconj}^n \end{pmatrix} \begin{pmatrix}  x \\ 
 y \end{pmatrix} +  \sum_{\substack{(i,j,k) \in \mathbb{N}^3 \\ 
 k \neq 0 \\ i + j + k = n }} \begin{pmatrix} a_{i,j,k} x^iy^j z^k \\
b_{i,j,k} x^i y^j z^k \\
\alpha_{i,j,k} x^i y^j z^k \\
\beta_{i,j,k} x^i y^j z^k 
\end{pmatrix},
\end{align*}
with $\mathcal{P}_{\centerconj}^n$ and $\mathcal{P}_{\hyperbolicconj}^n$ 
defined in Lemma \ref{thm: centerconjugacyformalseries}. We want to recursively
 find $\mathcal{P}_{\centerstableconj}^n$ together with $\zeta_n$. Thus, 
 suppose we have found $\mathcal{P}_{\centerstableconj}^m$ together 
 with $\zeta_m$ for $m \le n-1$ such that Equation \eqref{eq: centerstableconjugacy} 
 vanishes up order $n-1$. Then the $n$-th order of 
Equation \eqref{eq: centerstableconjugacy} is given by
\begin{align}
\sum_{\substack{(i,j,k) \in \mathbb{N}^3 \\ i + j + k = n }} \begin{pmatrix} 
b_{i,j,k} x^i y^j z^k - \gamma_{i,j} \delta_{k,0} x^iy^j - i a_{i,j,k} x^{i-1}y^{j+1} z^k - k 
\lambda_+ a_{i,j,k} x^i y^j z^k  \\
 - \varepsilon_{i,j} \delta_{k,0} x^iy^j - i b_{i,j,k} x^{i-1}y^{j+1} z^k - k 
 \lambda_+ b_{i,j,k} x^i y^j z^k  \\
\lambda_+ \alpha_{i,j,k} x^i y^j z^k - \delta_{j,0} \delta_{k,1} 
\zeta_{n} x^{n-1} z - i \alpha_{i,j,k} x^{i-1}y^{j+1} z^k - k
 \lambda_+ \alpha_{i,j,k} x^i y^j z^k  \\
  \lambda_- \beta_{i,j,k} x^i y^j z^k - i 
  \beta_{i,j,k} x^{i-1}y^{j+1} z^k - k
   \lambda_+ \beta_{i,j,k} x^i y^j z^k  
\end{pmatrix} - \mathcal{P}^n \begin{pmatrix} x \\ y \\ z 
\end{pmatrix}. \label{eq: nthordercsequation}
\end{align}
Here $\mathcal{P}^n$ consists of the $n$-th order terms of $\dynamics(\centerstableconj^{<n}(\boldy,z)) - D\centerstableconj^{<n}(\boldy,z) \left( \centerdynamics (\boldy) ,\stabledynamics^{<n}(x)z \right)$ where $\centerstableconj^{<n} = \sum_{m=1}^{n-1} P_{\centerstableconj}^m$ and $ \stabledynamics^{<n}(x) = \lambda_+ + \sum_{m=2}^{n-1}  \zeta_m x^{m-1}$. For part of the computational implementation of obtaining $\mathcal{P}^n$ we refer to Appendix \ref{sec:numericalCenterManifold}. Hence  
$\mathcal{P}^n$ depends on $\mathcal{P}_{\centerstableconj}^m$ and $\zeta_{m}$ 
for $m < n$. We will show that we can recursively make 
 \eqref{eq: nthordercsequation}  vanish, starting from $k=0$ up to $k = n$. 

When $k=0$, 
 \eqref{eq: nthordercsequation} reduces to
\begin{align}
\sum_{\substack{(i,j) \in \mathbb{N}^2 \\ i + j = n }} \begin{pmatrix} 
b_{i,j,0} x^i y^j - \gamma_{i,j}  x^iy^j - i a_{i,j,0} x^{i-1}y^{j+1}   \\
 - \varepsilon_{i,j}  x^iy^j - i b_{i,j,0} x^{i-1}y^{j+1}    \\
\lambda_+ \alpha_{i,j,0} x^i y^j   - i \alpha_{i,j,0} x^{i-1}y^{j+1}    \\
 \lambda_- \beta_{i,j,0} x^i y^j  - i \beta_{i,j,0} x^{i-1}y^{j+1}   
\end{pmatrix} - \mathcal{P}^n \begin{pmatrix} x \\ y \\ 0 \end{pmatrix}. \label{eq: centerstableconstantz}
\end{align}
We can show that $\mathcal{P}^n(x,y,0)$ and $\mathcal{P}^n(x,y)$ from 
 \eqref{eq: nthorderconjugacy} 
coincide as we assumed that $a_{i,j,0}$, $b_{i,j,0}$, $\alpha_{i,j,0}$, and $\beta_{i,j,0}$ coincide 
with $a_{i,j}$, $b_{i,j}$, $\alpha_{i,j}$, and $\beta_{i,j}$ for $i+j<n$. Hence 
 \eqref{eq: nthorderconjugacy} 
and 
 \eqref{eq: centerstableconstantz} coincide, and thus demanding $a_{i,j,0} = a_{i,j}$, $b_{i,j,0} = b_{i,j}$, 
$\alpha_{i,j,0} = \alpha_{i,j}$, and $\beta_{i,j,0} = \beta_{i,j}$ ensures that 
 \eqref{eq: centerstableconstantz} 
vanishes.

When $k=1$, 
 \eqref{eq: nthordercsequation} reduces to
\begin{align}
\sum_{\substack{(i,j) \in \mathbb{N}^2 \\ i + j  = n -1 }} z \begin{pmatrix} 
b_{i,j,1} x^i y^j  - i a_{i,j,1} x^{i-1}y^{j+1} - \lambda_+ a_{i,j,1} x^i y^j   \\
 - i b_{i,j,1} x^{i-1}y^{j+1} -  \lambda_+ b_{i,j,1} x^i y^j   \\
 - \delta_{j,0} \zeta_{n} x^{n-1} - i \alpha_{i,j,1} x^{i-1}y^{j+1}   \\
  2 \lambda_-  \beta_{i,j,1} x^i y^j - i \beta_{i,j,1} x^{i-1}y^{j+1} 
\end{pmatrix} - \mathcal{P}^{n,1} \begin{pmatrix} x \\ y \\ z \end{pmatrix}.  \label{eq: centerstablelinearz}
\end{align}
Here we used in the fourth coordinate that $\lambda_+ = - \lambda_-$. The polynomial $\mathcal{P}^{n,1}$ consists of the terms linear in $z$ in $\mathcal{P}^n$. We can find unique $b_{i,j,1}$, $a_{i,j,1}$, $\zeta_n$, 
 $\alpha_{i,j,1}$, $\beta_{i,j,1}$ such that 
  \eqref{eq: centerstablelinearz} vanishes. To do so, we first find 
 $b_{n-1,0,1}$, which determines $b_{i,j,1}$ and $a_{i,j,1}$ for $i+j = n-1$. Secondly, we find $\zeta_n$ and 
 $\alpha_{i,j,1}$ independently of each other. Finally, we find $\beta_{n-1,0,1}$ which determines 
 $\beta_{i,j,1}$ for $i+j = n-1$. 

When $k \ge 2$, 
 \eqref{eq: nthordercsequation} reduces to
 \begin{align}
\sum_{\substack{(i,j) \in \mathbb{N}^2 \\ i + j  = n -k }} z^k \begin{pmatrix} 
b_{i,j,k} x^i y^j  - i a_{i,j,k} x^{i-1}y^{j+1}  - k \lambda_+ a_{i,j,k} x^i y^j   \\
- i b_{i,j,k} x^{i-1}y^{j+1}  - k \lambda_+ b_{i,j,k} x^i y^j   \\
(1-k) \lambda_+ \alpha_{i,j,k} x^i y^j   - i \alpha_{i,j,k} x^{i-1}y^{j+1}    \\
 (k+1) \lambda_- \beta_{i,j,k} x^i y^j  - i \beta_{i,j,k} x^{i-1}y^{j+1}    \\
\end{pmatrix} - \mathcal{P}^{n,k} \begin{pmatrix} x \\ y \\ z \end{pmatrix}. \label{eq: centerstablezk}
\end{align}
We again used that $\lambda_+  = \lambda_-$, and the polynomial $\mathcal{P}^{n,k}$ consists of  the terms which are of order $z^k$ in $\mathcal{P}^n$. Similar to what 
we did for 
 \eqref{eq: centerstablelinearz}, we can make 
  \eqref{eq: centerstablezk} vanish. The only difference
 is that we first find $\alpha_{n-k,0,k}$ instead of $\zeta_n$, and that this uniquely determines $\alpha_{i,j,k}$. 

Thus, we can recursively make 
 \eqref{eq: nthordercsequation} vanish, and we see that only $\alpha_{0,n-1,1}$ 
is not uniquely determined.
\end{proof}

\section{Numerical calculation of the center manifold}
\label{sec:numericalCenterManifold}

In Lemmas \ref{thm: centerconjugacyformalseries} to 
\ref{thm: centerstableconjugacyformalseries} we show that there exists formal series 
for the center manifold, a stable branch on the center 
manifold, and the center-stable manifold. To compute for example the center manifold, 
we have to calculate some homogeneous polynomials $\mathcal{P}^n(x,y)$. 
In this case, $\mathcal{P}^n(x,y)$ is the 
homogeneous polynomial of degree $n$ of the expression 
$\dynamics \circ \conjugacy^{<n} - D\conjugacy^{n<} \cdot \centerdynamics^{<n}$, where $\conjugacy^{<n}$ and $\centerdynamics^{<n}$ are the Taylor polynomial up to order $n-1$ of $\conjugacy$ and $\centerdynamics$ respectively. If we also replace $\dynamics$ by its Taylor Polynomials, we have to find the homogeneous polynomials of degree $n$ in the expressions $ \left(\conjugacy^{<n} \right)^\alpha$ and $\left(\centerdynamics^{<n} \right)^\alpha$ for $|\alpha|_1 \le n$.

In \cite{mamotreto}, radial derivates are used to find expressions for those homogeneous polynomials, 
provided that their constant term is non-zero. However, we constructed $\conjugacy$ and 
$\centerdynamics$ in such a way that their constant terms vanish. We can still use the described 
method to find an expression for $g^N$ in the case that $g(\bzero)$ does vanish. Here $g$ is a multivariate polynomial, which consists of the homogeneous polynomials $g_j$, and $N$ is a scalar. Let $\mathcal{R}$ 
denote the radial derivative of $g$, that is $\mathcal{R}(g)(\boldx) \bydef D g(\boldx) \cdot \boldx$.
Then we find
\begin{align*}
 \mathcal{R}(g^N) (\boldx) =  \sum_{i} \frac{\partial}{\partial x_i} g^N(\boldx) \cdot x_i = N g^{N-1}(\boldx) \mathcal{R}(g)(\boldx).
\end{align*}
Let $h = g^N$, then we find $g \mathcal{R}(h) = N h \mathcal{R}(g)$. Since $\mathcal{R}(P) = m P$ for 
homogeneous polynomials of degree $m$, we find that the $m$-th order term of 
$g \mathcal{R}(h) - N h \mathcal{R}(g)$ is given by
\begin{align*}
\sum_{j=0}^m \left( (m-j) - N j \right) g_j h_{m-j}.
\end{align*}
Since $g_0$ vanishes, the first non-zero term of $h$ is $h_N \bydef g_1^N$, and for higher order 
terms we have the recurrence relation
\begin{align*}
g_1 h_m = \frac{1}{N - m} \sum_{j=2}^{m+1 - N} g_j h_{m+1-j} ( m + 1 - j - N j).
\end{align*}
We use this to find homogeneous terms of order $N$ in compositions $g_1 \circ g_2$, by 
replacing $g_1$ with its Taylor polynomial of order $N$, and finding the homogeneous 
terms of $g_1^i(g_2)$ for $1 \le i \le N$ using the previous recurrence relation.

\bibliographystyle{unsrt}
\bibliography{papers}

\end{document}